\documentclass{llncs}
\usepackage{amsmath}
\usepackage{amssymb}
\usepackage{bm}  

\usepackage[a4paper, total={7in, 9in}]{geometry}
\usepackage{graphicx}
\usepackage{float}
\usepackage{array}
\usepackage{booktabs}
\usepackage{longtable}
\usepackage{tabularx}

\usepackage{xcolor}
\usepackage{soul}

\usepackage{hyperref}

\usepackage[font=small,skip=12pt]{caption}
\usepackage{subcaption}

\usepackage{algorithm}
\usepackage{algpseudocodex}  

\usepackage{listings}

\usepackage{tikz}
\usetikzlibrary{shapes.geometric, arrows}
\usepackage{academicons}
\usepackage{enumerate}
\usepackage{lineno}

\makeatletter
\def\algbackskip{\hskip-\ALG@thistlm}
\makeatother

\AtBeginDocument{%
 \providecommand\BibTeX{{%
   \normalfont B\kern-0.5em{\scshape i\kern-0.25em b}\kern-0.8em\TeX}}}

\DeclareUnicodeCharacter{2212}{\ensuremath{-}}

\begin{document}
\title{Algorithmic Design and Graph-Based Classification for Rectilinear-Shaped Modules in Floor Plans}
\author{Rohit Lohani, Ravi Suthar, Krishnendra Shekhawat}
\institute{Department of Mathematics, Birla Institute of Technology and Science, Pilani, Pilani Campus, Vidya Vihar, Pilani, Rajasthan 333031, India}

\maketitle
 \begin{abstract}

This paper introduces a graph-theoretic framework for constructing floor plans that accommodate non-rectangular modules, with a focus on $L$-shaped and $T$-shaped geometries. In contrast to conventional methods that primarily address outer boundaries, the proposed approach incorporates structural constraints that arise when realizing modules with these more complex shapes. The study investigates how tools from algorithmic graph theory can be employed to embed such modules within rectangular floor-plan representations derived from triangulated graphs.\\
The analysis shows that not every triangulated graph can support the specified module geometries. To characterize feasible instances, a shape-preservation constraint is formulated, preventing the geometry of a module from being altered through boundary deformation. Such changes would either increase the combinatorial complexity of adjacent modules or disrupt the intended adjacency structure.\\
A linear-time algorithm is presented, based on a prioritized canonical ordering, that constructs $L$ and $T$-shaped modules within a floor plan. The method applies to triangulated graphs containing at least one internal $K_4$, or two internal $K_4$ subgraphs that satisfy certain existence conditions. The paper details the construction process, outlines the conditions required to realize the desired module shapes, and demonstrates how these modules can be produced within the final floor plan. The algorithm’s simplicity and direct implementability make it suitable for integration into practical layout-generation workflows. Future work includes extending the framework to additional module shapes and exploring broader classes of supporting graph structures.
\end{abstract}


\begin{figure}
\centering
    \includegraphics[width = 1.00 \textwidth]{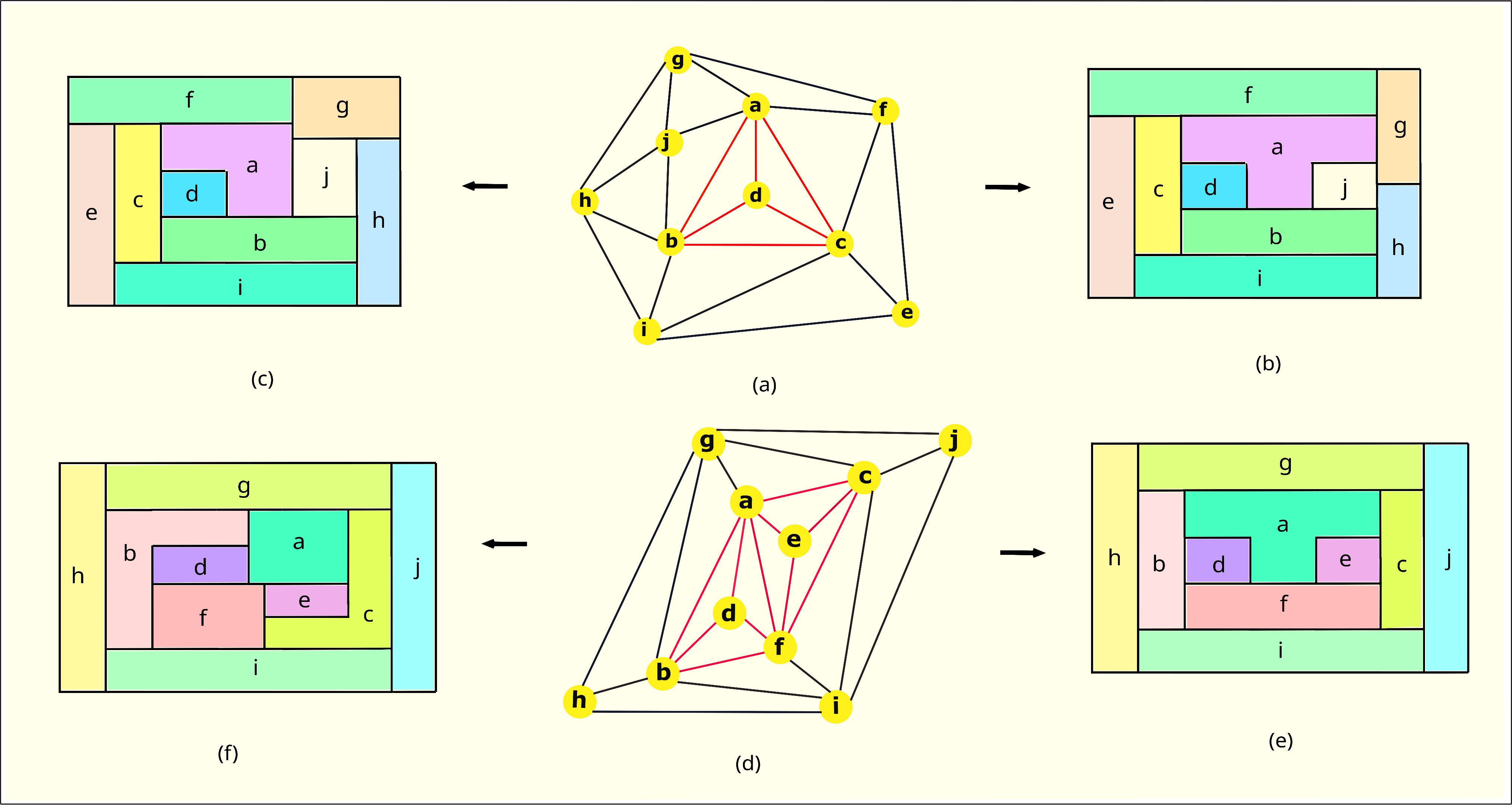}
    \caption{(a) Input graph $G_L$ containing $K_4$ (i.e., $K_L$). (b,c) Floor plans incorporating $L$ and $T$-shaped modules corresponding to $G_L$. (d) Input graph $G_T$ containing two $K_4$ sharing a common edge (i.e., $K_T$). (e,f) Floor plans containing a single $T$-shaped and two $L$-shaped modules corresponding to $G_T$.}
    \label{L-13}
\end{figure}

\section{Introduction}
\begin{figure}[H]
\centering
    \includegraphics[width = 1.05 \textwidth]{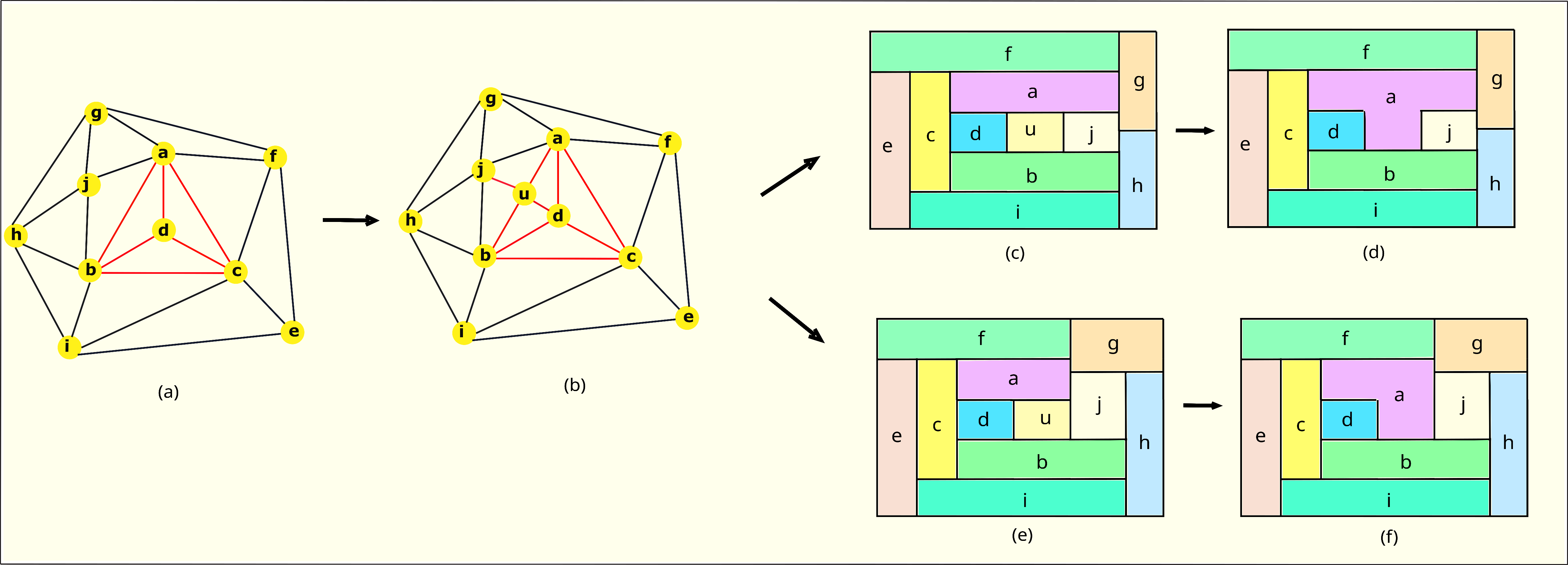}
    \caption{(a) Input graph $G_L$ containing $K_4$ (i.e., $K_L$). (b) Modified graph derived from $G_L$ after eliminating the complex triangle by adding an additional node $u$. (c,e) Obtained the rectangular floor plan for the modified graph. (d,f) Merging module $u$ to module $a$ to obtain a $L$-shaped module and a $T$-shaped module in the floor plans.}
    \label{L-14}
\end{figure}
\begin{figure}[H]
\centering
    \includegraphics[width = 1.05 \textwidth]{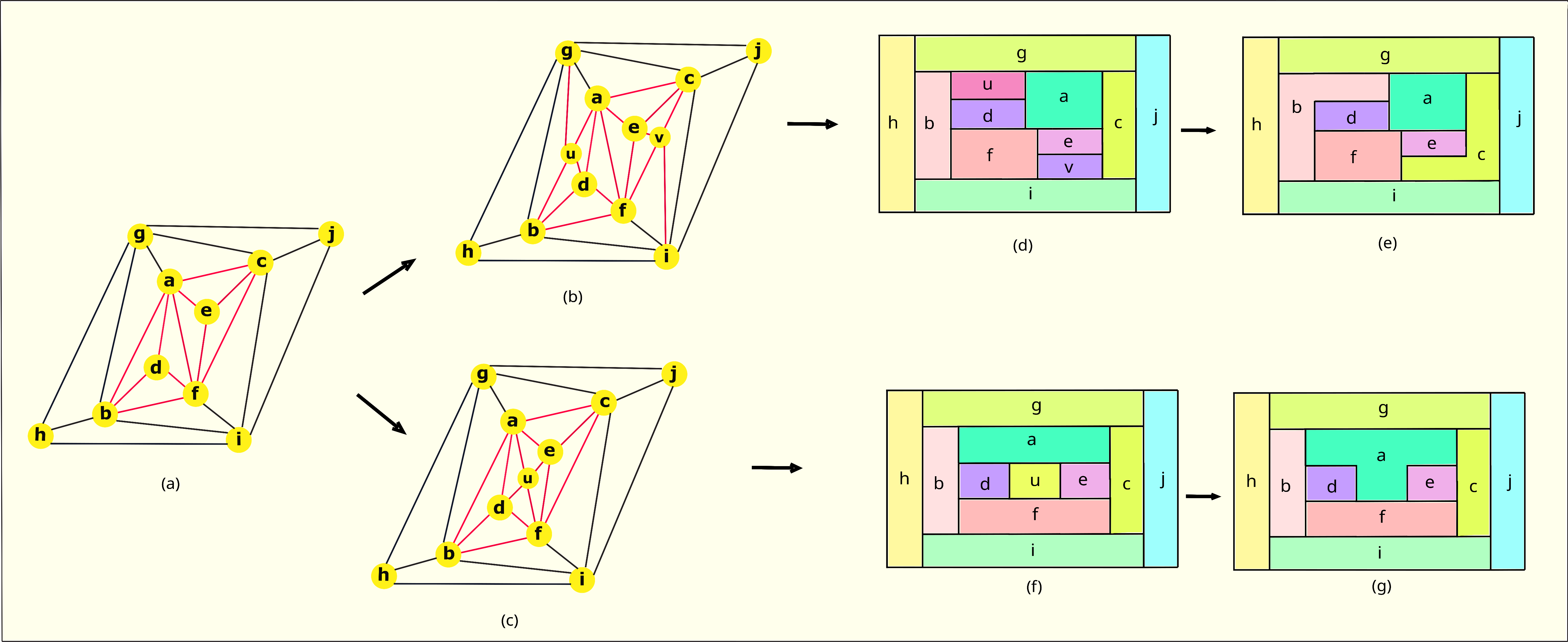}
    \caption{(a) Input graph $G_T$ containing two $K_4$ sharing common edge (i.e., $K_T$). (b, c) Two modified graphs derived from $G_T$ are shown, after eliminating a complex triangle by introducing either one node ($u$) or two extra nodes ($u$ and $v$). (d,f) Rectangular floor plans corresponding to the modified graphs. (e,g) The module $u$ is merged with module $a$ to obtain a $T$-shaped module, and the modules $u$ and $v$ are merged with module $b$ and $c$, respectively, to obtain two $L$-shaped modules ($b$ and $c$) in the floor plans.}
    \label{L-15}
\end{figure}
A floor plan refers to the division of a polygonal space into distinct rooms or modules using straight-line segments. Creating such layouts/floor plans is akin to assembling a complex puzzle, where each module represents a unique piece, and the spatial relationships among them add layers of complexity. Floor plan design has long been a key application of graph theory, particularly in areas like VLSI chip layout, architectural planning, and various domains within computer science. In architectural contexts, floor plans are essential for defining the spatial organization of a construction site. While much of the earlier research has focused on layouts composed of purely rectangular modules, there is increasing interest in exploring floor plans that incorporate non-rectangular shapes. These are referred to as orthogonal floor plans. Although multiple techniques exist for generating such layouts, developing algorithms that consistently produce specific non-rectangular modules remains a mathematically intricate and demanding task.\\
A triangulated graph \cite{he1995efficient} that lacks complex triangles having no more than four corner-implying paths can be represented using a rectangular floor plan. However, if a graph contains a complex triangle, then any floor plan derived from it will include at least one non-rectangular module \cite{sun1993floorplanning}. Therefore, to create a $L$-shaped module within a floor plan, the given graph must consists an internal complex triangle (see section \ref{5.2}), referred to as $K_L$ (as illustrated in Figure \ref{L-13}a, where $\triangle abc$ serves as the complex triangle). Similarly, constructing a $T$-shaped module requires the input graph to contain a specific internal configuration, namely, two complex triangles sharing a common edge (see section \ref{5.5}), denoted as $K_T$ (also illustrated in Figure \ref{L-13}b, where subgraph induced by vertices $a,b,c,d,e,f$ represents $K_T$).\\
Secondly, a notable limitation of the existing approach presented in \cite{shekhawat2023automated} is the lack of control over the specific shape of the module generated, i.e., breaking a complex triangle does not necessarily lead to the formation of a user defined desired $L$-shaped or $T$-shaped module (see Figures \ref{L-14}d, \ref{L-14}f, \ref{L-15}e, \ref{L-15}g). In contrast, the algorithms proposed in this work address this issue by ensuring the targeted creation of either a $L$ or $T$ module within the floor plan with minimal requirements in input graphs (i.e., either the presence of $K_L$ or $K_T$), see Figures \ref{L-14}f, \ref{L-15}g. This is achieved through the use of canonical ordering, a well-established method for ordering vertices in 4-connected triangulated graphs.\\
To generate a $L$-shaped module within a floor plan, the process begins by transforming the input graph into a 4-connected triangulated structure through the addition of auxiliary vertices and edges. Following this, canonical ordering is applied specifically prioritizing the vertices of the modified $K_L$ subgraph (refer to Section \ref{Preliminaries}) based on a defined category-wise priority scheme (Categories A–F, detailed in Section \ref{Category}), examining each category sequentially. A rectangular floor plan is then constructed using this priority-based canonical ordering, after which the auxiliary modules are integrated to produce the final floor plan containing the $L$-shaped module. Similarly, for constructing a $T$-shaped module, the input graph undergoes the same transformation into a 4-connected triangulated graph. Canonical ordering is applied, a rectangular floor plan is created, and then the added modules are merged to form the resulting floor plan featuring a $T$-shaped module.\\
Therefore, we introduce two linear-time algorithms, each with a complexity of $O(n)$ (where $n$ denotes the number of vertices), namely Algorithm \ref{L-shaped} and Algorithm \ref{TLabel}. These algorithms are designed to generate a $L$-shaped or $T$-shaped module within a floor plan corresponding to a given input graph that includes a minimum of one interior subgraph, either a $K_L$ or $K_T$.\\
The structure of this paper is organized as follows: Section \ref{Preliminaries} introduces the fundamental definitions and notations that serve as the foundation for our study. In Section \ref{Literature}, we present a detailed review of relevant literature, followed by an analysis of existing research gaps in Section \ref{Gaps}. Section \ref{Methodology}: (Methodology) outlines the essential conditions required in the input graphs for generating the desired modules, $L$ and $T$, along with their respective construction procedures through proposed algorithms. In Sections \ref{correctness} and \ref{Time}, we provide a thorough analysis of each algorithm's correctness and computational complexity. Finally, Section \ref{conclusion} concludes with a discussion on the study's limitations and potential directions for future work.

\section{Terminology}
\label{Preliminaries}

This section defines the key terms and notations employed in this paper to ensure clarity and consistency in the subsequent discussion of graph-based and floor plan-related concepts.

A graph \cite{kozminski1985rectangular} $G$ consists of two components: a vertex set $V(G)$ and an edge set $E(G)$. The 
vertex set is a finite, nonempty collection of nodes. In contrast, the edge set contains subsets $V_i \subset P(V(G))$ with a cardinality of exactly 2, representing the connections between vertices. A graph is categorized as a planar graph if it can be drawn on a 2-dimensional plane without edge intersections. When a planar graph is drawn this way, it is referred to as a plane graph \cite{kozminski1985rectangular}, and this representation divides the two-dimensional plane into regions known as faces. The unbounded region is called the external face, whereas all other regions are referred to as internal faces.

\begin{definition}
    \textit{k-connected graph}: A graph with a path between every pair of vertices is called a connected graph. This concept can be further generalized to k-connected graphs. Specifically, a graph with an order more than k is said to be k-connected if it remains connected after removing fewer than k vertices (see Figure \ref{connected}).
         \begin{figure}
   \centering
    \includegraphics[width=0.7\textwidth]{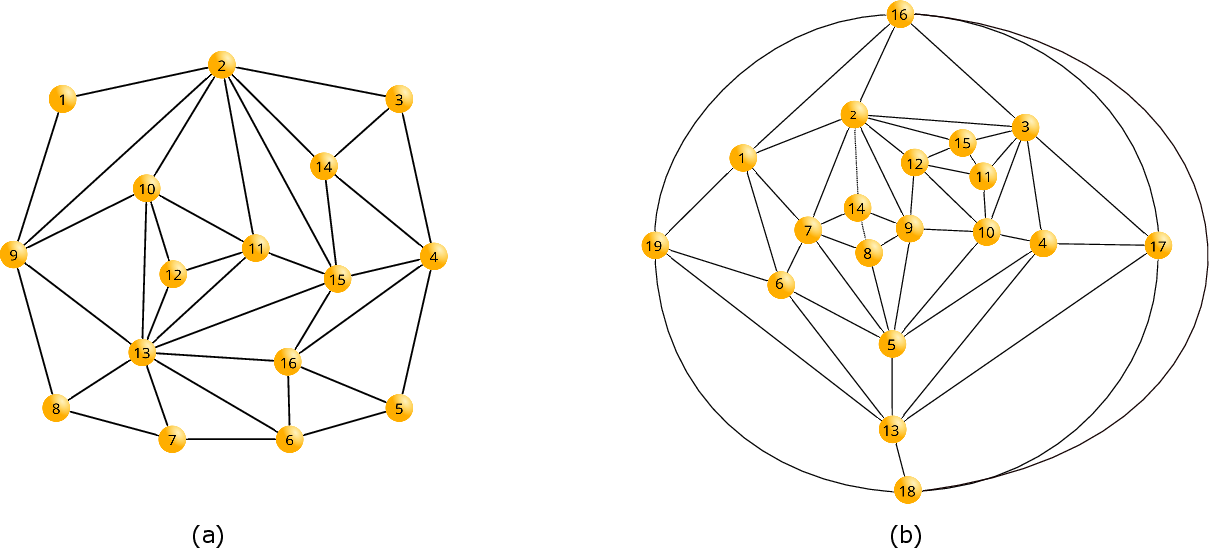}
    \caption{(a) A 3-connected graph. (b) A 4-connected graph.}
   \label{connected}
 \end{figure}
    
\end{definition}
\begin{definition}
\label{PTG}
    \textit{Plane triangulated graph [PTG] \cite{he1995efficient}}: A bi-connected graph is termed PTG if each of its faces, except the external one, is triangular (see Figure \ref{e1}a).
\end{definition}

\begin{definition}
\label{CT}
    \textit{Complex Triangle [CT]}: A CT is a cycle of length three, which contains at least one vertex positioned inside it  (see Figure \ref{e1}a).
      \begin{figure}
   \centering
    \includegraphics[width=0.5\textwidth]{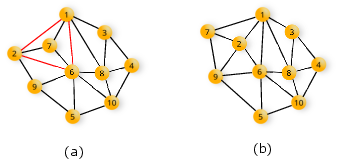}
    \caption{ (a) A PTG with complex triangle (1, 6, 2). (b) PTPG.}
   \label{e1}
 \end{figure}
\end{definition}

\noindent A plane triangulated graph (PTG, see Definition \ref{PTG}) is termed a Properly Triangulated Plane Graph (PTPG) if it contains no complex triangles (see Figure \ref{e1}b).

\begin{definition}
    Canonical Ordering \cite{he1995efficient}: Let $G$ be a 4-connected plane triangulated graph (PTG) that contains three distinct exterior vertices denoted as $N, S,$ and $W$. A canonical ordering of $G$ is a vertex ordering $(v_n,v_{n−1},...,v_1)$ satisfying the following conditions:
    \begin{itemize}
        \item[(i)] Vertices $v_1, v_2,$ and $v_n$ must correspond to $W, S,$ and $N$ respectively. For $4 \leq j \leq n$, the subgraph $G_{j−1}$ induced by $\{v_1,v_2,...,v_{j−1}\}$ is biconnected, with its exterior face forming a cycle $C_{j−1}$ containing the edge $(W, S)$.
        \item[(ii)] Vertex $v_j$ lies strictly outside $G_{j−1},$ $(i.e., v_j \notin V(G_{j−1})$. Furthermore, all neighbors of $v_j$ within $G_{j-1}$ must be arranged consecutively along the path $C_{j-1} \setminus (W, S)$, with at least two such neighbours.
        \item[(iii)] For $j \leq n-2$, it is necessary that $v_j$ has at least two neighbors in the graph $G \setminus G_{j-1}$, ensuring progressive connectivity during the ordering process. (See Figure \ref{CL1} where for input graph $G$, a canonical ordered graph is generated).
    \end{itemize}
   
      \begin{figure}
   \centering
\includegraphics[width=0.9\textwidth]{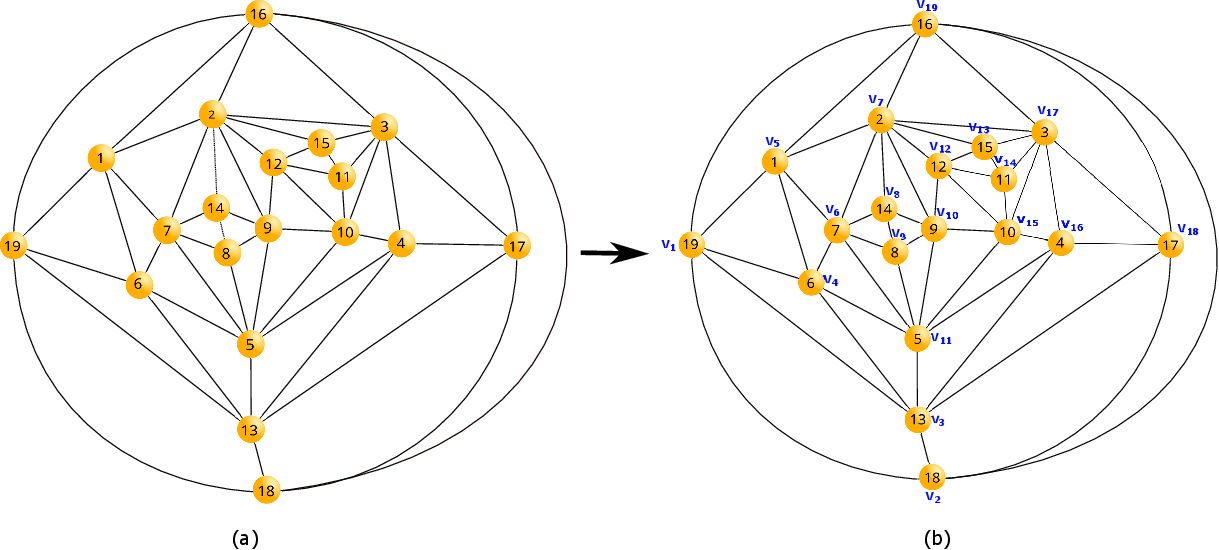}
    \caption{(a) A 4-connected plane graph $G$. (b) Canonical ordered graph of $G$.}
   \label{CL1}
 \end{figure}
\end{definition}

\begin{definition}
    Regular Edge Labeling [REL] \cite{kant1997regular}: A regular edge labeling for a bi-connected PTPG $G$ with four outer vertices N, W, S, E (ordered counterclockwise), constitutes a partition and orientation of its interior edges into two disjoint subsets, $T_1$ and $T_2$, satisfying the following conditions:
    \begin{itemize}
        \item[(I)] For every interior vertex $v$, the incident edges are arranged counterclockwise around $v$ in the sequence:
        \begin{itemize}
            \item Directed toward $v$: incoming $T_1$ edges.
            \item Directed away from $v$: outgoing $T_2$ edges
            \item Directed away from $v$: outgoing $T_1$ edges
            \item Directed toward $v$: incoming $T_2$ edges
        \end{itemize}

        \item[(ii)] Edges incident to vertex N are included in $T_1$ and are directed toward N. Conversely, edges incident to W are part of $T_2$ and are directed away from W. Edges incident to S are contained in $T_1$ and are directed away from S, while edges incident to E lie in the set $T_2$ and are directed toward E.
    \end{itemize}
\end{definition}

\noindent See Figure \ref{Def-REL} (a-b) where for input graph $G$, the regular edge labeling is generated. For every regular edge labeling (REL) of a PTPG $G$, there exists a rectangular floor plan $\mathcal{F}$ (see Figure \ref{Def-REL} (b-c)). In this context, each vertex in $G$ corresponds to a rectangular module in $\mathcal{F}$. Moreover, the directed edges in subset $T_1 $ represent vertical wall sharing between modules, which are aligned along the $y-axis$, while the directed edges in subset $T_2$ signify horizontal wall sharing between modules, aligned along the $x-axis$.
   
\begin{definition}
    Floor plan $(\mathcal{F})$ [\cite{rinsma1988existence}]: A Floor plan decomposes a polygon into smaller component polygons via straight-line segments. The outer polygon is called boundary of the floor plan, while the smaller component polygons are termed modules. Two modules are adjacent if they share a wall segment; mere point contact (four joints) does not constitute adjacency. A special class of floor plans is the rectangular floor plan (RFP), in which the boundary and all modules are rectangular (see Figure \ref{Def-REL}c). Another generalization is the orthogonal floor plan (OFP). In an OFP, the boundary is rectangular, but unlike RFPs, modules in an OFP may be rectilinear shapes, such as L-shaped or T-shaped polygons, provided all edges remain axis-aligned (see Figure \ref{L-3}b).
 \begin{figure}[H]
   \centering
\includegraphics[width=0.70\textwidth]{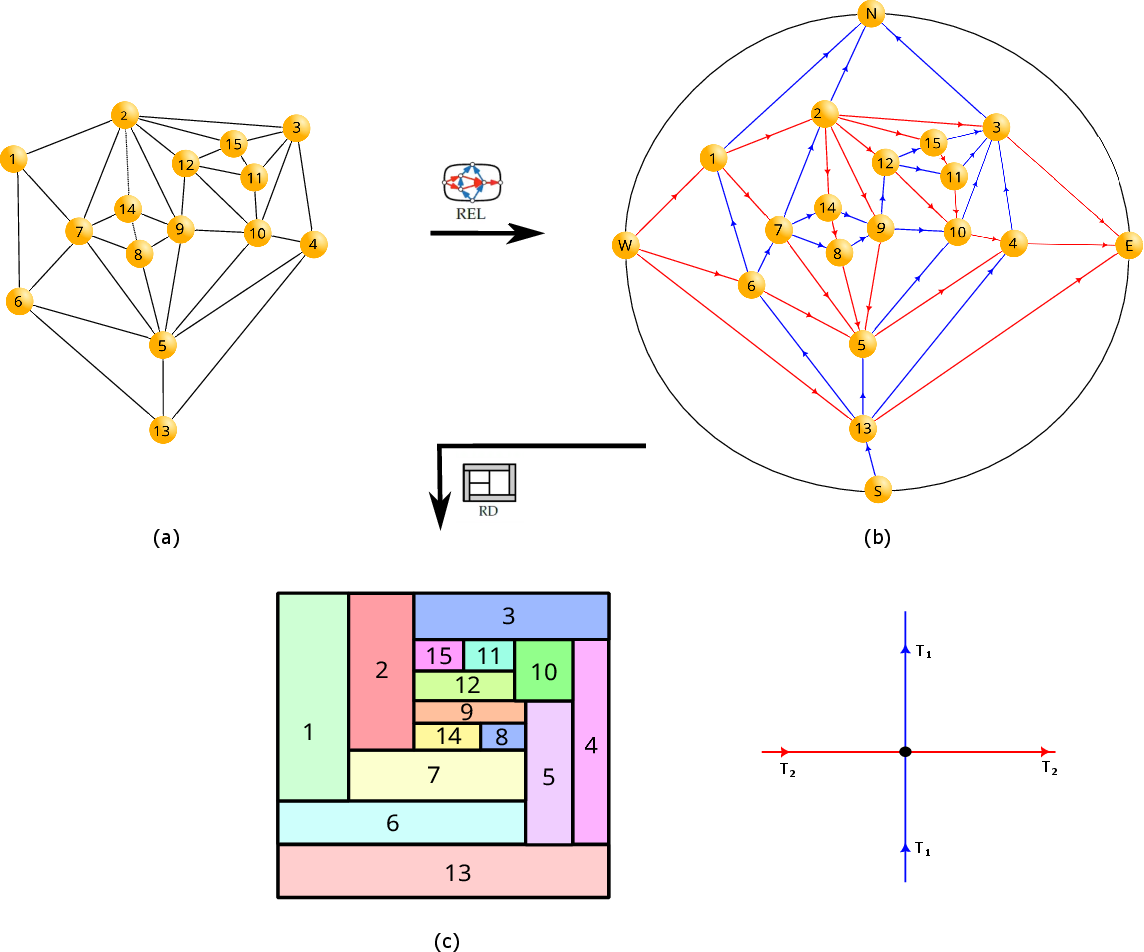}
    \caption{(a) A PTPG $G$. (b) Regular Edge labeling of graph $G$. (c) Floor plan generated using the corresponding REL of graph $G$.}
   \label{Def-REL}
 \end{figure}
\end{definition}
\textbf{Notations}: \\

\begin{itemize}
    \item[1.] \textbf{$G_L$}: An input triangulated plane graph includes at least one interior complex triangle ($K_L$), with the exterior face having a length $\geq 3$ (see Figure \ref{L-14}a).
    \item[2.] \textbf{$G_T$}: An input triangulated plane graph includes at least one subgraph ($K_T$), with the exterior face having a length $\geq 3$ (see Figure \ref{L-15}a).
    \item[3.] \textbf{$n$}: The total number of nodes in the input graph $G_a$ (a can be $L$ or $T$).
    \item[4.] \textbf{$ST_a$} (a can be $L$ or $T$): A set containing a collection of complex triangles $s_i$ of $G_a$.
    \item[5.] $M(l,m,F_a)$: Merge module $l$ with module $m$ in floor plan $F_a$.
    \item[6.] $K_L$: Complex triangle $K_4$ (the vertices of $K_L$ are ordered as: first $a$, then $b$, followed by $c$, and concluding with $d$ in a counter-clockwise sequence in $G_L$ (see Figure \ref{L-14}a).
    \item[7.] $K_T$: Two complex triangle ($K_4$) sharing an edge (the vertices of $K_T$ are ordered as: first $a$, then $b$, followed by $c$, $d$, $e$ and concluding with $f$ in a counter-clockwise sequence in $G_T$ (see Figure \ref{L-15}a).
    \item[8.] $modified$ $K_L$: The exterior edge $(a, b)$ of the complex triangle $K_L$ is subdivided by inserting a new vertex $u$, with additional edges introduced to maintain triangulation, yielding the $modified$ $K_L$ (Figure \ref{L-14}b).
    \item[9.] $modified$ $K_T$: The interior edge $(a, f)$ of $K_T$ is subdivided by inserting a new vertex $u$, with additional edges introduced to maintain triangulation, yielding the $modified$ $K_T$ (Figure \ref{L-15}c).
    \item[10.] \textbf{$G^1_a$ (where $a$ is $L$ or $T$)}: A 4-connected graph formed by augmenting $G_a$ with additional vertices and edges.
    \item[11.] \textbf{$G_{a_r}^1$}: The vertex-induced subgraph of $G^1_a$ defined by the set ${v_1, v_2,..., v_r}$.
    \item[12.] \textbf{ $C^a_r$}: A cycle representing the outer face boundary of $G_{ar}^1$, composed of the edge $(v_1, v_2)$.
    \item[13.] \textbf{$G^2_a$}: A canonically ordered graph derived for generated graph $G^1_a$.
    \item[14.] \textbf{$F'_a$ (where $a$ is $L$ or $T$)}: A rectangular floor plan structurally represent the graph $G^2_a$.
    \item[15.] \textbf{$F_a$}: An orthogonal floor plan incorporating an $a$-shaped module, corresponding to the input graph $G_a$.
    \item[16.] $S_a$: A subset of edges needed to remove complex triangles in $G_a$, excluding the $K_a$ subgraph, i.e., S $=$ \{($a_1$, $b_1$),.....($a_i$, $b_i$)\}.
    \item[17.] $Enodes_{a}$ (where $a$ is $L$ or $T$): A set of extra vertices introduced into $G_a$ to enable the removal of complex triangle within $G_a$ (excluding the $K_a$ subgraph), i.e., $extra$ $nodes$ $=$ \{$u_1$, .., $u_i$\} whereas each $u_i$ is added in $G_a$ corresponding to each ($a_i$, $b_i$) of $S_a$.
    \item[18.] $PLabel(L)$: Assign canonical orders to the vertices in the set $L$ ($L$ $=$ \{$p_1$, $p_2$, $p_3$, $p_4$\})  in descending order, beginning with $p_1$, then $p_2$, followed by $p_3$, and concluding with $p_4$.
    
\end{itemize}

\section{Literature Survey}  
\label{Literature}
\begin{itemize}
    \item[1.] \textbf{Origins of Architectural Graph Theory:} In the late 1900s, Paul and Grason represented architectural floor plans through the use of graphs. 
\end{itemize}
\begin{center}

\label{table1}
\scriptsize
\noindent\setlength\tabcolsep{5pt}
\centering
\begin{longtable}{p{1.5cm} p{3cm} p{5cm} p{5cm}}

\toprule 

 \textbf{\textit{ Reference}}  &  \textbf{\textit{ Contributors}} & \textbf{\textit{ Approach}}  &  \textbf{\textit{ Insight}} 
 
 \\
 
\midrule

\cite{levin1964use} & Paul Levin $(1960s)$ & Graph-theoretic approach & This paper introduced dual graphs for spatial adjacency in architecture by representing rooms as vertices and adjacencies as edges. Levin’s framework abstracted floor plans into connectivity graphs. However theoretical framework lacked computational implementation.

\\

\cite{grason1971approach}& John Grason $(1970)$ & Graph-theoretic approach  &  Grason formalized automated floor plan synthesis using dual graphs. He implemented his findings with the experimental CAD tool GRAMPA. But it is limited to axis-aligned rectangular modules.
\\

\bottomrule
\end{longtable}

\end{center}

\begin{itemize}
    \item[2.] \textbf{Foundational Theory:} A notable advancement in this direction occurred in the 1980s when contributors developed methodologies for automatically generating rectangular floor plans derived from abstract adjacency graphs.
\end{itemize}
\begin{center}

\label{table1}
\scriptsize
\noindent\setlength\tabcolsep{5pt}
\centering
\begin{longtable}{p{1.5cm} p{3cm} p{5cm} p{5cm}}

\toprule 

 \textbf{\textit{ }}  &  \textbf{\textit{ }} & \textbf{\textit{ }}  &  \textbf{\textit{ }} 
 \\

\cite{kozminski1985rectangular} & Kozminski and Kinnen $(1984-1985)$ & Graph theoretic approach & They developed the necessary and sufficient conditions for rectangular duals, which include planar triangulation and exclusion of complex triangles. They produced a Quadratic-time $O(n^2)$ algorithm for verification of triangulation and generation of rectangular floor plans.

\\

 \cite{bhasker1988linear}& Bhasker and Sahni $(1988)$ & Graph-theoretic approach & The authors developed an algorithm that operates in linear time $(O(n))$ for constructing rectangular duals and checking the triangulation. However, the coordinates of the rectangular dual (floor plan) generated from the proposed algorithm are real and do not relate to the structure of the input plane graph.

\\

\cite{he1995efficient} &   Xin He $(1995)$   &  Graph-theoretic approach  &  The author discusses a parallel algorithm for constructing rectangular duals of plane triangular graphs in $O(log^2n)$ time with $O(n)$ processors on a CRCW PRAM

\\

 \cite{kant1997regular}& Goos Kant, Xin He $(2019)$ & Graph-theoretic approach & This paper presents two linear-time $O(n)$ algorithms, namely edge contraction and canonical ordering, developed for constructing regular edge labeling (REL) for 4-connected triangulated plane graphs. The algorithms ensure that the coordinates of the constructed rectangular dual are integers, addressing a notable limitation of \cite{bhasker1988linear}.

\\
\bottomrule
\end{longtable}

\end{center}

\begin{itemize}
    \item[3.] \textbf{Recent Developments on Rectangular Floor plans (RFPs):} Over the years, floor plan generation has progressively developed from optimization techniques and rule-based graph transformations to advanced machine learning and deep-learning approaches. The initial methods include reproducing floor plans using graph algorithms and mathematical optimization. Later, hybrid methods integrated evolutionary and greedy algorithms, followed by reinforcement learning for better constraint handling. Most recently, deep learning and graph neural networks have enabled the generation of realistic, constraint-compliant floor plan layouts, marking a shift toward data-driven, intelligent design systems.
\end{itemize}
\begin{center}

\label{table1}
\scriptsize
\noindent\setlength\tabcolsep{5pt}
\centering
\begin{longtable}{p{1.5cm} p{3cm} p{5cm} p{5cm}}

\toprule 

 \textbf{\textit{ }}  &  \textbf{\textit{ }} & \textbf{\textit{ }}  &  \textbf{\textit{ }} 
 \\

\cite{wang2018customization}& X. Wang \textit{et al.} $(2018)$ & Graph-transformations  & The authors introduced a graphical approach to design generation (GADG) of RFPs based on existing legacy floor plans using dual graphs of PTPGs. This approach employs a rectangular dual-finding technique to automatically reproduce a new set of floor plans, which may be further refined and customized.

\\

\cite{nisztuk2019hybrid} & M. Nisztuk, P. Myszkowski $(2019)$ & Greedy-based and Evolutionary approaches  & The authors present a hybrid framework combining Evolutionary and Greedy algorithms for automated floor plan generation that meets adjacency and dimensional constraints. The Evolutionary Algorithm optimizes room sequences, scaling, and axis transformations, while the Greedy method incrementally places rooms based on these parameters, ensuring adjacency and non-overlap.

\\

\cite{shi2020addressing}& Feng shi \textit{et al.} $(2020)$ & Graph-theoretic approach and reinforcement learning  & The Authors present the Monte-Carlo Tree Search approach, which is based on reinforcement learning algorithms. To use this approach, a decision tree is required where the layout and constraints of the rectangular floor plan are defined as the input and output for an optimised layout.

\\

\cite{upasani2020automated}& N. Upasani \textit{et al.} $(2020)$ & Graph theory and mathematical optimisation & This research introduces a computational method for generating dimensional rectangular floor plans while preserving adjacencies derived from existing rectangular layouts. The authors employ linear optimisation techniques on the vertical and horizontal flow networks to accommodate the user-defined dimensional constraints and to obtain a feasible solution with minimum area that satisfies the given adjacencies. However, no remarks concerning the optimality of the solution were discussed.

\\

\cite{hu2020graph2plan} & Ruizhen Hu \textit{et al.} $(2020)$ &  Machine learning approach and deep neural network & Authors present a learning framework to automate floor plan generation, incorporating user-defined constraints (boundary, number of rooms, and adjacencies) to produce layout graphs retrieved from the floor plans within the RPLAN training dataset. These layout graphs are then adjusted to the input boundary. The Graph2Plan model, based on a graph neural network (GNN), generates a corresponding floor plan.

\\

\cite{wang2023automated}  & Lufeng Wang \textit{et al.} $(2023)$ & Deep learning and graph theory techniques & The authors present a framework combining deep learning and graph algorithms for automated building layout generation. The algorithm was trained over the unique GeLayout annotated dataset. The system optimises layout selection by employing Euclidean distance, Dice coefficient, and a force-directed graph algorithm.

\\

\cite{liu2024intelligent} & J. Liu \textit{et al.} $(2024)$ & Graph theoretic approach and Deep learning techniques & The paper presents a framework that employs a Graph-Constrained Generative Adversarial Network (GC-GAN) specifically for generating Modular Housing and Residential Building (MHRB) floor plans. This GC-GAN includes knowledge graphs to guarantee that the generated floor plans are realistic. It also incorporates an image-to-vector conversion algorithm for compatibility with a flat-design standardisation library. A significant aspect is the automated advancement of BIM models that adhere to modularity standards for efficient formation.

\\

\bottomrule
\end{longtable}

\end{center}

\begin{itemize}
    \item[4.] \textbf{Transition to Non-Rectangular Modules (1990s–2025s):} The transition from RFPs (1993) to OFPs incorporates non-rectangular rooms through advancements in graph theory. Early research focused on pruning of complex triangles for RFPs and the verification of $L$-shaped modules. Following these methodologies, linear-time algorithms were developed that facilitated the inclusion of $T$-shaped, $I$-shaped modules (1999-2003), and optimizations including spanning trees (2003). The exploration of rectilinear polygons (2011) and hexagonal tiling (2012) further enhanced geometric flexibility. Recently, the evolution of methodologies has transitioned from obstruction removal to module merging and ultimately to topological manipulation, resulting in a reduction in time complexity from quadratic to linear while accommodating irregular contours and user-defined geometries.
\end{itemize}

\begin{center}

\label{table1}
\scriptsize
\noindent\setlength\tabcolsep{5pt}
\centering
\begin{longtable}{p{1.5cm} p{3cm} p{5cm} p{5cm}}

\toprule 

 \textbf{\textit{ }}  &  \textbf{\textit{ }} & \textbf{\textit{ }}  &  \textbf{\textit{ }} 
 \\

 \cite{tsukiyama1993algorithm}& S. Tsukiyama \textit{et al.} $(1993)$ & Graph theoretic approach & This paper demonstrated that certain planar triangulated graphs (PTGs) resist a standard rectangular floor plan (RFP) due to embedded complex triangles. To address this limitation, they introduced an algorithm that removes these obstructive substructures and reconstructs an RFP on the pruned graph. Their algorithm runs in quadratic $O(n^2)$ time.

\\

 \cite{sun1993floorplanning}& Y. Sun \textit{et al.} $(1993)$ & Graph-theoretic approach & This paper presents an algorithm for whether a given graph admits a $L$-shaped dual with the complexity of this determination being $O(n^{\frac{3}{2}})$. If a $L$-shaped module exists, it can construct the rectangular floor plan with a $L$-shaped module in quadratic $O(n^2)$ time.

\\

 \cite{he1999floor}& X. He $(1999)$ & Graph-theoretic approach & The paper presents a linear time algorithm for the construction of floor plans for PTG using only 1- and 2-rectangle modules. The findings demonstrate a clear advancement over previous research \cite{sun1993floorplanning} conducted by Yeap and Sarrafzadeh, which demonstrated that PTG could be represented using 1-, 2-, and 3-rectangle modules.

\\

 \cite{kurowski2003simple}& M. Kurowski $(2003)$ & Graph-theoretic approach & The paper presents an algorithm for computing a floor plan in linear $O(n)$ time. The theory employs modules formed by merging two rectangles: $T-$, $L-$, or $I$-shaped. The dimension of the generated floor plan is at most $n\times n- 1$.

\\

 \cite{liao2003compact}&   CC Liao \textit{et al.} $(2003)$   &  Graph-theoretic approach  &  The paper introduces the algorithm, which is based upon orderly spanning trees to extend canonical ordering to plane graphs that do not require triangulation. This approach bypasses the complicated rectangular-dual phase and facilitates the computation of an orderly pair in linear time.

\\

 \cite{alam2011linear}&   MJ Alam \textit{et al.} $(2011)$  &  Graph-theoretic approach  &  This paper presents a study on proportional contact representations that use rectilinear polygons without wasted areas. The authors introduced a novel algorithm that ensures 10-sided rectilinear polygons and operates in linear O(n) time. These results improve the previous work that claimed to generate 12-sided rectilinear polygons within time complexity $O(nlogn)$. Additionally, they proposed a linear-time algorithm for proportional contact representation of planar 3-trees with 8-sided rectilinear polygons and showed that this is optimal.

 \\
 
\cite{duncan2012optimal} & CA Duncan \textit{et al.} $(2012)$ & Graph theoretic approach & The authors present a study demonstrating that hexagons are necessary and sufficient for depicting all planar graphs that pentagons cannot represent. It is possible to construct a touching hexagon representation of graph G in linear time on an $O(n)\times O(n)$ grid with convex regions.

\\

\cite{shekhawat2017rectilinear} & K. Shekhawat \textit{et al.} $(2017)$ & Graph theoretic approach & The author proposes a graph theory-based framework for generating rectilinear floor plans within non-rectangular contours, satisfying room adjacency and size constraints. Building upon prior rectangular models, it supports complex layouts, such as hospitals and offices, through polygonal boundaries and user-defined adjacencies. Limitations include a lack of real-time adaptability and limited multi-story integration.

\\

\cite{shekhawat2023automated} & K. Shekhawat \textit{et al.} $(2023)$ & Graph theoretic approach & The author presents an innovative algorithm for the generation of rectilinear floor plans. This research advances prior investigations focusing on orthogonal floor plans without considering specific room shapes. The proposed research framework utilizes complex triangles to generate $L$-shaped, $T$-shaped, $C$-shaped $F$-shaped, stair-shaped, and plus-shaped (cross-shaped) rooms. However, the proposed algorithm does not possess the capability to generate a specific shape for a given PTPG

\\

\bottomrule
\end{longtable}

\end{center}

\section{A Comparative Review of Literature Gaps}\label{Gaps}
Early research on automated floor plan generation based on graph-theoretic approaches [\cite{kozminski1985rectangular},
 \cite{bhasker1988linear}, 
\cite{he1995efficient}, 
 \cite{kant1997regular}] has primarily been focused on creating rectangular floor plans. As the field progressed, algorithms [\cite{duncan2012optimal}, \cite{shekhawat2023automated}] have emerged to generate more complex non-rectangular modules, including $L$-shaped, $T$-shaped, $C$-shaped, $F$-shaped, stair-shaped, and plus-shaped (cross-shaped) rooms, by leveraging the classification of graphs with complex triangles. These frameworks mainly operate by first identifying complex triangles within the graph, subdividing their edges according to specific rules, and then merging the resulting modules to form non-rectangular modules. However, a notable limitation of these methodologies is their inability to guarantee the generation of a specific non-rectangular shape for a given input graph; for the same graph and subdivision process, different non-rectangular shapes may emerge unpredictably. As a result, these frameworks lack precise control over the final module shapes, making them unsuitable for applications where specific room geometries are required.

In contrast, our work emphasizes the systematic generation of $L$ and $T$-shaped modules through the integration of rectangular components. We present a refined classification of graphs and establish necessary conditions for the construction of $L$ and $T$ shapes by manipulating complex triangles in conjunction with priority-based canonical ordering. Our approach recognizes that both $L$ and $T$ shapes can arise from the same graph classification; however, achieving the desired outcome requires a specific order in the placement of modules associated with the complex triangle. By incorporating Regular Edge Labeling (REL) and prioritizing canonical ordering, we ensure that the generation process is both deterministic and efficient.

A notable advancement of our research over previous work [\cite{shekhawat2023automated}, \cite{shekhawat2017rectilinear}] is the substantial reduction in computational complexity. Our proposed method achieves linear time generation for $L$ and $T$-shaped rooms. This not only guarantees the existence of the desired module shape for a given graph but also makes the approach scalable and practical for real-world architectural design applications. Both $L$ and $T$ shapes are classified under the same foundational graphs, their distinction arises only when further classified using Regular Edge Labeling (REL) and priority canonical ordering. Thus, it is necessary to have a separate algorithmic approach for $L$ and $T$ shapes.

\section{Methodology} \label{Methodology}
This section outlines a linear-time algorithm designed to construct either a $L$-shaped or a $T$-shaped module within a floor plan, starting from a triangulated plane graph whose outer face has at least three edges. The algorithm relies on the existence of one or more interior complex triangles (i.e., subgraphs isomorphic to $K_4$) depending on the specific module to be generated. In the upcoming sections, we will describe the proposed algorithms for generating each module type (i.e., either $L$ or $T$ shaped) separately. It is important to emphasize that the generation of such modules is feasible only when the graph contains the necessary complex triangles to guide the construction.

\begin{figure}
   \centering
    \includegraphics[width=1.0\textwidth]{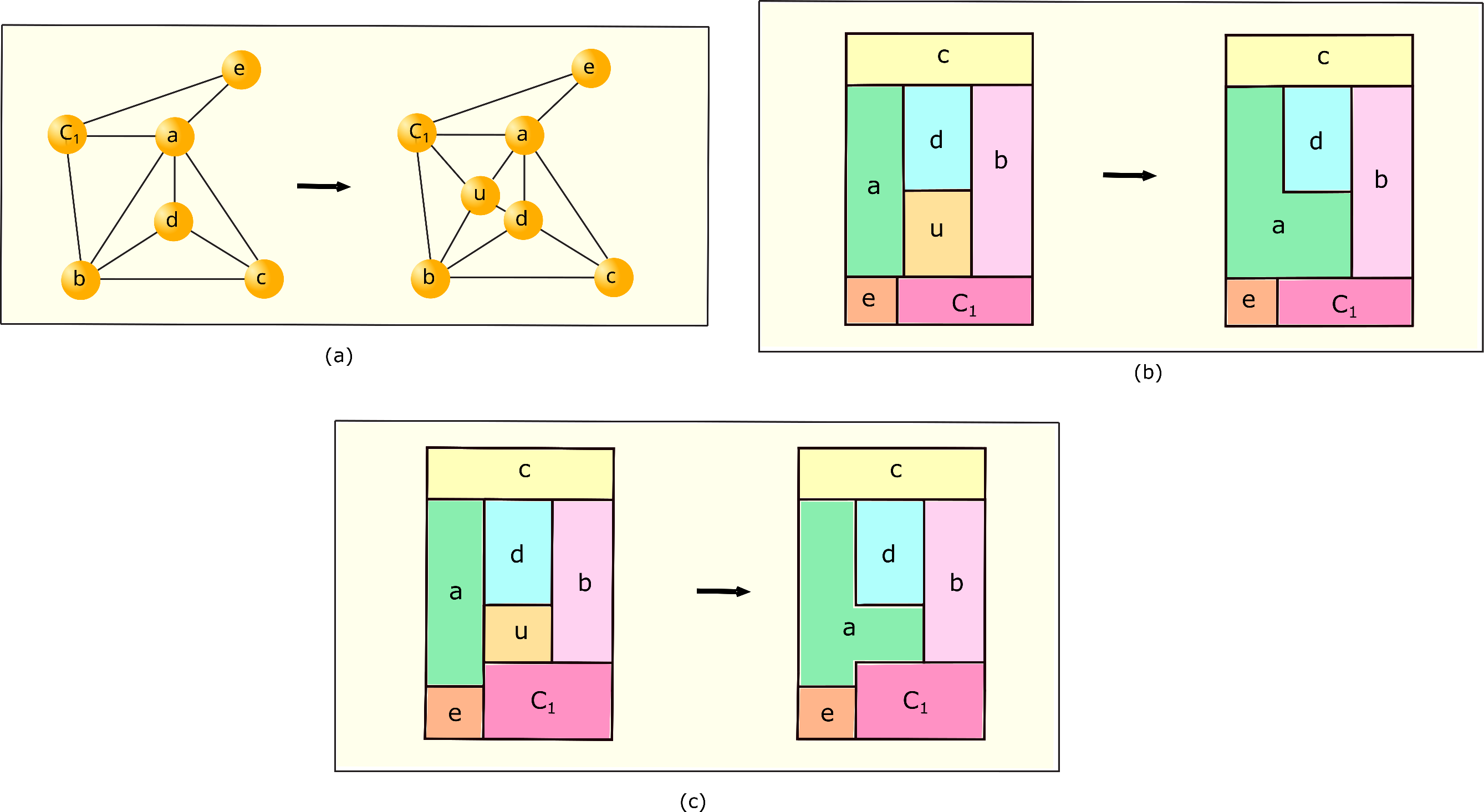}
    \caption{ (a) Inserting an additional vertex $u$ into the input graph $K_L$ for complex triangle removal. (b–c) A $L$-shaped module is generated in the resulting floor plan, while a trivial $T$-shaped module is created for the same $modified$ $K_L$ graph.}
   \label{L-2}
 \end{figure}
\begin{figure}
   \centering
    \includegraphics[width=0.7\textwidth]{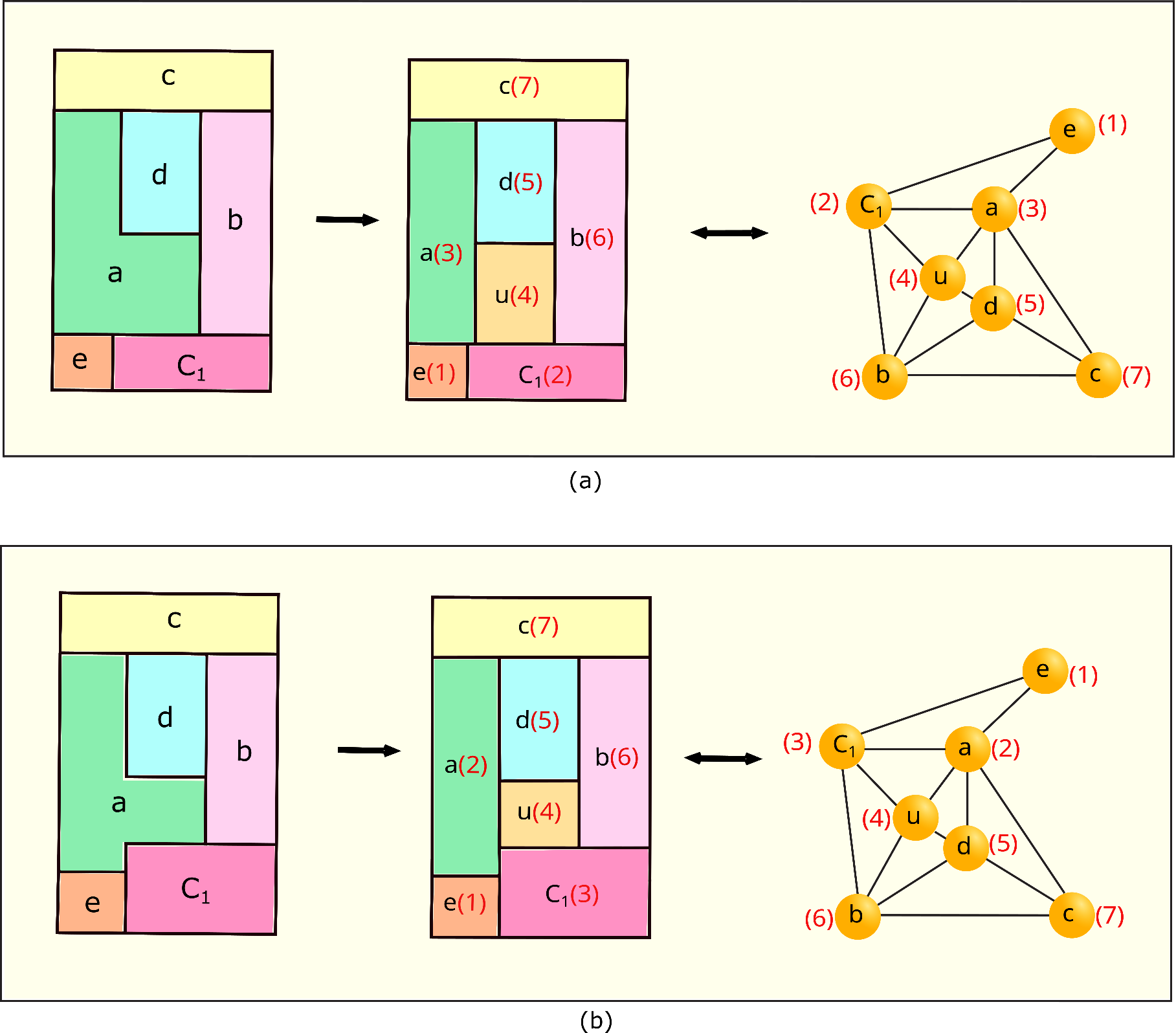}
    \caption{ (a-b) Canonical ordered graphs generated from floor plans, utilizing the algorithm in \cite{kant1997regular}, which include $L$-shaped and trivial $T$-shaped modules.}
   \label{L-4}
 \end{figure}
  \begin{figure}
   \centering
    \includegraphics[width=0.50\textwidth]{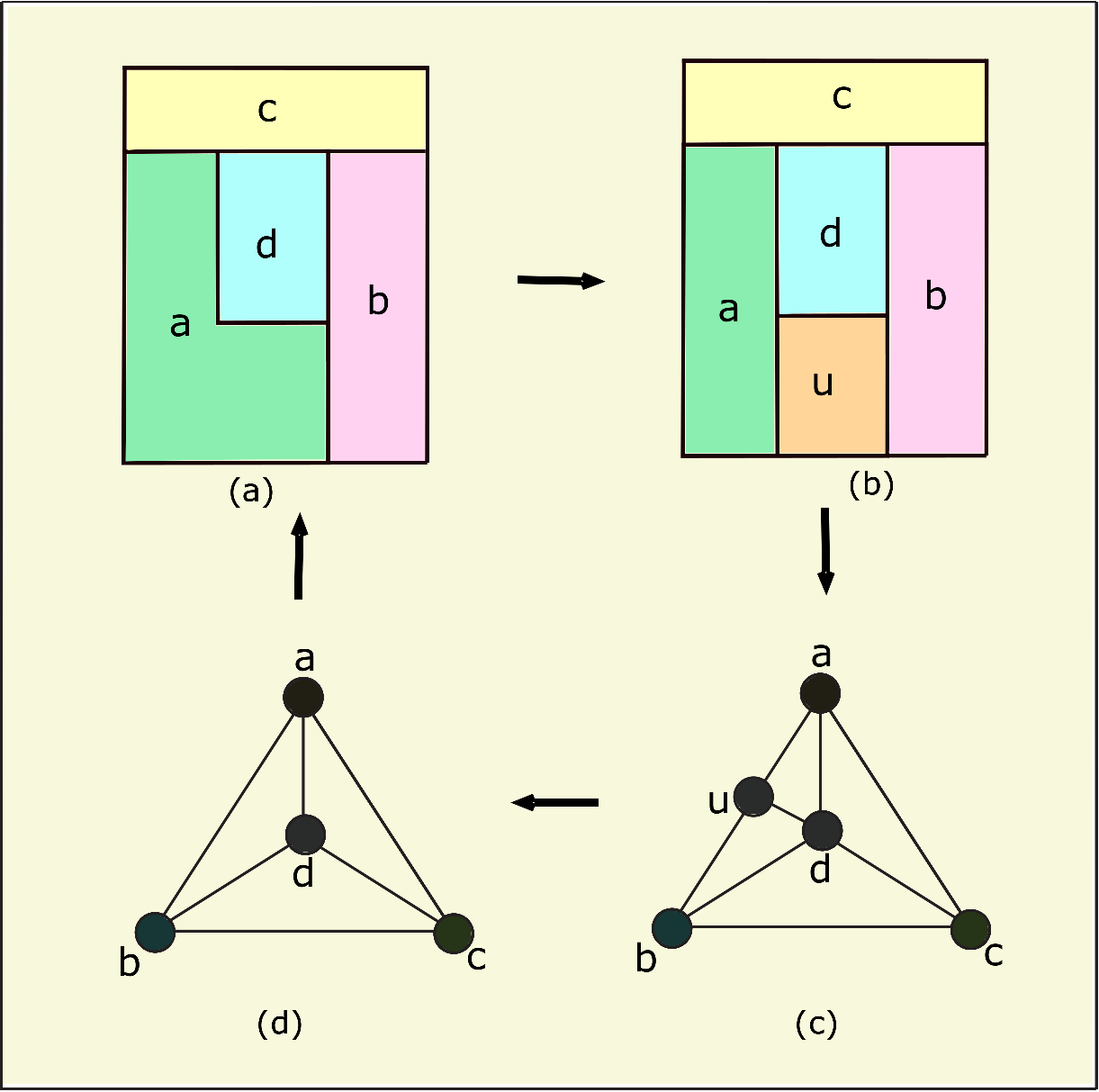}
    \caption{ (a-d) Requirement of an interior complex triangle $K_L$ for the generation of a $L$-shape module.}
   \label{L-3}
\end{figure}
\begin{figure}
    \centering
    \includegraphics[width=0.75\textwidth]{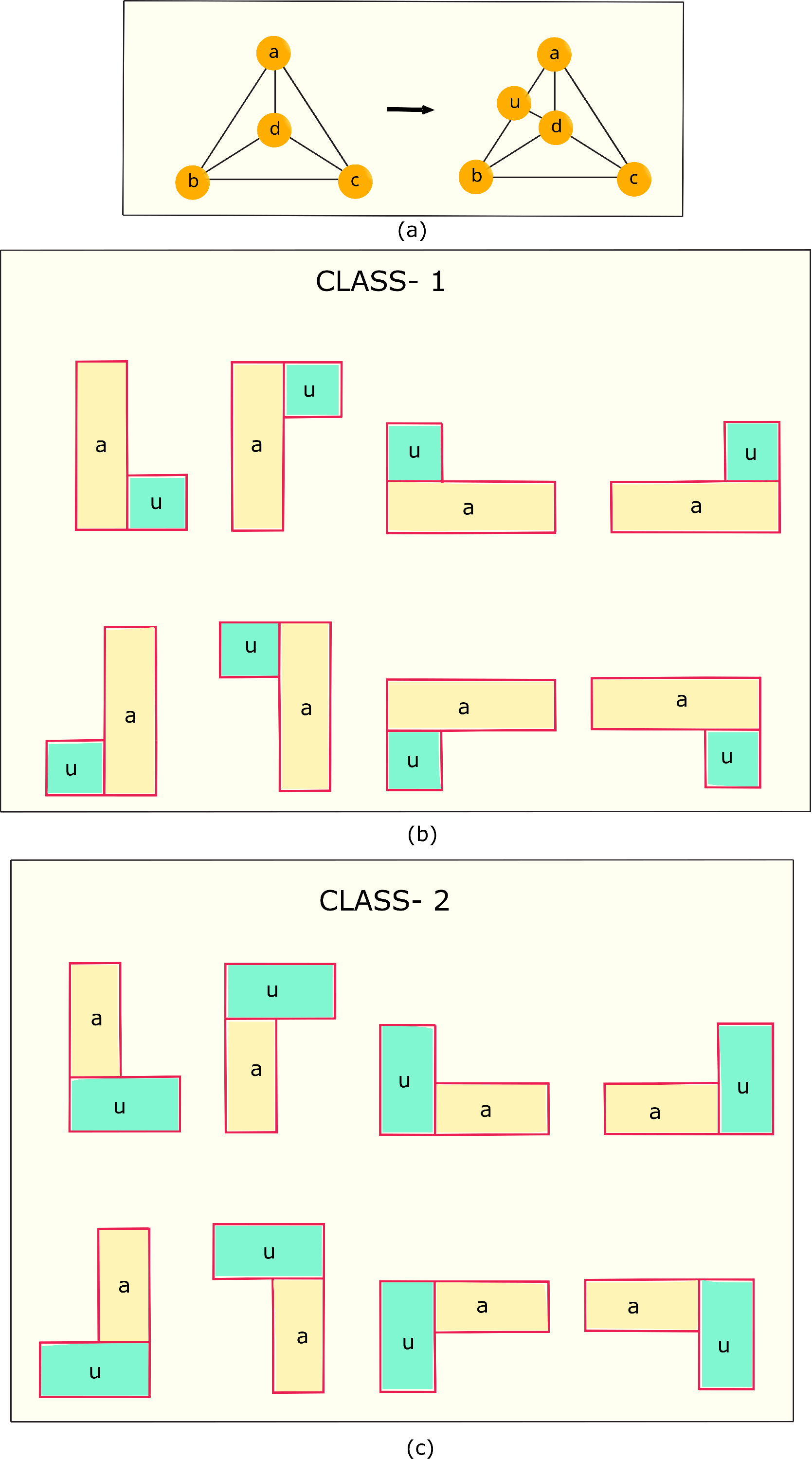}
    \caption{ (a-c) Several ways for merging module $u$ with module $a$.}
   \label{L-1}
 \end{figure}
   \begin{figure}
   \centering
    \includegraphics[width=0.70\textwidth]{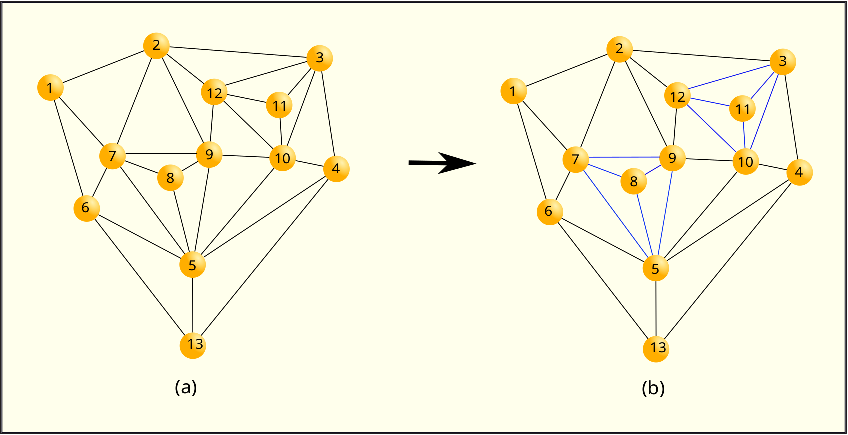}
    \caption{ (a-b) Identifying the complex triangle (i.e., $K_L$ and others)  in $G^1_L$. }
   \label{L-9}
 \end{figure}
    \begin{figure}
   \centering
    \includegraphics[width=0.90\textwidth]{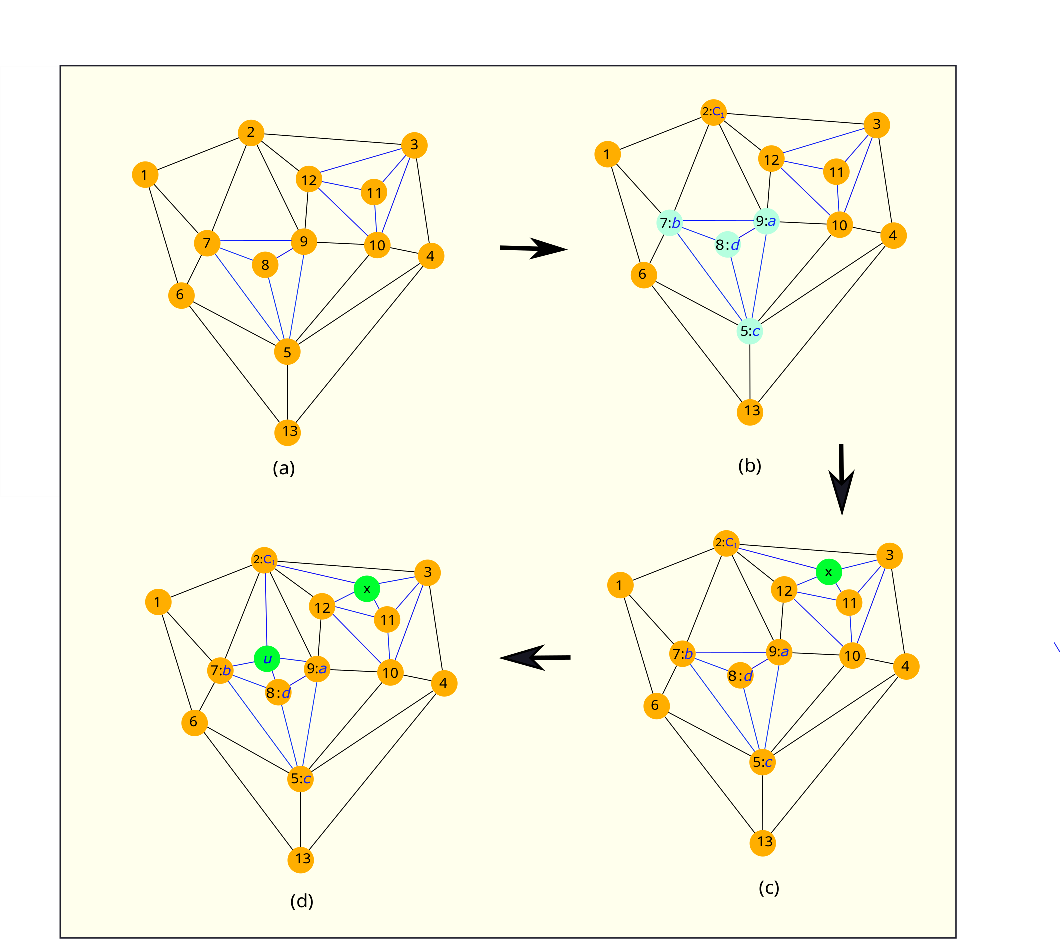}
    \caption{ (a-d) Label the complex triangle $K_L$ (i.e., $a$, $b$, $c$, and $d$ in counterclockwise order) and breaking the complex triangles in $G^1_L$.}
   \label{L-11}
 \end{figure}
   \begin{figure}
   \centering
    \includegraphics[width=1.00\textwidth]{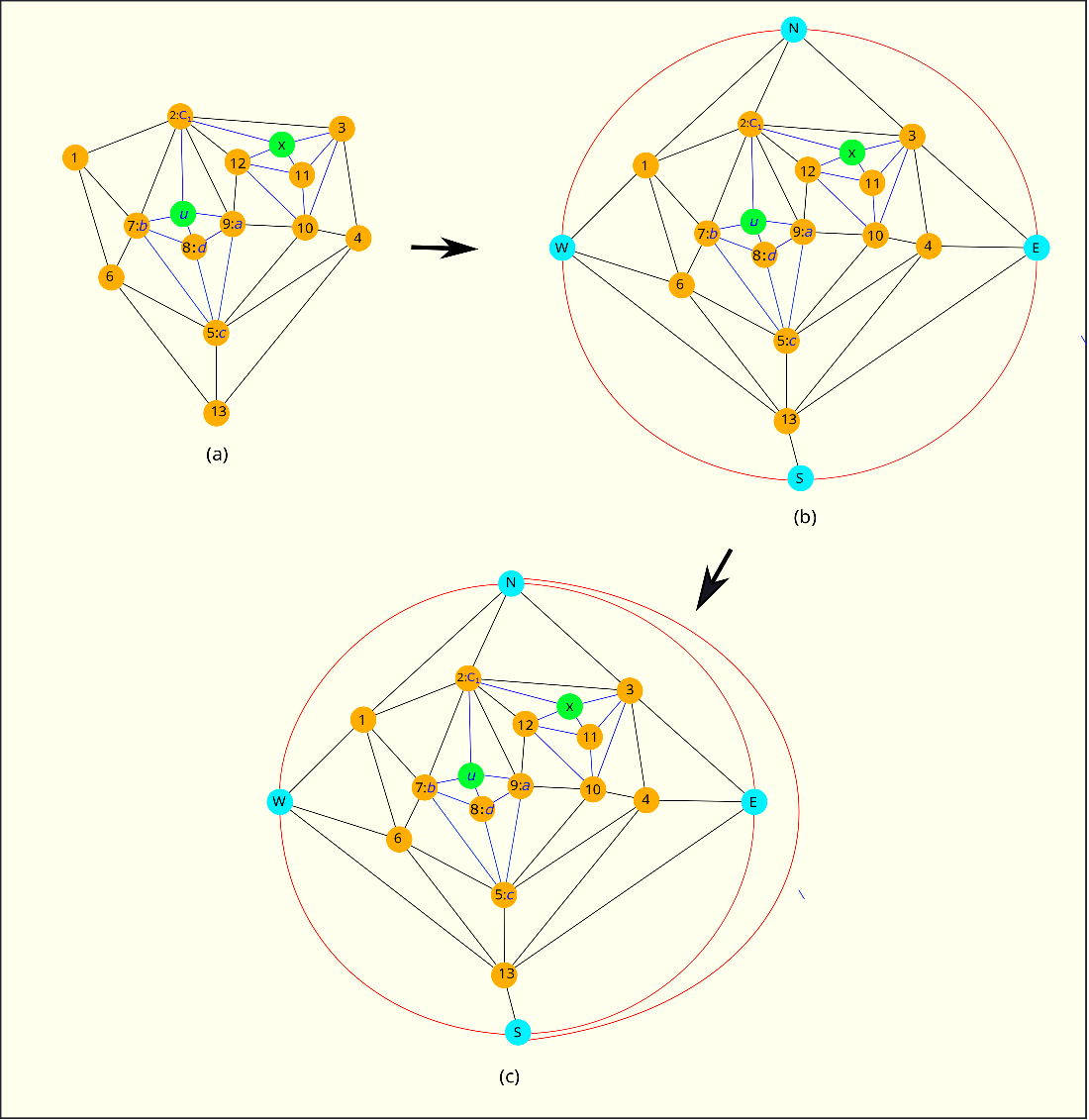}
    \caption{ (a-c) Applying Four-Completion and adding extra edge in $G^1_L$ to construct $4$ connected triangulated graph.}
   \label{L-12}
 \end{figure}

\subsection{An Overview of Our Proposed Work for $L$- Shaped Module Generation}
Various RFPs are generated by modifying the $K_4$ subgraph through the subdivision of one of its outer edges, using the method described in \cite{kant1997regular}. After merging the additional module $u$, which arises from the elimination of a complex triangle $K_4$, various floor plans featuring both $L$-shaped and trivial $T$-shaped modules are produced (see Figure \ref{L-2}). This demonstrates that simply breaking a single edge of a complex triangle does not necessarily result in a $L$-shaped module within a floor plan. This suggests that when a graph contains $K_4$ as a subgraph, an extra step or method is needed to ensure the creation of a specific $L$-shaped module in the final floor plan.\\
Additionally, we found that the canonical vertex ordering described in \cite{kant1997regular} differs between the graphs representing floor plans containing $L$-shaped modules and those with trivial $T$-shaped modules (see Figure \ref{L-4}). This suggests that to successfully create a $L$-shaped module within the floor plan, it is necessary to prioritize the canonical ordering of the modified $K_L$ subgraph during the ordering process of the input graph $G_L$.\\
Hence, we aim to develop an algorithm that assigns canonical orders to the vertices of the input graph $G_L$ and the modified $K_L$ subgraph following a predetermined priority sequence (see Figures \ref{L-5}, \ref{L-6}. \ref{L-7}: where 10 possible canonical ordering derived for the generation of $L$-shaped module), ensuring the resulting floor plan includes a $L$-shaped module. Accordingly, $Algorithm$ \ref{L-shaped} constructs a $L$-shaped module within the floor plan $F_L$ for any given PTG $G_L$ that contains at least one interior complex triangle $K_L$.
\subsection {Requirement of an Interior Complex Triangle $K_L$ within the Graph $G_L$} \label{5.2}
The structural role of complex triangles within a plane graph is a critical element in determining the characteristics of the subsequent floor plan. When analyzing a triangulated plane graphs that lack complex triangles and permit a maximum of four corner-implying paths, it follows that such a plane graph can be represented as a rectangular floor plan. Conversely, including complex triangles introduces geometric complexities that prevent the establishment of a solely rectangular floor plan. This phenomenon is attributable to the fact that complex triangles ensure the presence of non-rectangular modules within a rectangular floor plan \cite{sun1993floorplanning}.\\
To design a floor plan ($F_L$) that incorporates a module of $L$-shaped, it is essential to include a subgraph $K_L$ within $G^1_L$ (see Figure \ref{L-3}). Therefore, the presence of a subgraph $K_L$ in the input graph is necessary to successfully form the module of shape $L$ within the floor plan. The outcome highlights the fundamental requirements for the construction of a $L$-shaped module within a floor plan. Consequently, the process begins by identifying a region $K_L$ in the input graph $G_L$, with its vertices ordered $a$, $b$, $c$, and $d$ in counterclockwise order.\\
To integrate a module $L$ in the floor plan, there are two classes to merge the additional module $u$ with module $a$ according to the degree of the vertex $u$ (see Figure \ref{L-1}). These configurations involve $u$ sharing horizontal or vertical walls with $a$. For each class, there are eight ways to merge module $u$ with module $a$; in this study, we will only generate a $L$-shaped module in the floor plan with respect to Class-1.\\
Based on the types in Class-1, we derived ten categories of $L$-shaped modules that can be present in the final floor plan. For each of the ten categories (1–10), it is possible to construct a $L$-shaped module within the floor plan (see Figure \ref{L-5}, \ref{L-6}, \ref{L-7}). However, the proposed Algorithm \ref{L-shaped} must define the canonical ordering for each category based on priority. There are six unique cases in total, where some share the same canonical ordering but differ in their regular edge labeling (REL), as explained in a later section.\\
This study concentrates on building a $L$-shaped module within the floor plan, focusing on one class out of two possible cases (see Figure \ref{L-1}). The presence of such a $L$-shaped module for any of the ten categories (for Class-1) within a given graph is formally established in the correctness section. Furthermore, for each category, the module $u$ must have degree 4 to ensure its inclusion in the resulting floor plan (see Figure \ref{L-1}). This requirement indicates that the vertex $u$ in the modified subgraph $K_L$ of the graph $G_L$ must be placed as an interior vertex. Therefore, our focus remains on forming a $L$-shaped module derived from the ten defined categories, considering the interior configuration of $K_L$.

 \begin{figure}
   \centering
    \includegraphics[width=0.9\textwidth]{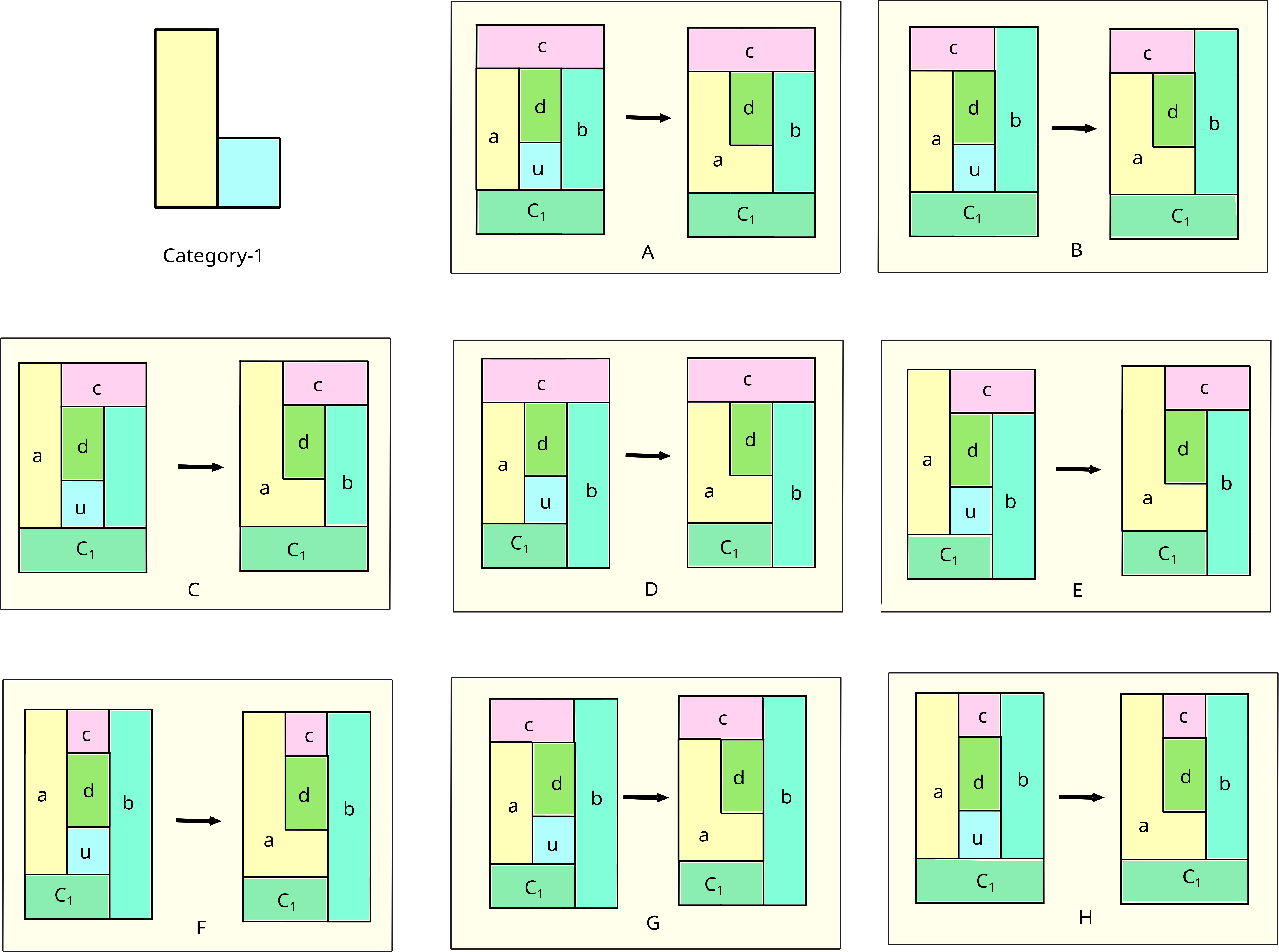}
    \caption{Eight possible $L$-shaped configuration types (A–E) associated with Category-1 (see Figure \ref{L-5}) in the floor plan.}
   \label{L-8}
 \end{figure}
 \begin{figure}
   \centering
    \includegraphics[width=0.80\textwidth]{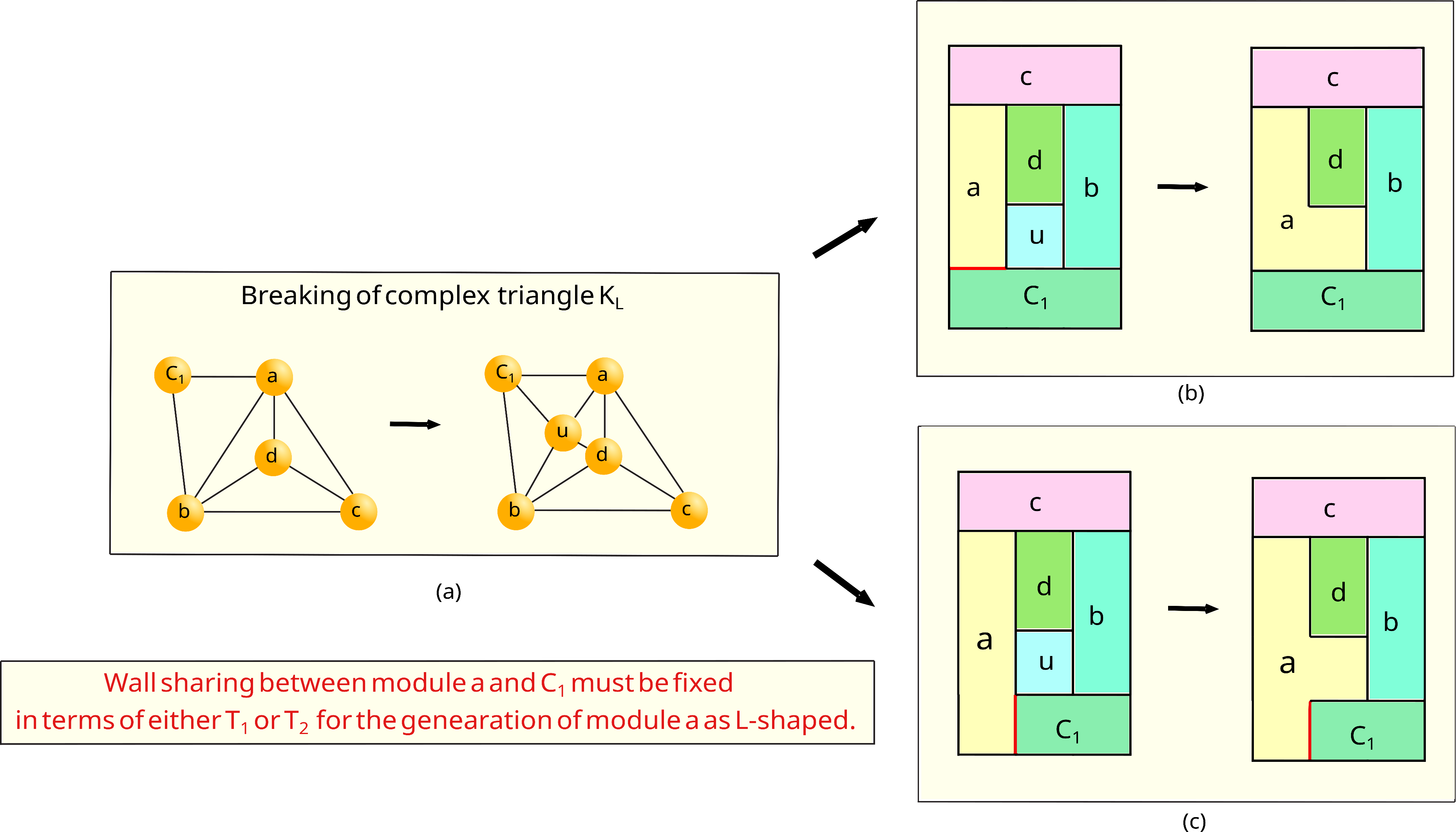}
    \caption{ (a) Modified graph $K_L$. (b-c) Sharing of a wall between module $a$ and $C_1$ may lead to the absence of a $L$-shape in the final floor plan.}
   \label{L-10}
 \end{figure}

\begin{figure}
   \centering
    \includegraphics[width=1.04\textwidth]{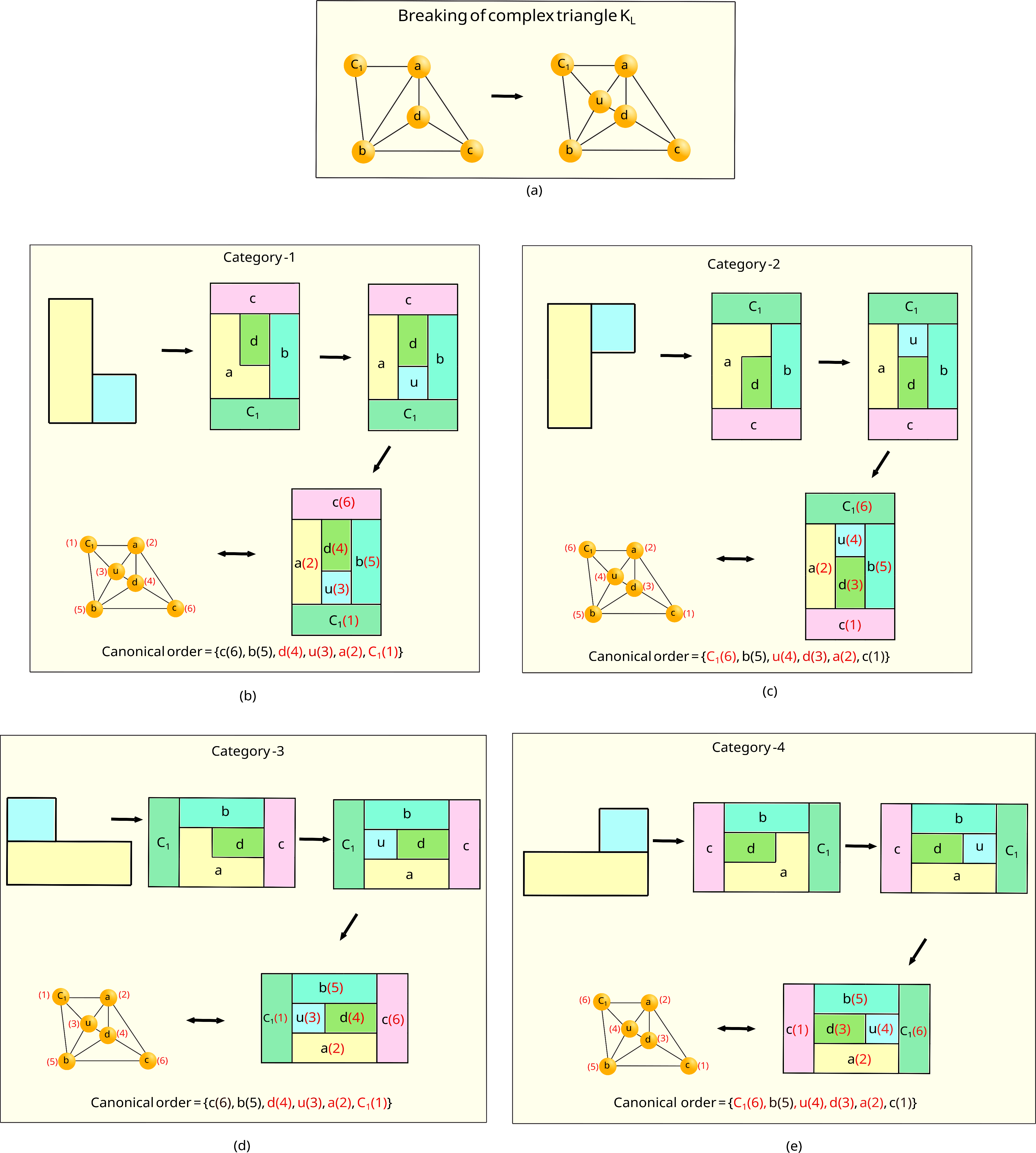}
    \caption{ (a-e) Canonical ordering associated with the $modified$ graph $K_L$ is shown for different categories (Category 1–4), each defined based on a $L$-shaped module relative to Class-1.}
   \label{L-5}
 \end{figure}
 \begin{figure}
   \centering
    \includegraphics[width=1.05\textwidth]{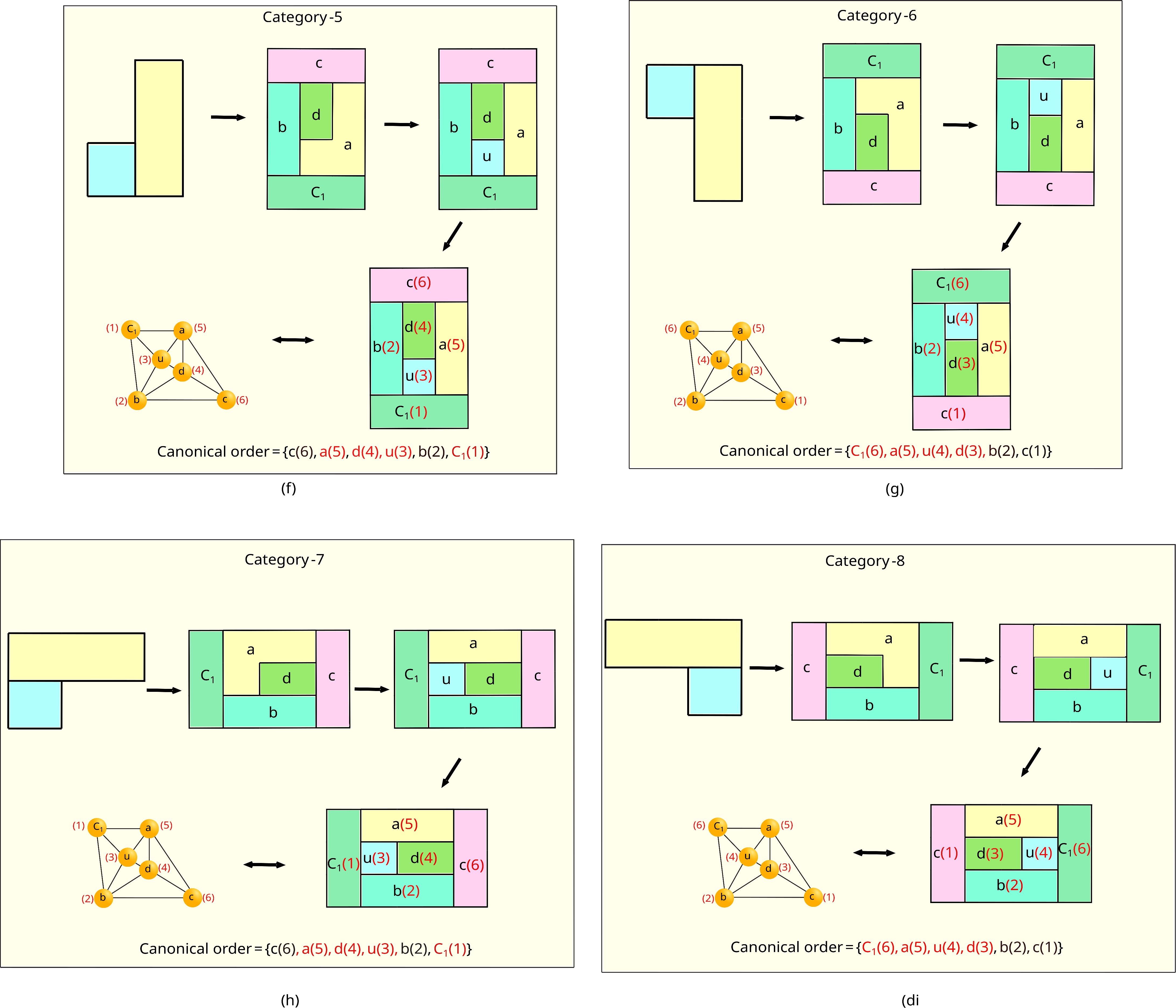}
    \caption{ (f-i) Canonical ordering associated with the $modified$ graph $K_L$ is shown for different categories (Category 5–8), each defined based on a $L$-shaped module relative to Class-1.}
   \label{L-6}
 \end{figure}
 \begin{figure}
   \centering
    \includegraphics[width=1.04\textwidth]{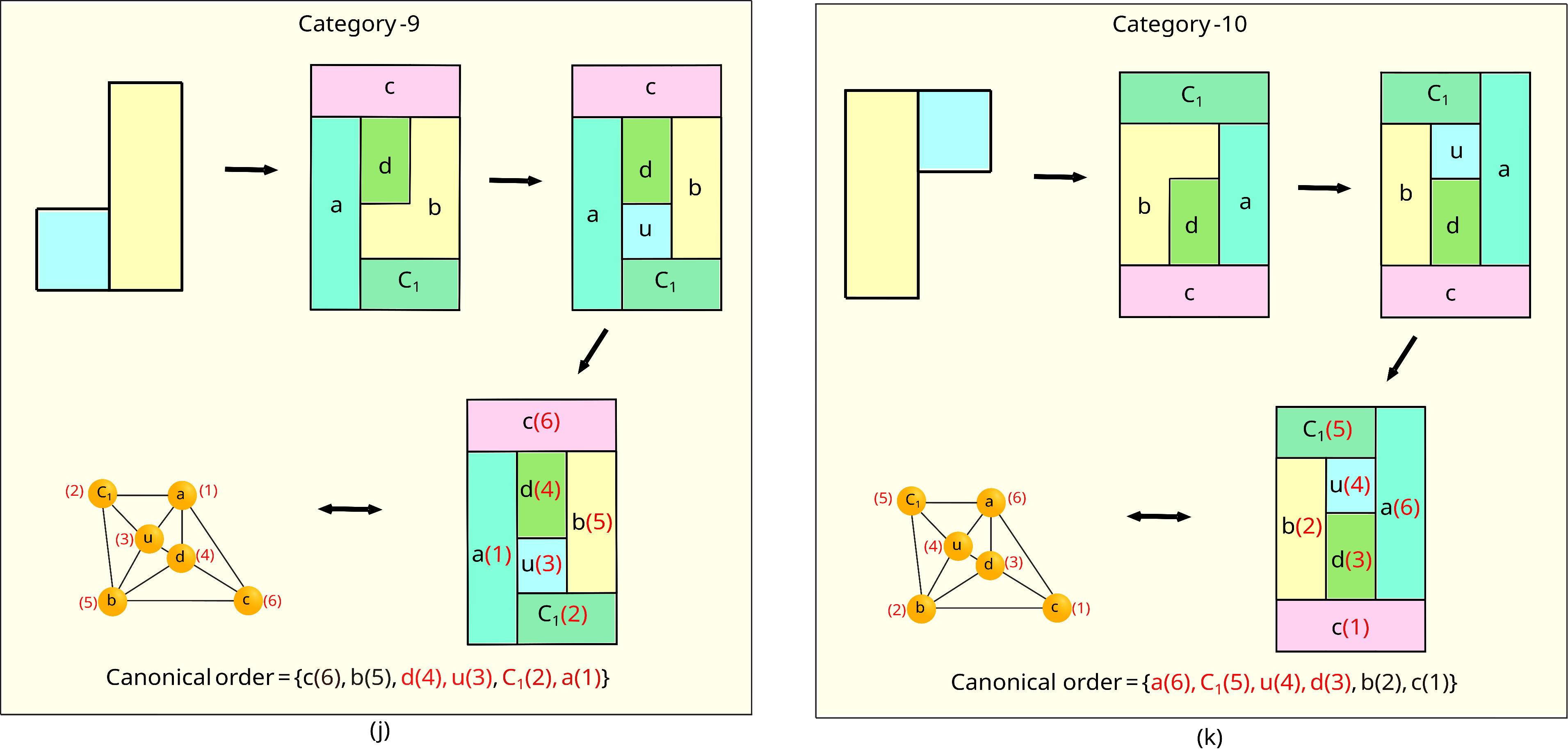}
    \caption{ (j-k) Canonical ordering associated with the $modified$ graph $K_L$ is shown for different categories (Category 9-10), each defined based on a $L$-shaped module relative to Class-1.}
   \label{L-7}
 \end{figure}
\subsection {$L$-shaped module generation within floor plan $F_L$}
\label{Category}
\begin{algorithm}
\caption{: Given a $G_L$ (triangulated graph) with an internal subgraph $K_L$, a $L$-shaped module can be constructed within the corresponding floor plan. }
\label{L-shaped}
\textbf{Input:} {A triangulated graph $G_L(V,E)$ having at least one internal subgraph isomorphic to $K_L$
 (see Figure ).} \\
\textbf{Output:} {An orthogonal floor plan $F_L$ containing a $L$-shaped module associated with an interior $K_L$ (see Figure ).}
 \algnewcommand\To{\textbf{to }}
  \begin{algorithmic}[1]
\If{\{$ \exists s_i \in ST_{L}/$\{$K_L$\}\}}
                    \State Call $Complex$ $Triangle$ $Removal$ ($G_L(V, E)$) algorithm \cite{roy2001proof}.
                 \Else
                      \State Go to line 5.
                  \EndIf
                  \State $\{C_1\} \gets \{nbd(a) \cap nbd(b)\} - \{c,d\}$, $V \gets V + \{u\}$, $E \gets E - \{(a,b)\}$,  $E \gets E + \{(u,C_1),(u,d),(u,a),(u,b)\}$, $E_1$ $=$ \{$a, b, c, d, u, C_1\}$, $G^1_L(V^1,E^1)$  $=$ $G_L(V,E)$.   
     \Comment{nbd($x$) = vertices adjacent to $x$}
       
                   \State Call $Four$-$Completion$ ($G^1_L(V, E)$) algorithm \cite{kant1997regular}, $G^1_L(V^1, E^1)$ $=$ $E \gets E + \{(N,S)\}$  
\State $ch(v) = 0$, $vi(v) = 0$, $St(v) = F$ $\forall v \in V^1$.
 \State $W$ = $v_1$, $S$ = $v_2$,  $St(W) = T$, $vi(N) = 2$ and $vi(E) = 1$,  $St(S) = T$, $L = \emptyset$, $i = n$ (number of vertices in $G^1_L$).
 
\Function{$CanonLabel$}{$G^1_{L}(V^1,E^1), PLabel(L), r$}  
   \State $i$ = $r$.
      \For{$r \gets i$ \To $3$}
      \If {$|L| > 0$} 
             \If{\{$ \exists v \in L$ s.t. $St(v) = F$, $ch(v) = 0$, $vi(v) \geq 2$, satisfies the priority order($L$)\}}
                    \State $v$ = $v_{r}$; $St(v) = T$. 
                  \Else
                \If {\{$ \exists v \in V^1 \setminus L$ s.t. $St(v) = F$, $ch(v) = 0$, $vi(v) \geq 2$\}}
            \State $v$ = $v_{r}$, $St(v) = T$.
           \Else  
           \State Go to Line 28 (break for loop).
            \EndIf   
            \EndIf
        \Else 
         \If {\{$ \exists v \in V^1\setminus E_1$ s.t. $St(v) = F$, $ch(v) = 0$, $vi(v) \geq 2$\}}
            \State $v$ = $v_{r}$, $St(v) = T$.
           \Else  
           \State Go to Line 28 (break for loop).
            \EndIf
        \EndIf

         \State {Let \{$w{_p},..., w{_q}$\} are the neighbours of $v{_{r}}$ (in this order around $v_{r}$) with $St(w{_s}) = F$ for $p\leq s \leq q$.}
        \State Increase $vi(w{_s})$ by 1 for $p\leq s \leq q$.
          \State Update $ch({v})$ for $w{_p},..., w{_q}$ and their neighbors.
     \EndFor
     \State \Return ($G_{L}^\text{cl}(V^1,E^1)$  $=$ $G^1_L(V^1,E^1)$, $r$). 
   \EndFunction 
   \Function{$Types$ $of$ $Priority$ $label_L$}{$G_{L}^\text{cl}(V^1,E^1)$, $r$} 

   \State $G_{L1}^\text{cl}(V,E)$  $=$ $G_{L}^\text{cl}(V^1,E^1)$, $r^1$ $=$ $r$.
 
    \State ($G^*$, $r^{*}$) = $CanonLabel(G_{L1}^\text{cl}(V,E),PLabel(\{C_1,u,d,a\}),r^1)$.
      \If{$r^{*} == 3$} \Comment{$r^{*} == 3$ implies each vertex of  $G^1_L$ is canonical ordered.}
       \State \Return $G^*(V,E)$ and Go to Line 53.
      \Else
          \State ($G^*$, $r^{*}$) = $CanonLabel(G_{L1}^\text{cl}(V,E),PLabel( \{d,u,a,C_1\}),r^1)$.
          
           \If{$r^{*} == 3$}
               \State  \Return  $G^*(V,E)$ and Go to Line 53.
           \Else
             \State ($G^*$, $r^{*}$) =  $CanonLabel(G_{L1}^\text{cl}(V,E),PLabel( \{d,u,C_1,a\}),r^1)$.
               \If{$r^{*} == 3$}
                 \State \Return  $G^*(V,E)$ and Go to Line 53.
               \Else
                 \State ($G^*$, $r^{*}$) = $CanonLabel(G_{L1}^\text{cl}(V,E),PLabel(\{a,d,u,C_1\}),r^1)$.
                  \If{$r^{*} == 3$}
                  \State  \Return  $G^*(V,E)$ and Go to Line 53.
                  \Else
                  \State ($G^*$, $r^{*}$) = $CanonLabel(G_{L1}^\text{cl}(V,E),PLabel(\{C_1,a,u,d\}),r^1)$.
                  \If{$r^{*} == 3$}
                  \State  \Return  $G^*(V,E)$ and Go to Line 53.
                   \Else
                 \State ($G^*$, $r^{*}$) = $CanonLabel(G_{L1}^\text{cl}(V,E),PLabel(\{a,C_1,u,d\}),r^1)$.
                  \State  \Return  $G^*(V,E)$ and Go to Line 53.
                   \EndIf
                  \EndIf
               \EndIf
         \EndIf
     \EndIf
     \EndFunction 
     
      \State  $G^2_L(V^2, E^2)$ $=$  $G^*(V,E)$.
      \State Call $REL$ $Formation$  $(G^2_L(V^2, E^2))$: $Algorithm$ \ref{RELF}. \Comment{ return $G^3_L(V^3, E^3)$, $m$}
    
     \State Call $Rectangular$ $floor plan$ $(G^3_L(V^3, E^3))$: Algorithm  \cite{roy2001proof}. \Comment{return $F'_L$}
\Function{$Merge$ $Rooms$}{$G^2_L(V^2,E^2)$, $F'_L$, $Enodes_L$, $m$} 
     \For{$u_i \in Enodes_L$}                          
           \State $M(x,a_i,F'_L)$ or $M(x,b_i,F'_L)$. \Comment{Since each $u_i$ =  ($a_i$, $b_i$) of $S_L$.}
            \EndFor
             \If{$m == 1$}
       \State  $M(u,a,F'_L)$
      \Else
          \State $M(u,b,F'_L)$ 
        \EndIf
    \State \Return $F_L$ $=$ $F'_L$
\EndFunction
\end{algorithmic} 
\end{algorithm}

This section illustrates our proposed Algorithm \ref{L-shaped} using an example where we generate a $L$-shaped module within the floor plan $F_L$ for the input graph $G_L$ with an internal subgraph $K_L$. \\\\
\textbf{1. Steps [1 to 4] of Algorithm \ref{L-shaped}} : \textbf{Complex Triangle Identification and Removal (Except $K_L$)} :\\
The method described by Roy et al. \cite{roy2001proof}  provides a way to identify and remove complex triangles from a graph. In this approach, once the complex triangles are identified, an edge of each complex triangle is selected and then subdivided by introducing a new vertex. As a result, the complex triangle is eliminated and transforms into a 4-cycle. The vertices that were located inside the original complex triangle, specifically the interior node, now form a rectangular sub-floor plan (as shown in Figure \ref{L-3}b, where the module $d$ illustrates this sub-floor plan). The four adjacent nodes are transformed into rectangular modules that surround and contain this sub-floor plan. The vertex introduced during the splitting process corresponds to a module that will eventually be merged with one of the neighbouring nodes from the original complex triangle. To eliminate all complex triangles within a graph, one must first identify a subset of edges (denoted as $S_L$) such that every complex triangle contains at least one edge from this subset. Next, new vertices are inserted along each edge in $S_L$, effectively splitting them. This systematic modification ensures the removal of all the complex triangles while preserving the triangularity in the graph.\\
Therefore, we apply the Complex Triangle identification and Removal algorithm as described in \cite{roy2001proof} to first identify all complex triangle in $G_L$ (see Figure \ref{L-9}(a-b)) and remove all complex triangles from the input triangulated graph $G_L$ (if there exist), leaving only a single complex triangle, denoted as $K_L$ (see Figure \ref{L-11}(a-c)). If any complex triangle other than $K_L$ exists, we split it by adding a new vertex $x$ and re-triangulate the input graph $G_L$. This produces an updated $G_L$ that contains no complex triangles except $K_L$ (see Figure \ref{L-11}(a-c)). \\\\
\textbf{2. Steps [5 to 6] of Algorithm \ref{L-shaped}} : \textbf{Removal of Complex Triangle $K_L$ and Four Completion Phase}:\\
To modify the remaining interior complex triangle $K_L$, we proceed by choosing the edge $(a,b)$ and eliminating it through the introduction of a new vertex $u$ (see Figure \ref{L-11}d). Subsequently, we insert new edges (i.e., $\{(u,C_1),(u,d),(u,a),(u,b)\}$) to maintain the triangulated structure of the graph. The resulting graph is denoted as $G^1_L$ (refer to Figure \ref{L-11}d). As a result of this transformation, $G^1_L$ no longer contains any complex triangles. \\
Once complex triangle elimination has been completed, the graph $G^1_L$ enters the four-completion process:  Kant and He \cite{kant1997regular} proposed a method for generating a rectangular floor plan from a bi-connected PTPG by introducing four new vertices, each representing one of the directions: East, South, West, and North. This approach requires identifying four corner vertices on the graph's outer boundary, which can be determined by applying the CIP technique introduced by Bhasker and Sahni \cite{bhasker1988linear}. The rectangular floor plan is constructed by selecting four boundary paths of the PTPG. Following the approach outlined in \cite{kant1997regular}, four paths {$P_4$, $P_3$, $P_2$, $P_1$} are first identified in $G^1_L$, after which directional vertices ($E$, $W$, $S$, $N$) are inserted into their respective paths. These inserted vertices correspond to the rectangular boundary modules that define the floor plan boundary. This procedure constitutes the four-completion phase, as illustrated in Figure \ref{L-12}(a-b).\\
After completing the four-completion phase, the edge $(N, S)$ is introduced into $G^1_L$, resulting in a 4-connected triangulated graph ($G^1_L$), as shown in Figure \ref{L-12}c. The graph $G^1_L$ then moves forward to the priority-wise canonical ordering step. \\
\begin{algorithm}
\caption{: Construction of REL-Formation from canonical ordered graph $G^2_L(V^2,E^2)$. }
\label{RELF} 
\textbf{Notations:}\\
{In a canonical ordered directed graph $G$,}\\
{1. $b_k$: $basic$-$edge$ ($b_k$) of a vertex $v_k$ is defined as the incoming edge ($v_l,v_k$) $=$ $b_k$ for which $l$ $<$ $k$ and is minimal (here $i \leq l \leq j$).}\\
{2. $C_k$: The set $C_k$ $=$ \{$v_i,.....,v_j$\} of $v_k$  is defined as the incoming edges from $v_i,.....,v_j$  belonging to $C_{k-1}$ (the exterior face of $G_{k-1}$) along the path from $v_1$ to $v_2$ in this order.}\\
{3. $R_k$: The set $R_k$ $=$ \{$v_m,.....,v_n$\} of $v_k$  is defined as the outgoing edges from $v_m,.....,v_n$ in ($G$ $-$ $G_{k-1}$) in anti-clockwise direction around $v_k$.}\\
{4. $lp_k$: $left$ $point$ of $v_k$, $rp_k$: $right$ $point$ of $v_k$, $le_k$: $left$ $edge$ of $v_k$, $re_k$: $right$ $edge$ of $v_k$. }\\
\textbf{Input :} { $G^2_L(V^2,E^2)$.}\\
\textbf{Output :} { $G^3_L(V^3,E^3)$ (REL), $m$.}
 \algnewcommand\To{\textbf{to }}

  \begin{algorithmic}[1]
 \State $E^2 \gets E^2 - \{(N,S)\}$, $T_1$ $=$ $\emptyset$ ,$T_2$  $=$ $\emptyset$.    
   
   \State For each edge \{$v_i,v_j)$\} $ \in G^2_L(V^2, E^2)$: direct $v_i \to v_j$ for i $<$ j (Except (W, S), (S, E), (E, N), (N, W)).   
     \For{$k=n-2$ to 3}   \Comment{see Figure \ref{CL}}
     
        
       \State Compute $b_k$ of $v_k$.
       \State Compute $C_k$ $=$ \{$v_i,.....,v_j$\} of $v_k$ where   $lp_k$ = $v_i$ and $rp_k$ = $v_j$. 
       \State  Compute $R_k$ $=$ \{$v_m,.....,v_n$\} of $v_k$ where $(v_i,v_m)$ $=$ $le_k$ and  $(v_i,v_n)$ $=$ $re_k$.  
    \EndFor
      \For{$k=3$ to $n-2$}   
        \State $le_k$ of $v_k \in T_1$.  
        \State $re_k$ of $v_k \in T_2$. 
        \State $b_k$ $=$ ($v_l,v_k$) of vertex $v_k \in T_2$ if $v_l$ $=$ $lp_k$ and $ \in T_1$ if $v_l$ = $rp_k$, otherwise,  $ \in T_1$ or $T_2$.
        \EndFor
        \State $(a,u) \in T_{x1}$, $(a,C_1) \in T_{x2}$, $(b,u) \in T_{x3}$, $(b,C_1) \in T_{x4}$. \Comment{$T_{x1}, T_{x2}, T_{x3}, T_{x4} \in  \{T_1,T_2\}$ i.e., x1, x2, x3, x4 $\in  \{1,2\}$}
        \If{ $x1$ $\neq x2$}
          \State $m$ $=$ $1$. 
        \ElsIf{$x3$ $\neq x4$}
          \State $m$ $=$ $2$
        \Else
        \State $\{x\} \gets \{nbd(a) \cap nbd(C_1)\} - \{u\}$
        \State $(C_1,x) \in T_{x5}$, $(a,x) \in T_{x6}$
        \If{$x2 == x5$}
        \State $(a,C_1) \in T_{x6}$, $m$ $=$ $1$. 
        \Else
        \State $(x,C_1) \in T_{x2}$, $(a,C_1) \in T_{x6}$, $m$ $=$ $1$. 
        \EndIf
         \EndIf

       
    \State \Return  $G^3_L(V^3,E^3)$ = $G^2_L(V^2,E^2)$, $m$. 
\end{algorithmic} 
\end{algorithm}
\textbf{3. Steps [7 to 53] of Algorithm \ref{L-shaped}} : \textbf{Priority wise Canonical ordering}:
This section presents Algorithm \ref{L-shaped}, which computes a canonical ordering graph $G^2_L$ as defined in Definition 4 (see Terminology Section \ref{Preliminaries}) based on a 4-connected triangulated graph $G^1$, derived through Steps 7–53 of Algorithm \ref{L-shaped}. The prioritized canonical ordering is tailored to support the generation of a $L$-shaped module in the floor plan $F_L$, with respect to the associated PTG $G_L$, which includes at least one interior complex triangle $K_L$. \\
Figures \ref{L-5}, \ref{L-6}, and \ref{L-7} depict ten distinct categories (\textbf{Category 1–10}) that characterize the canonical ordering of the modified graph $K_L$, which corresponds to a $L$-shaped module in the context of Class-1 (refer to Figure \ref{L-1}). These category types will serve as a reference framework for constructing the canonical ordering of the graph $G^1_L$ in subsequent steps. Also, we observed that when creating module $a$ as a $L$-shaped module in the final floor plan related to the subgraph $K_L$, the sharing of walls, i.e., horizontal ($T_1$) and vertical ($T_2$) between module $a$ and module $c$, and between module $b$ and modules $c$, module $C_1$, do not interfere with the formation of module $a$ as a $L$-shape. In other words, module $a$ can able to share both $T_1$ and $T_2$ with module $c$, and module $b$ can able to share $T_1$ and $T_2$ with modules $c$ and $C_1$, without preventing module $a$ from forming a $L$-shape (as illustrated in Figure \ref{L-8} under Category 1, where eight valid $L$-shaped configurations are shown). However, the wall shared between module $a$ and module $C_1$ must be fixed. As shown in Figure \ref{L-10}, if this wall is not fixed, module $a$ could instead form a $T$-shaped configuration rather than the intended $L$-shape, leading to inconsistencies in the layout. Therefore, it is not necessary to specify a fixed canonical order for vertices $b$ and $c$. From the set of vertices $\{a, b, c, d, u, C_1\}$, only the subset $\{a, d, u, C_1\}$ needs to have an assigned order based on priority during the canonical ordering process of graph $G^1_L$.\\
Refer to Figures \ref{L-5}, \ref{L-6}, \ref{L-7}: Among the ten categories (1–10) generated using Class-1 (where module $u$ is merged with $a$ as shown in Figure \ref{L-1}), only six of them produce unique canonical ordering for the vertices $\{a, u, d, C_1\}$. In some cases, the canonical order remains the same, though the Regular Edge Labeling differs. As a result, the ten categories are grouped into six broader classes (with respect to canonical order), ordered A through F:\\
\textbf{1. Category A:} $\{C_1, u, d, a\}$ – the canonical order in Category 2 matches that of Category 4.\\
\textbf{2. Category B:} $\{d, u, a, C_1\}$ – the canonical order in Category 1 matches that of Category 3.\\
\textbf{3. Category C:} $\{d, u, C_1, a\}$ – corresponds to Category 9.\\
\textbf{4. Category D:} $\{a, d, u, C_1\}$ – the canonical order in Category 5 matches that of Category 7.\\
\textbf{5. Category E:} $\{C_1, a, u, d\}$ – the canonical order in Category 6 matches that of Category 8.\\
\textbf{6. Category F:} $\{a, C_1, u, d\}$ – corresponds to Category 10.\\
Algorithm \ref{L-shaped} is structured to assign orders to the vertices, namely {$a, d, u, C_1$}, based on a predefined priority sequence outlined by \textbf{Categories A–F}. The canonical ordered graph $G^2_L$ is then constructed by executing Steps 7 through 53 of Algorithm \ref{L-shaped}, as detailed below.\\\\
\textbf{(A) Step 7 of Algorithm \ref{L-shaped})}: For each vertex $v$ in $G^1_L$, we maintain the following three variables:\\ \textbf{1. St($v$):}  A status flag indicating whether $v$ has been ordered: (T) for true, (F) for false.\\
\textbf{2. vi($v$):} The number of neighbouring vertices $u$ for which St($u$) = T, i.e., neighbours that have already been ordered.\\
\textbf{3. ch($v$):}  The count of chords connected to $v$ within the subgraph of $G^1_L$ formed by the vertex set $V^1$ excluding any vertex $u$ such that St($u$) = T.\\
Using Algorithm \ref{L-shaped}, we iteratively update vertex properties by examining their neighbours and evaluating predefined conditions. Vertices are ordered in a specific sequence, with each step updating the associated variables and determining the next candidate vertices for canonical ordering. \\
 Figures \ref{L1}-\ref{L7} provide a detailed breakdown of this ordering process for the canonical ordering of vertices in $G^1_L$.\\\\
 \textbf{(B) Steps [8 to 28] of Algorithm \ref{L-shaped}:} At the outset, in the input graph $G^1_L$, the outer boundary vertices $N$, $W$, and $E$ are canonical ordered as $v_{19}$, $v_1$, and $v_2$, respectively. The external face of $G^1_{L_{18}}$ corresponds to the cycle $C^L_{18}$ (refer to Figure \ref{L1}(ii)). The process begins by identifying the unordered neighbors of vertex $v_{19}$ along the cycle $C^L_{18}$ i.e., vertices $E$, $3$, $2$, and $1$. For each of these, we update two attributes: ch($v$) and vi($v$). Following this update, vertex $E$ satisfies the required criteria for canonical ordering: ch($E$) = 0 and vi($E$) = 2. Consequently, $E$ is canonically ordered as $v_{18}$, and the relevant properties for its neighbouring vertices are adjusted accordingly (see Figure \ref{L1}(iii)).\\
 Next, we proceed with the step-by-step canonical ordering of the remaining unmarked vertices in $G^1_L$ (excluding those in the set \{$a, b, c, d, u, C_1\}$). At each step, we identify an unordered vertex $v$ that meets the following criteria: $ch(v)$ = 0 (indicating no internal chords) and $vi(v)$ $\geq 2$ (i.e., it has at least two already ordered neighbours). If multiple vertices satisfy these conditions, one is selected, and the canonical order is $v_r$ (for the $r$th step of ordering). After ordering $v_r$, we update its neighbours' properties, specifically, the chord and visit status on the current boundary cycle $C^L_{k-1}$. The process then repeats to find and order the next suitable vertex, $v_{r-1}$ (refer to Figures \ref{L1} to \ref{L3}(ix)). This iterative ordering continues until only the next possible vertex to label is from the set \{$a, b, c, d, u, C_1\}$ (see Figure \ref{L3}(x)). At this point, the canonical ordering process paused for transition into the next stage (since we have to move to the next step, where we will order the remaining vertex with respect to the defined priority order), and the resulting ordered graph is denoted as $G_{L1}^\text{cl}(V^1, E^1)$ (see Figure \ref{L3}(x)).\\\\
 \textbf{Point A:} In Steps 8 through 28 of Algorithm \ref{L-shaped}, the main aim is to canonical order as many vertices of $G^1_L$ as possible, excluding the specific set  \{$a, b, c, d, u, C_1\}$. However, certain vertices outside this set may remain unordered due to constraints like having a count ($ch(v)$) of one or a vertex ($vi(v)$) less than one, for instance, vertices 1, 6, and 13 in Figure \ref{L3}(ix). Thus, the initial goal is to canonical order the majority of vertices in $G^1_L$, leaving out only the designated set \{$a, b, c, d, u, C_1\}$. \\\\
 \textbf{Point B:} Procedure for ordering the set \{$a, b, c, d, u, C_1\}$ vertices: 
For a set of \{$a, b, c, d, u, C_1\}$ vertices, a distinct canonical ordering process is used by calling function ($Types$ $of$ $Priority$ $label_L$) in Algorithm \ref{L-shaped} (see Figures \ref{L3} to \ref{L7}). These vertices are canonically ordered with respect to any one of the six Categories (Categories A-F). We will try to order the set \{$a, d, u, C_1\}$ vertices using any one of the six categories, while checking one by one (since it is not necessary to specify a fixed canonical order for vertices $b$ and $c$: explained above). \\ 
Suppose a valid ordering of the vertices \{$a, d, u, C_1\}$ can be achieved under any of the defined categories (A through F). In that case, we will canonical order it and forward the generated graph $G^2_L$ for constructing the corresponding regular edge labeling.\\\\ 
\textbf{(C) Steps [29 to 53] of Algorithm \ref{L-shaped}: Vertex Ordering in $G^1_L$ (including the set {$a, b, c, d, u, C_1$}):}\\
Given the previously obtained canonically ordered graph $G_{L1}^\text{cl}(V^1, E^1)$, we begin by attempting to assign an ordering to the subset \{$a, d, u, C_1\}$ in accordance with each of the six predefined categories (A–E), evaluated sequentially (refer to Figure \ref{L3}(xi)). The input graph does not satisfy the structural conditions required for Category A: see Figure \ref{L3} (xi)-(xiv) (though if such a match were found for a different graph, it would proceed to the regular edge labeling stage). A similar failure occurs when tested with respect to Category B (see Figure \ref{L4}(xv)-(xvii)). Subsequently, the graph is evaluated for compatibility with Category C, which also proves unsuccessful (see Figure \ref{L4}(xviii)-(xx)). Finally, the ordering is examined under Category D, where the required conditions are satisfied (see Figures \ref{L5} to \ref{L7}). The resulting canonically ordered graph is then designated as $G^2_L$. \\
Figures \ref{L1} through \ref{L7}  present a step-by-step breakdown of the process used to derive the canonical ordered graph $G^2_L$ from the initial input graph $G^1_L$, following Steps 7 to 53 of the proposed Algorithm \ref{L-shaped}. The resulting graph is now forward for the subsequent regular edge labeling phase.\\\\
\textbf{4. Step 54 of Algorithm \ref{L-shaped}: Generation of Regular Edge Labeling:}\\
To generate a rectangular floor plan of $G^2_L$, we will apply Algorithm \ref{RELF}, which constructs a regular edge labeling for $G^2_L$ based on its generated canonical ordering, following the methodology outlined in \cite{kant1997regular}.\\\\
\textbf{Step [1-6] of Algorithm \ref{RELF}:} Firstly computes the directed edges for $G^2_L$ (see Figure \ref{RELL1}(a-b)), and then generates a list of basis edges $b_k$, along with the corresponding sets $C_k$ and $R_k$ for all vertices (see Figure \ref{RELL1}c).\\\\
\textbf{Step [7-10] of Algorithm \ref{RELF}:} Using the basis edges $b_k$ along with the sets $C_k$ and $R_k$ for each vertex, a regular edge labeling (REL), consisting of $T_1$ and $T_2$, is constructed for $G^2$ (see Figure \ref{RELL1}d).\\\\
\textbf{Step [11-23] of Algorithm \ref{RELF}:} We will examine which type of wall i.e., either $T_1$ or $T_2$ is shared between module $a$ and $C_1$, as well as between module $a$ and $u$, and do the same for module $b$ (refer to Figure \ref{RELL1}d). To ensure that module $a$ forms a $L$-shaped structure, it must share different types of walls with $C_1$ and $u$, meaning one connection should fall under $T_1$ and the other under $T_2$. There are three possible cases to consider. We evaluate each case individually and determine the value of $m$, which indicates whether module $u$ should be merged with module $a$ or with module $b$.\\\\
\textbf{Case 1:}  If module $a$ shares walls of opposite types ($T_1$ and $T_2$) with both $u$ and $C_1$, then set $m = 1$, indicating that module $u$ should be merged with module $a$ (  see Figure \ref{RELL1}d: with respect to input, Case 1 exist). \\
\textbf{Case 2:} If Case 1 does not apply, but module $b$ shares walls of different types with $u$ and $C_1$, then set $m = 2$, meaning module $u$ should be merged with module $b$ (see Figure \ref{L-7}j). \\
\textbf{Case 3:} If neither Case 1 nor Case 2 is satisfied, we modify the edge ($x$, $C_1$) (Case-A) or ($x$, $C_1$) and then ($a$, $C_1$) (Case-B): by flipping it ($T_1$ or $T_2$), where $x$ = \{$ \text{nbd}(a) \cap \text{nbd}(C_1)\}$ - \{$u\}$. After this adjustment, set $m = 1$ (refer to Figure \ref{CL-6}(e-h)).\\
This results in generation of Regular Edge Labeling $G^3_L$ with returning $m$ value (either 1 or 2).\\\\
\textbf{5. Step 55 of Algorithm \ref{L-shaped}: Generation of Rectangular Floor plan:}\\
Once we have generated the Regular Edge Labeling ($G^3_L$) using the canonical order, we use this REL to construct a rectangular floor plan ($F'_L$) for the generated graph $G^2_L$.
This is done by following the method described by Bhasker and Sahni \cite{bhasker1988linear} (see Figure \ref{REL-L2}a).\\\\
\textbf{6. Steps [56 to 63] of Algorithm \ref{L-shaped}: Generation of $L$-shaped Module within Orthogonal Floor plan by Merging Modules:}\\
The $Merge$ $Rooms$ function describes the method for integrating rectangular modules that correspond to extra vertices introduced during the removal of complex triangles in $G^1_L$. It requires four inputs: the rectangular floor plan $F'_L$ derived from the graph $G^2_L$, the canonical ordered graph $G^2_L$ itself, a set of extra nodes $Enodes_L$, and a parameter $m$, which determines whether the extra module $u$ should be combined with module $a$ or module $b$. The outcome is a $L$-shaped module in the final floor plan that aligns with the structure of $K_L$.\\
See Figure \ref{REL-L2}(b-c) (here $S_L$ = $(12, 3)$ and $Enodes_L$ = $x$: merging module $x$ with $3$ and module $u$ with $a$), where we obtained the floor plan $F_L$ with a $L$-shaped module from a rectangular floor plan $F'_L$ while using the $function$ $Merge$ $Rooms$.\\\\
\textbf{Hence, given a $G_L$ (triangulated graph) with an internal subgraph $K_L$, a $L$-shaped module can be constructed within the corresponding floor plan $F_L$ using our proposed Algorithm \ref{L-shaped}}.

 \begin{figure}
   \centering
    \includegraphics[width=1\textwidth]{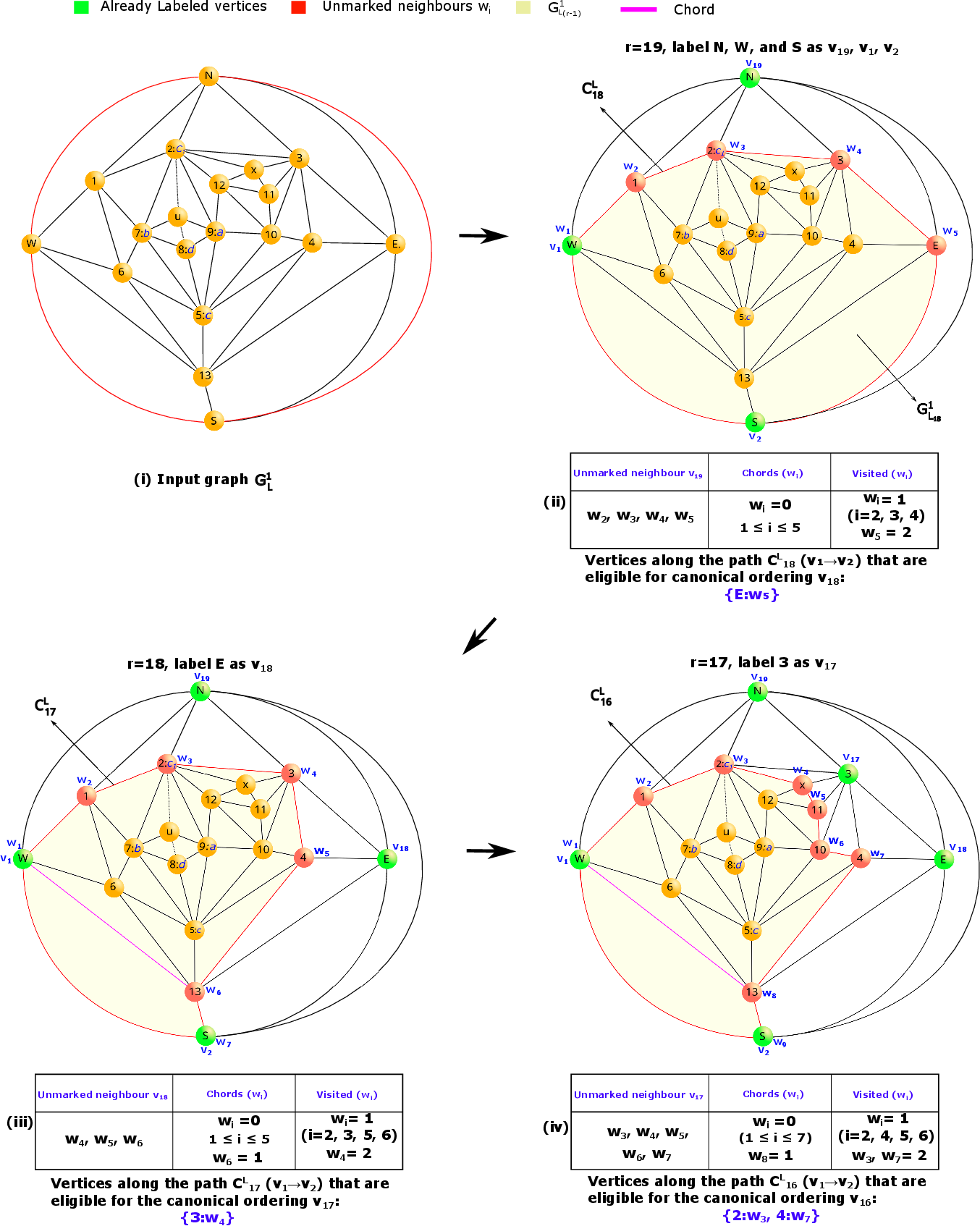}
    \caption{(i-iv) Step-by-step canonical ordering for the input graph $G^1_L.$}
   \label{L1}
 \end{figure}

  \begin{figure}
   \centering
    \includegraphics[width=1\textwidth]{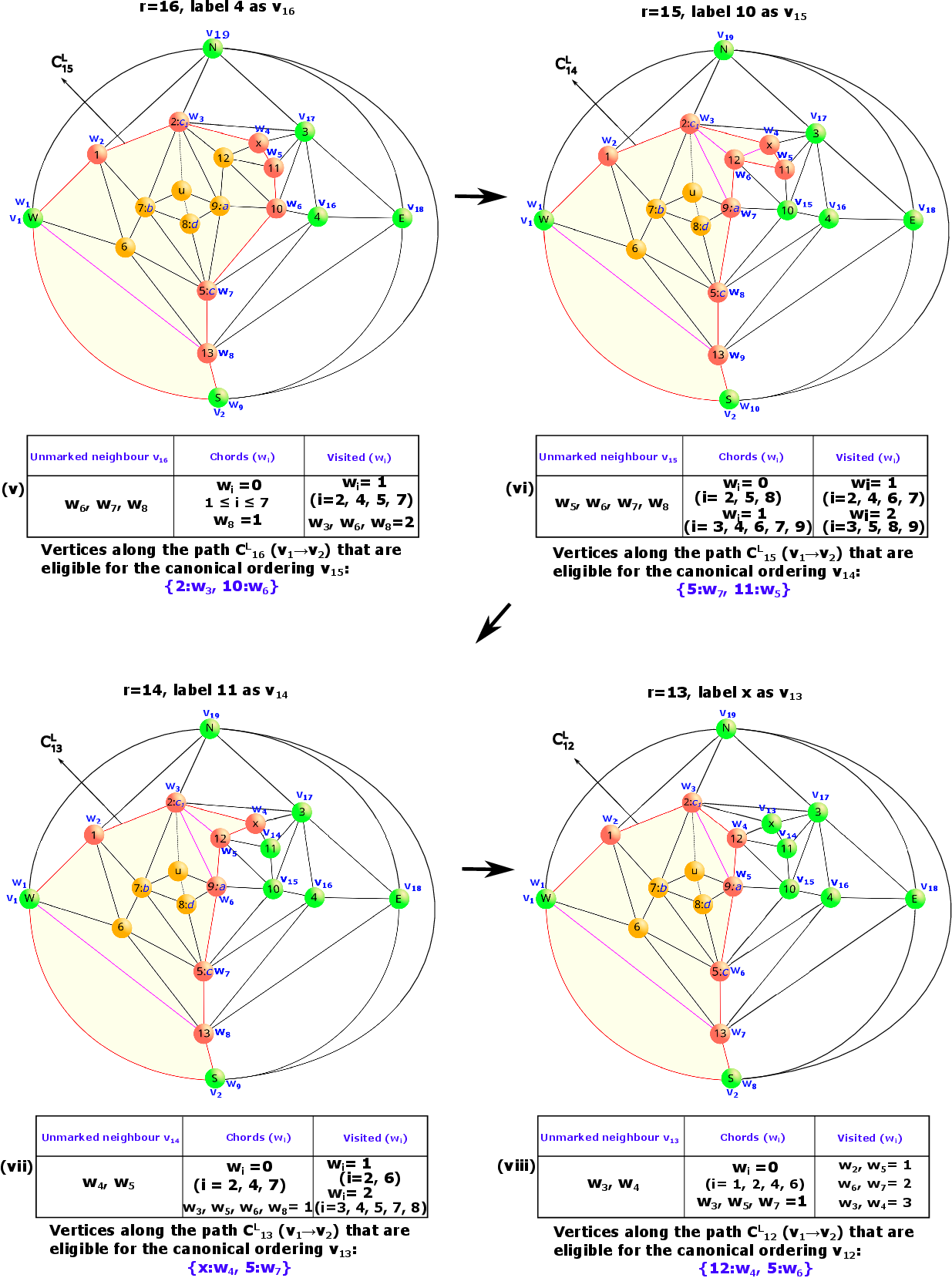}
    \caption{(v-viii) Step-by-step canonical ordering for the input graph $G^1_L.$}
   \label{L2}
 \end{figure}

   \begin{figure}
   \centering
    \includegraphics[width=1\textwidth]{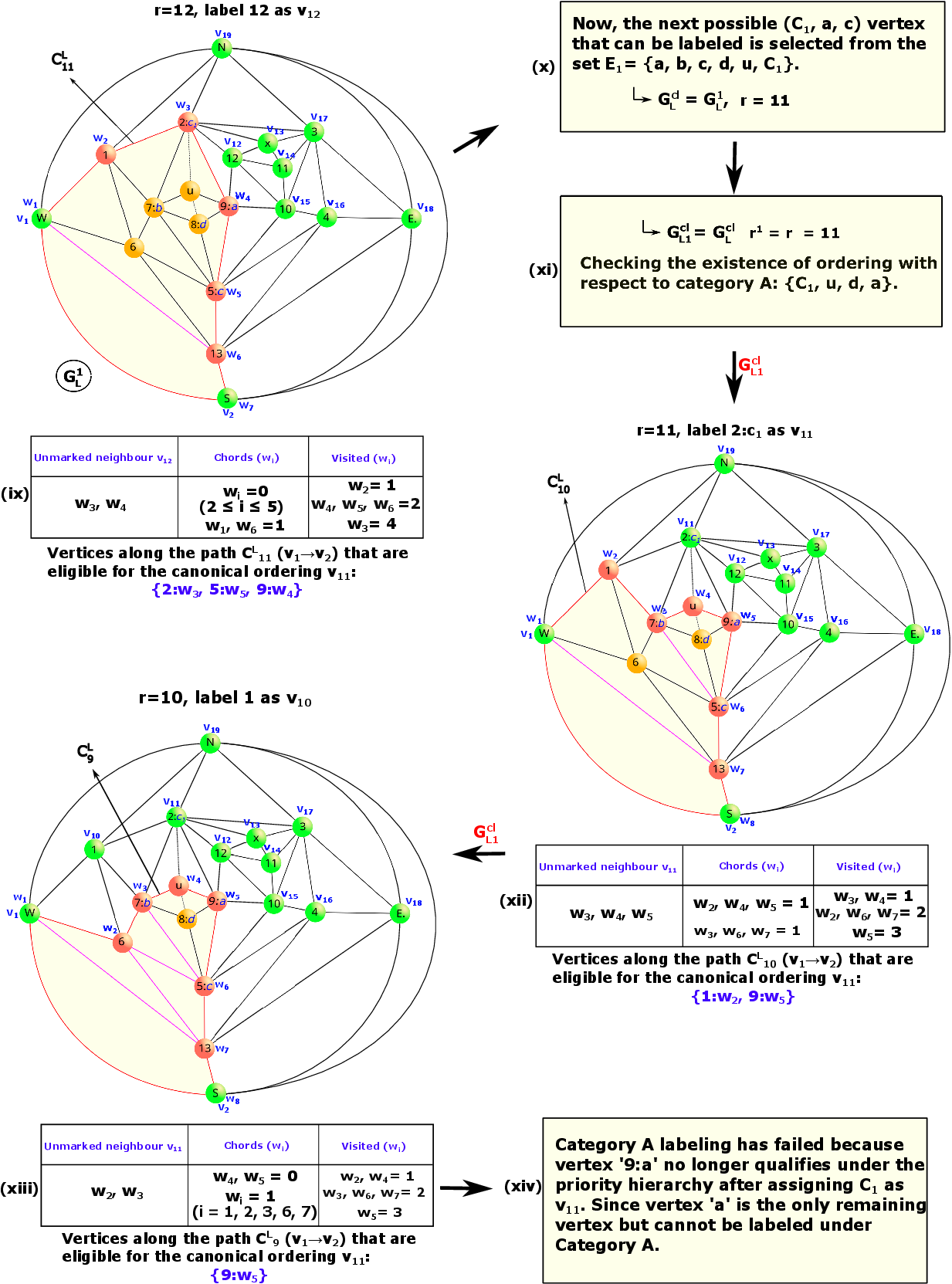}
    \caption{(ix-xiv) Checking the existence of a canonical ordering associated with category A}
   \label{L3}
 \end{figure}

  \begin{figure}
   \centering
    \includegraphics[width=1\textwidth]{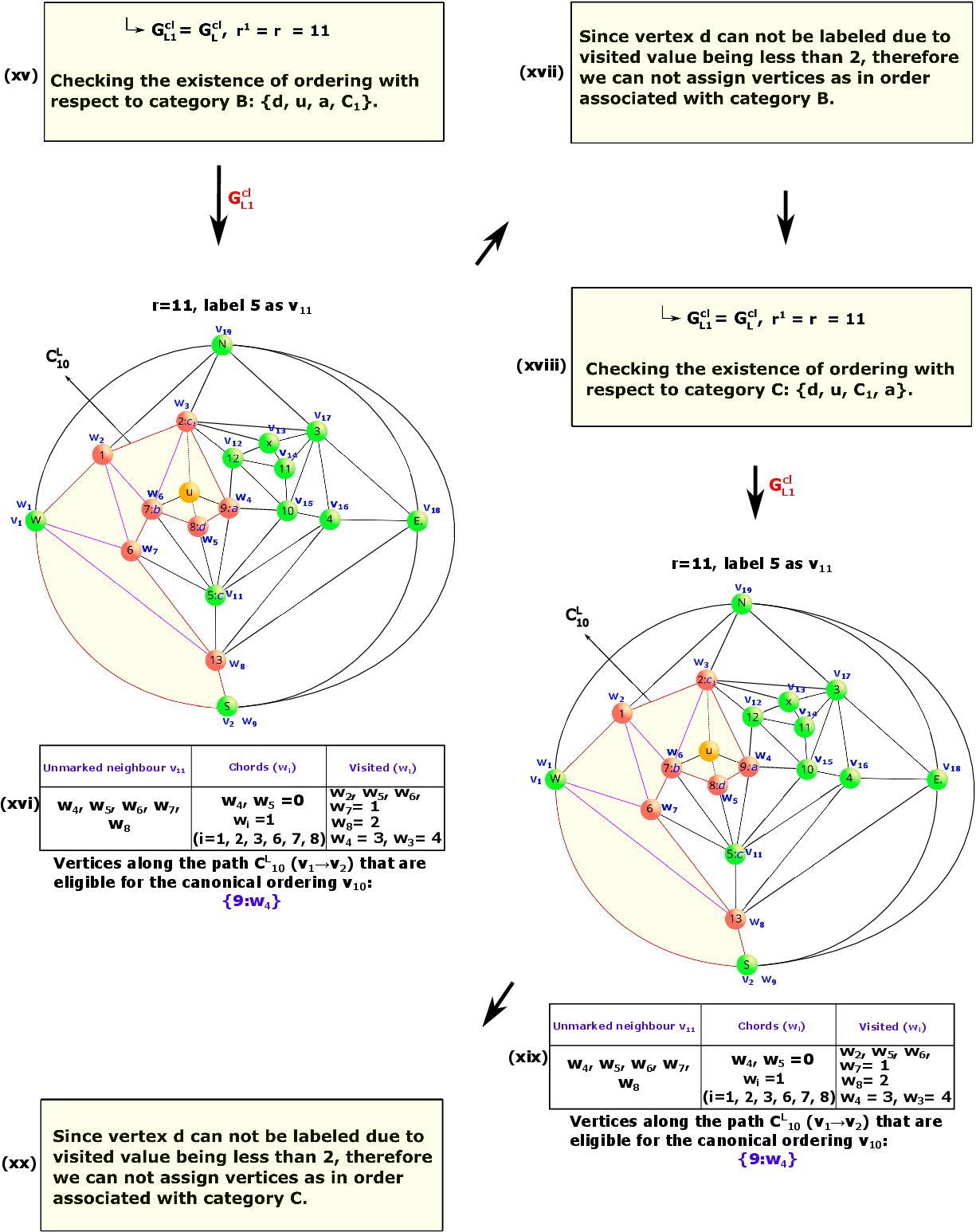}
    \caption{(xv-xvii) Checking the existence of a canonical ordering associated with category B. (xviii-xx) Checking the existence of a canonical ordering associated with category C.}
   \label{L4}
 \end{figure}

   \begin{figure}
   \centering
    \includegraphics[width=1\textwidth]{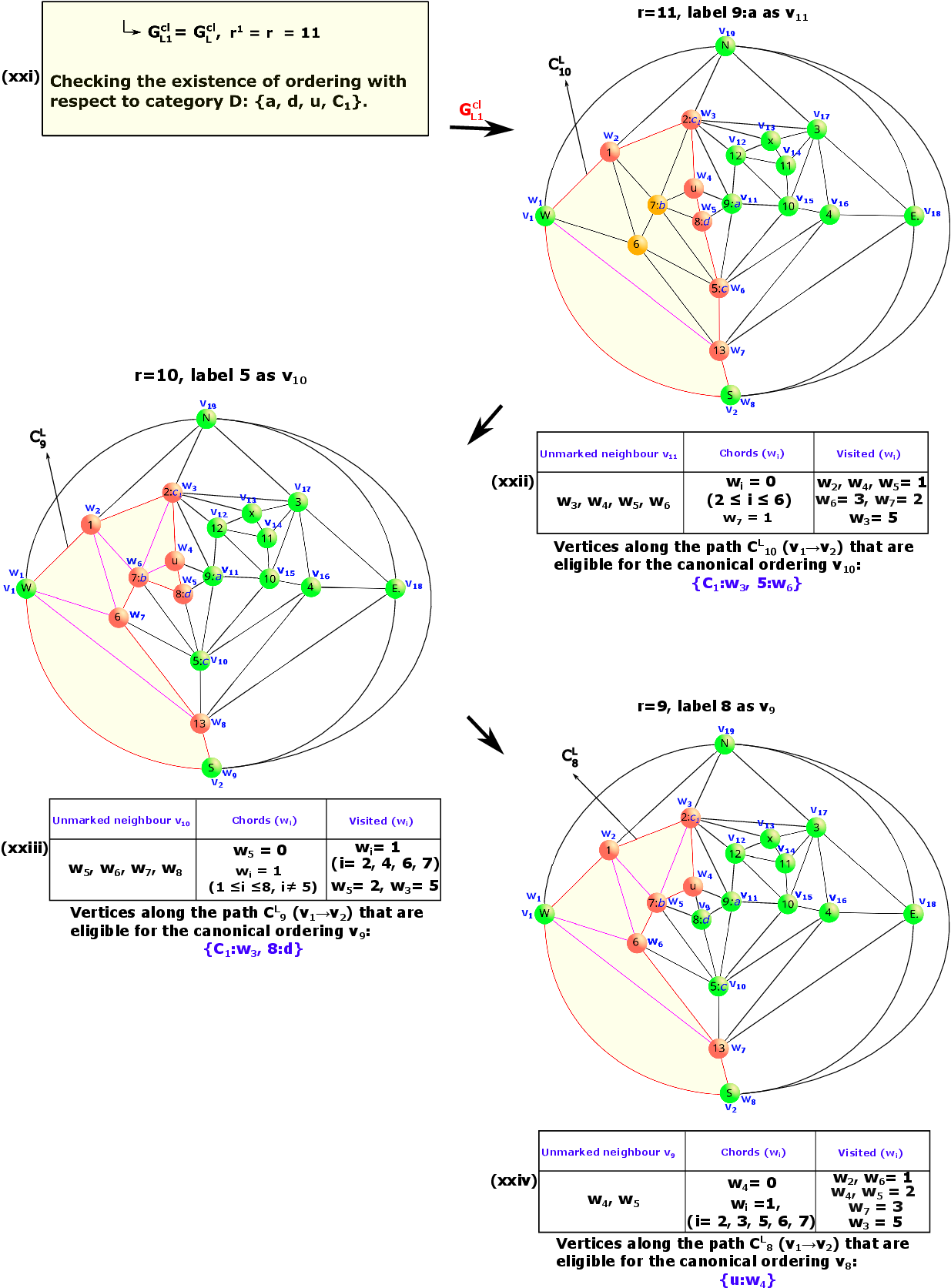}
    \caption{(xxi-xxiv) Checking the existence of a canonical ordering associated with category D.}
   \label{L5}
 \end{figure}

 \begin{figure}
   \centering
    \includegraphics[width=1\textwidth]{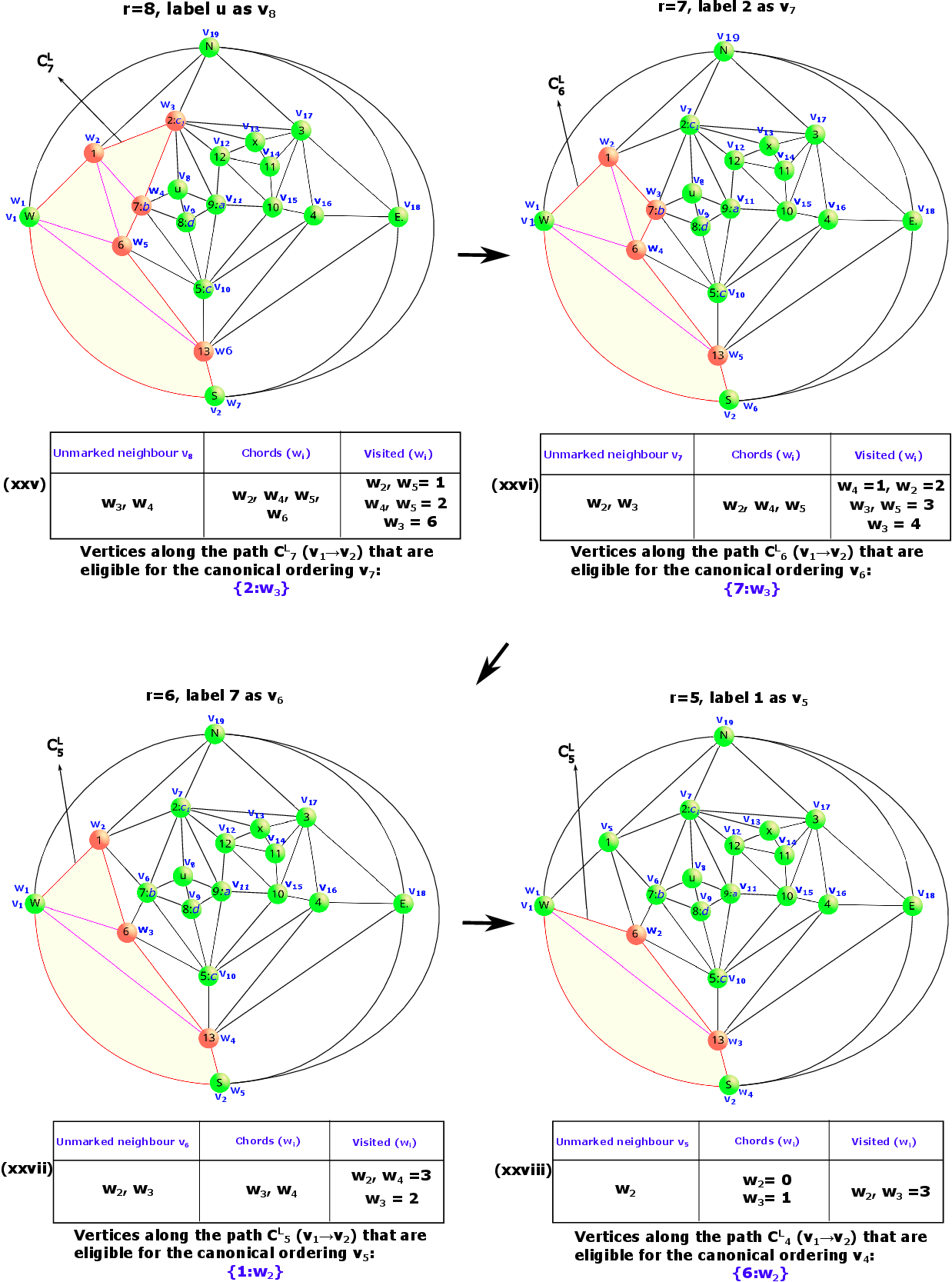}
    \caption{(xxv-xxviii) Assigning a priority-wise order to the vertices in accordance with category D.}
   \label{L6}
 \end{figure}

 \begin{figure}
   \centering
\includegraphics[width=0.9\textwidth]{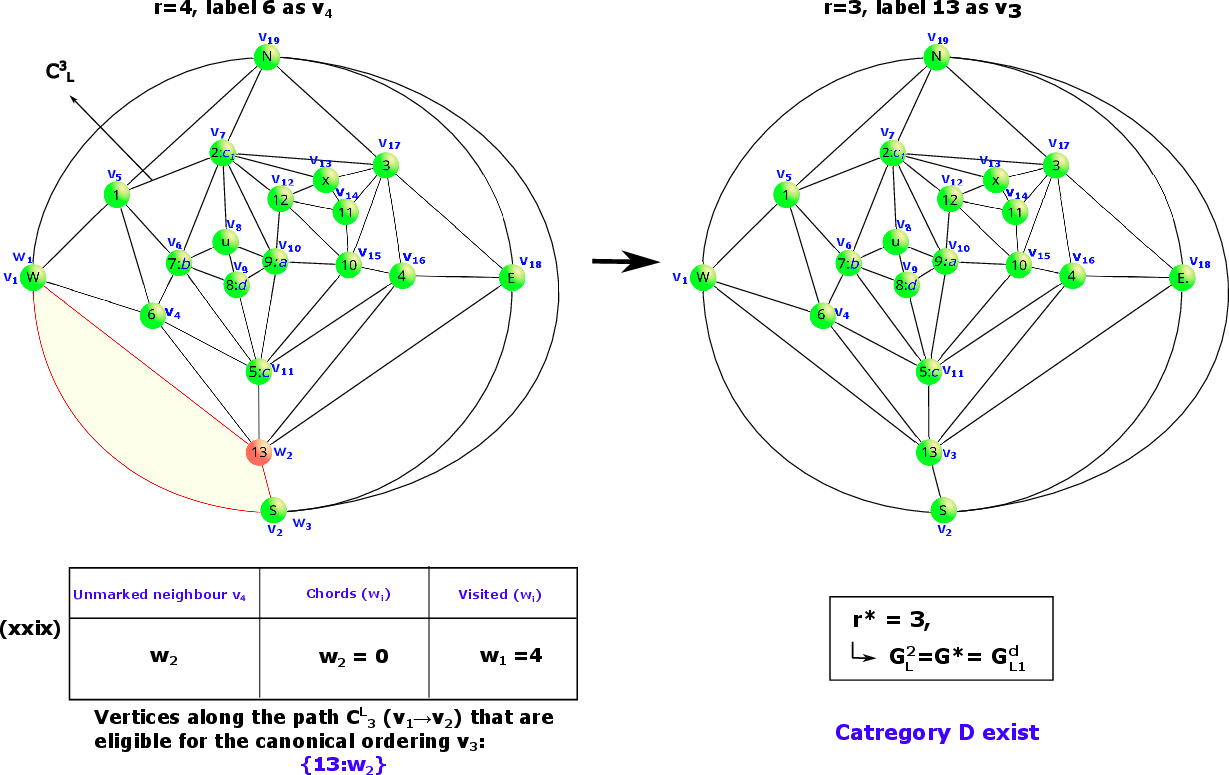}
    \caption{Category D exists. Hence obtained a canonical ordered graph $G^2_L$
.}
   \label{L7}
 \end{figure}
   \begin{figure}
    \centering
    \includegraphics[width=0.5\textwidth]{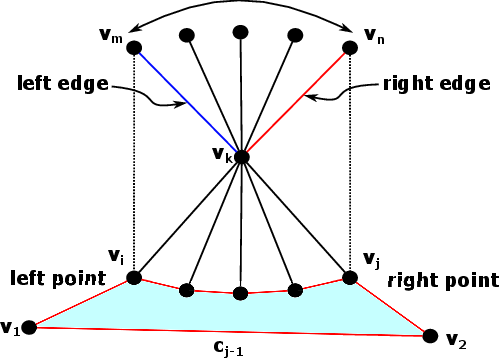}
    \caption{For each $j$, the outer face boundary of $G_{j−1}$ along with the adjacent vertices of $v_j$.}
   \label{CL}
 \end{figure}
\begin{algorithm}[H]

\caption{: Given a $G_T$ (triangulated graph) with an internal subgraph $K_T$, a $T$-shaped module can be constructed within the corresponding floor plan.}
\label{TLabel}
\textbf{Input:} { A $G_T$ (triangulated graph) with an internal subgraph $K_T$ (see Figure).} \\
\textbf{Output:} {An orthogonal floor plan $F_T$ containing a $T$-shaped module associated with an interior $K_T$ (see Figure ).}
 \algnewcommand\To{\textbf{to }}
  \begin{algorithmic}[1]
\If{\{$ \exists s_i \in ST_{T}/$\{$K_T$\}\}}
                    \State Call $complex$ $Triangle$ $Removal$ ($G_T(V, E)$) algorithm \cite{roy2001proof}.
                 \Else
                      \State Go to line 5.
                  \EndIf
                  \State $E \gets E - \{(a,c)\}$, $V \gets V + \{u\}$, $E \gets E + \{(u,e),(u,f),(u,a),(u,c)\}$, $G^1_T(V^1,E^1)$  $=$ $G_T(V,E)$.

                   \State Call $Four$-$Completion$ ($G^1_L(V, E)$) algorithm \cite{kant1997regular}, $G^1_T(V^1, E^1)$ $=$ $E \gets E + \{(N,S)\}$  
\State $ch(v) = 0$, $vi(v) = 0$, $St(v) = F$ $\forall v \in V^1$.
 \State $W$ = $v_1$, $S$ = $v_2$,  $St(W) = T$, $vi(N) = 2$ and $vi(E) = 1$,  $St(S) = T$.
 \State  $t = n$ (number of vertices in $G^1_T$).
\Function{$CanonLabel_T$}{$G^1_T(V^1,E^1), t$}  
      \For{$t \gets i$ \To $3$}
          \State $ \exists v \in V^1$ s.t. $St(v) = F$, $ch(v) = 0$, $vi(v) \geq 2$\ \textbf{then}
            \State $v$ = $v_t$, $St(v) = T$.
            \State {Let \{$w{_p},..., w{_q}$\} are the neighbours of $v{_t}$ (in this order around $v{_t}$) with $St(w{_r}) = F$ for $p\leq r \leq q$.}
            \State Increase $vi(w{_r})$ by 1 for $p\leq r \leq q$.
            \State Update $ch({v})$ for $w{_p},..., w{_q}$ and their neighbours.
     \EndFor
     \State \Return $G^2_T(V^2, E^2)$ $=$ $G^1_T(V^1,E^1)$
   \EndFunction 
      \State Call $REL$ $Formation$  $(G^2_T(V^2, E^2))$: Algorithm \cite{kant1997regular}. \Comment{ return $G^3_T(V^3, E^3)$}
    
     \State Call $Rectangular$ $floor plan$ $(G^3_L(V^3, E^3))$: Algorithm  \cite{kant1997regular}. \Comment{return $F'_T$}
\Function{$Merge$ $Rooms$}{$G^2_T(V^2,E^2)$, $F'_T$, $Enodes_T$} 
     \For{$u_i \in Enodes_T$}                          
           \State $M(x,a_i,F'_T)$ or $M(x,b_i,F'_T)$. \Comment{Since each $u_i$ =  ($a_i$, $b_i$) of $S_{T}$.}
            \EndFor
            
             \State $M(u,a,F'_T)$ or $M(u,c,F'_T)$.
    \State \Return $F_T$ $=$ $F'_T$
\EndFunction
\end{algorithmic} 
\end{algorithm}

   \begin{figure}
   \centering
    \includegraphics[width=1\textwidth]{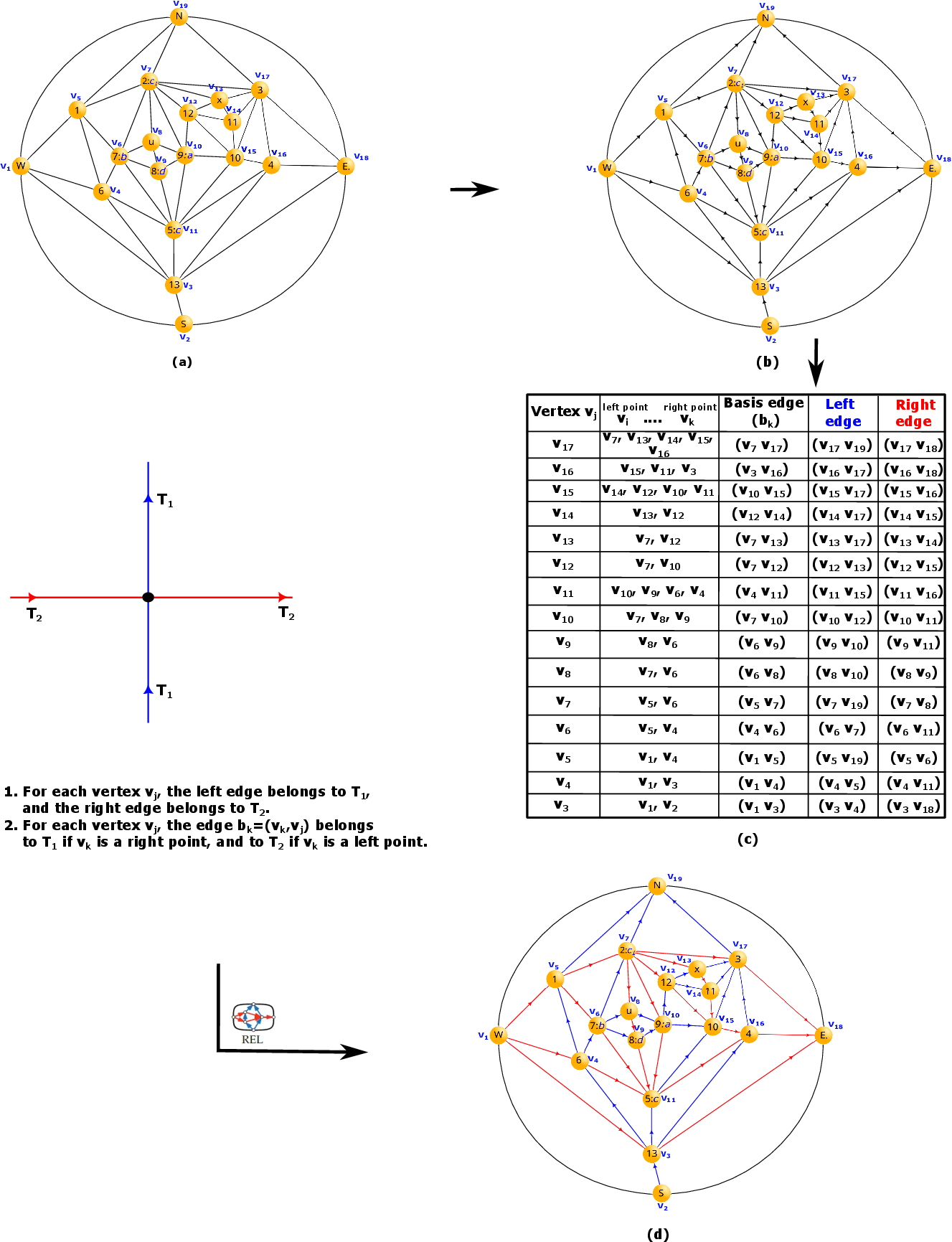}
    \caption{REL construction of the canonical ordered graph $G^2_L(V^2, E^2)$.}
   \label{RELL1}
 \end{figure}

 \subsection{An Overview of Our Proposed Work for $T$- Shaped Module Generation}
 \begin{figure}
   \centering
    \includegraphics[width=0.8\textwidth]{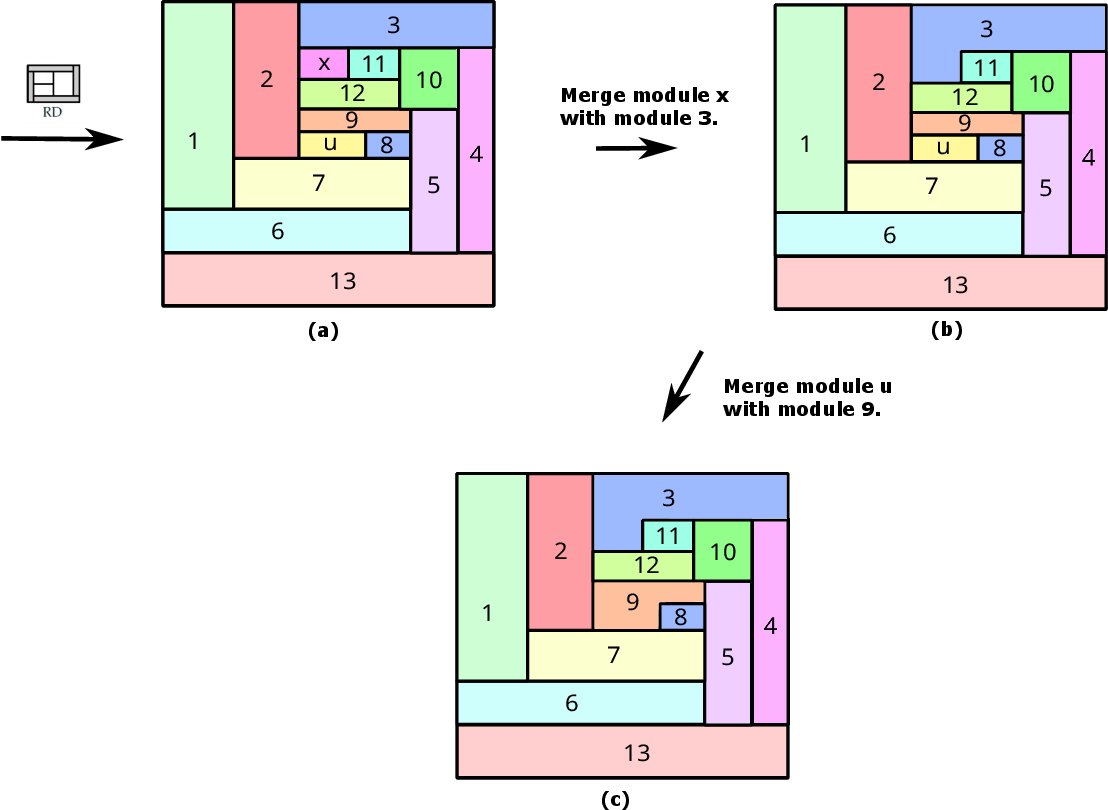}
    \caption{Getting coordinates for Rectangular dual corresponding to $G^3_L(V^3, E^3)$.}
   \label{REL-L2}
 \end{figure}

Different rectangular floor plans (RFPs) are constructed by altering the $K_T$ subgraph through subdividing its edges, following the approach outlined in \cite{kant1997regular}. Once the additional components (i.e., vertices) introduced by removing a complex triangle $K_T$ are combined, a variety of floor plans emerge, containing both simple (trivial: wall shrink to become $L$-shape module) and $T$-shaped modules (refer to Figure \ref{T-5}). This observation indicates that merely splitting any edge of a complex triangle does not guarantee the appearance of a  $T$-shaped module within the resulting floor plan $F_T$. Consequently, if a graph includes $K_T$ as a subgraph, further processing or techniques are necessary to produce a targeted $T$-shaped module in the final design.\\
Furthermore, we observe that by specifically subdividing a certain internal edge of $K_T$, it becomes feasible to generate a  $T$-shaped module within the floor plan $F_T$. Our goal is to design an algorithm (i.e.,  $Algorithm$ \ref{TLabel}) that produces a $T$-shaped module for any triangulated plane graph $G_T$ that contains at least one interior subgraph $K_T$.

\begin{figure}
   \centering
    \includegraphics[width=0.90\textwidth]{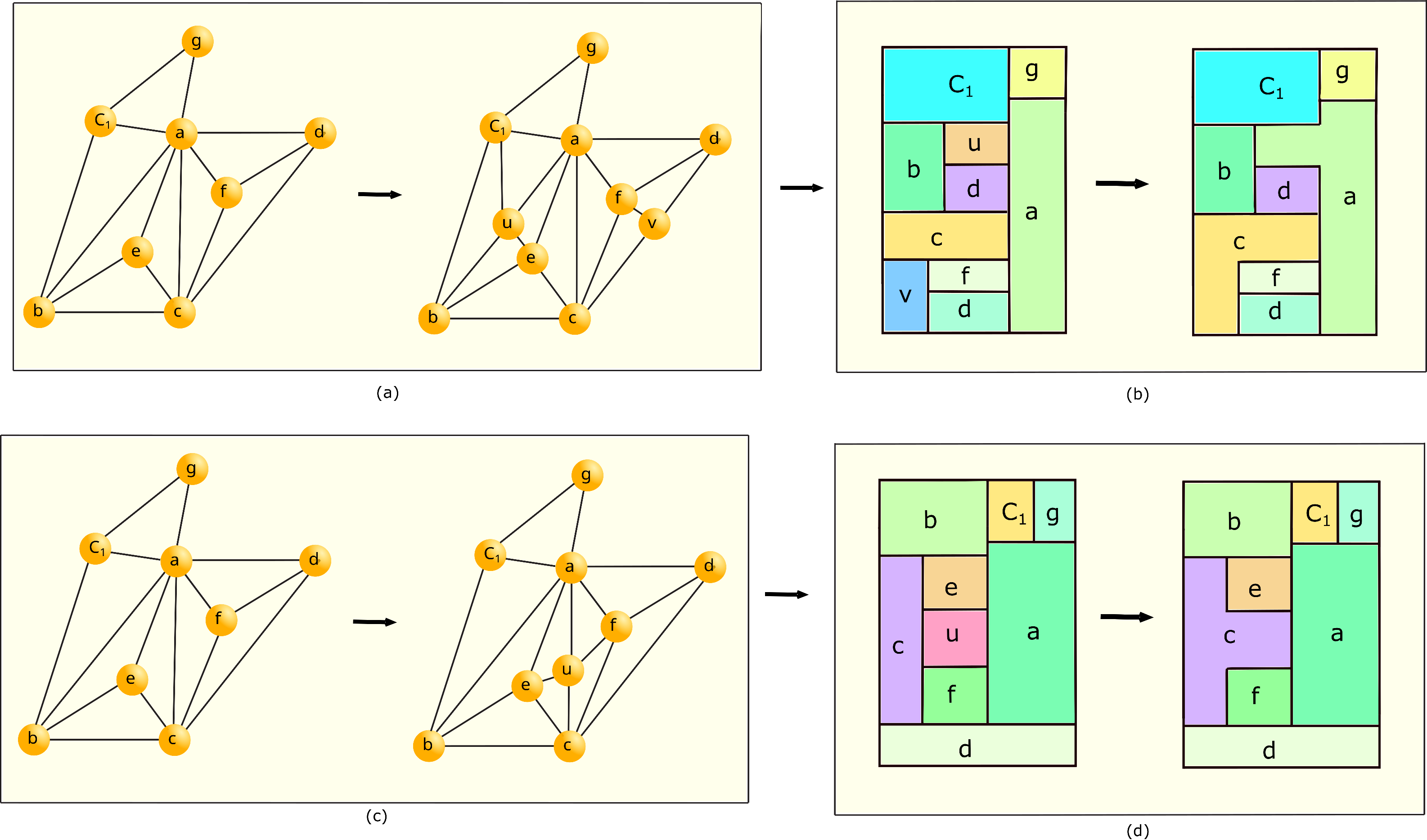}
    \caption{(a) Inserting additional vertices $u$ and $v$ into the input graph $K_T$ for complex triangle removal. (b) A trivial $T$-shaped module and a $L$-shaped module are generated in the resulting floor plan. (c) Inserting an additional vertex $u$ into $K_T$. (d) A $T$-shaped module is generated in the resulting floor plan.}
   \label{T-5}
 \end{figure}
   \begin{figure}
   \centering
    \includegraphics[width=0.65\textwidth]{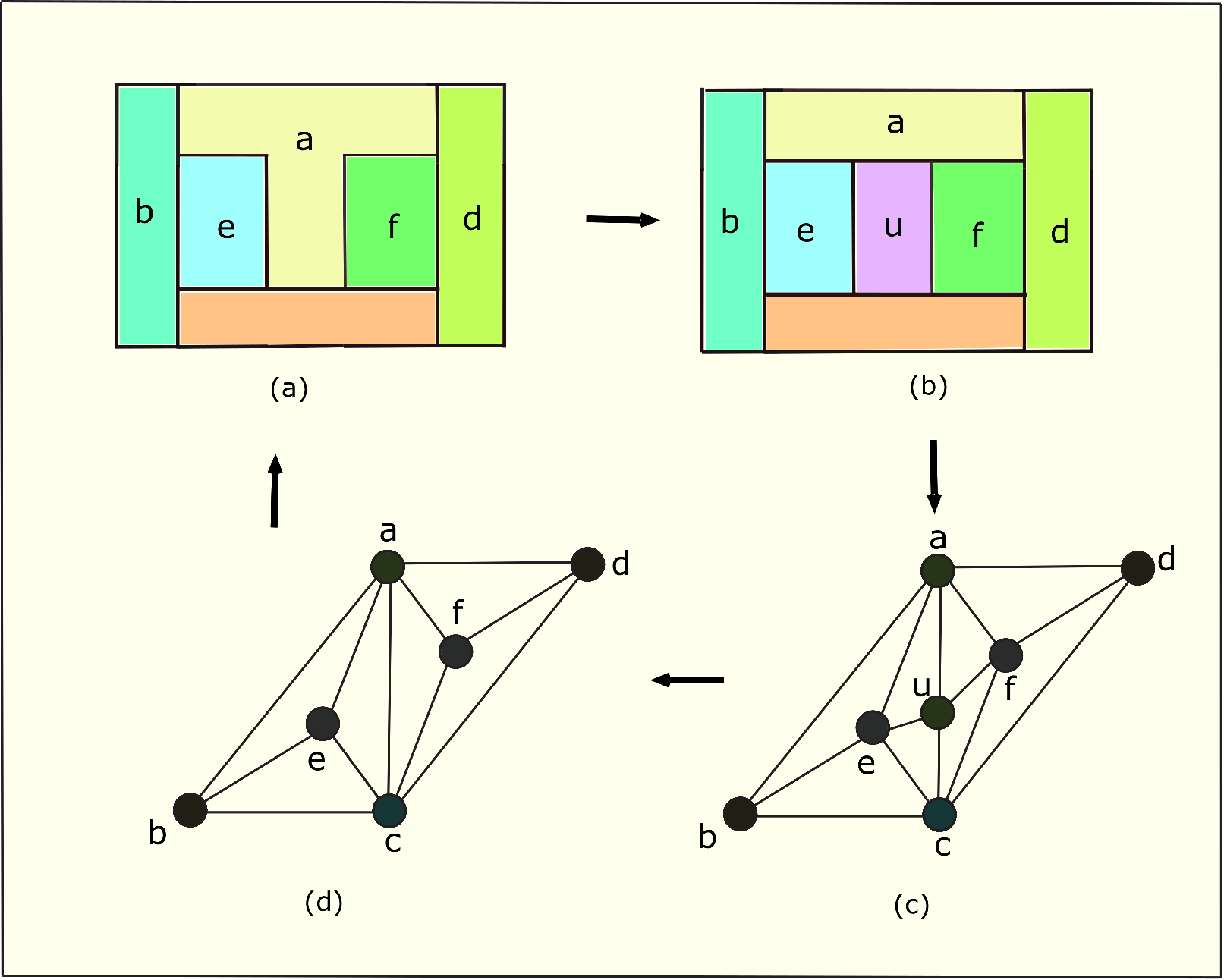}
    \caption{ (a-d) Requirement of a subgraph $K_T$ for the  generation of a $T$-shape module within floor plan.}
   \label{T-3}
 \end{figure}
 \begin{figure}
   \centering
    \includegraphics[width=0.80\textwidth]{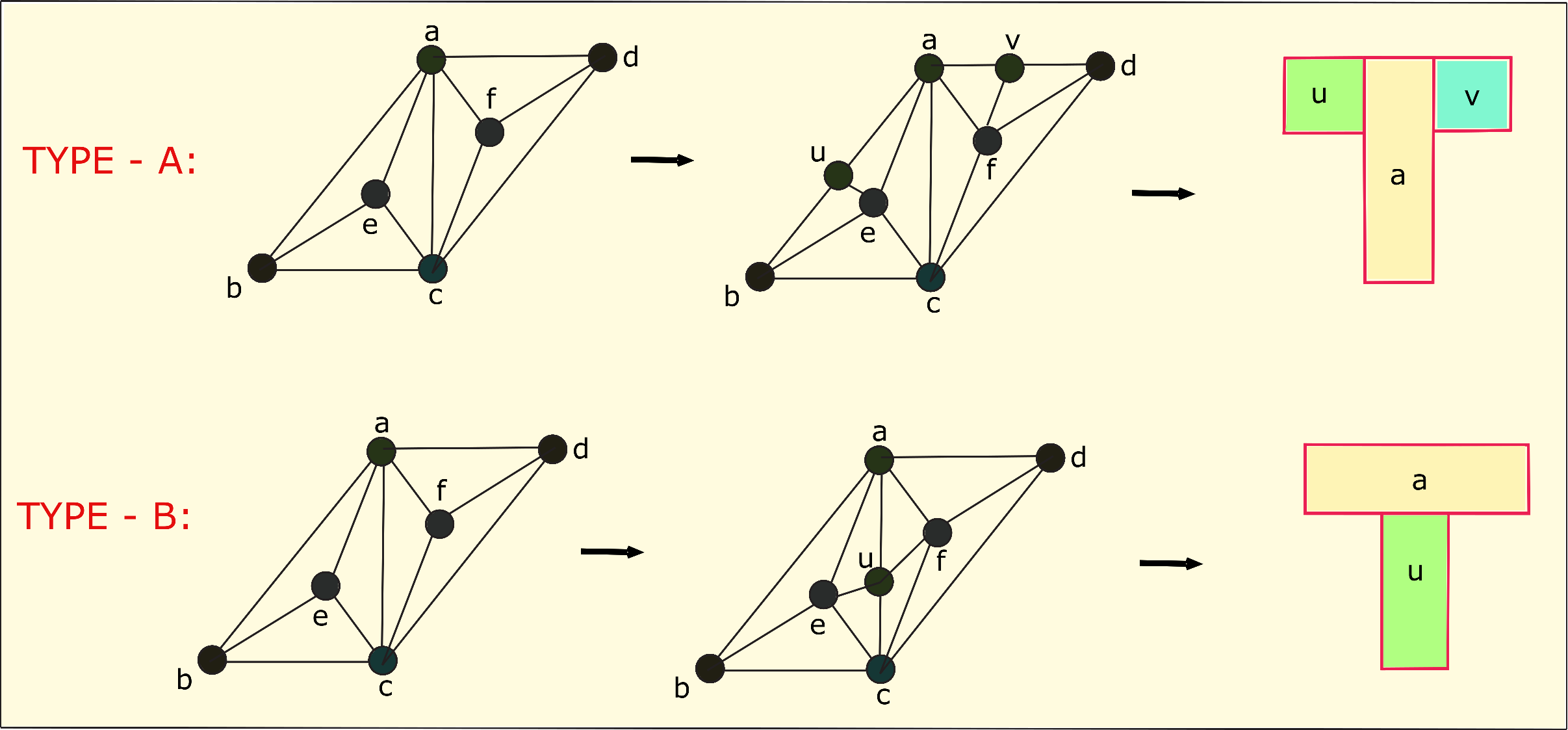}
    \caption{$T$-shape module generation through two different ways (Type A and Type-B).}
   \label{T-1}
 \end{figure}
 \begin{figure}
   \centering
    \includegraphics[width=0.80\textwidth]{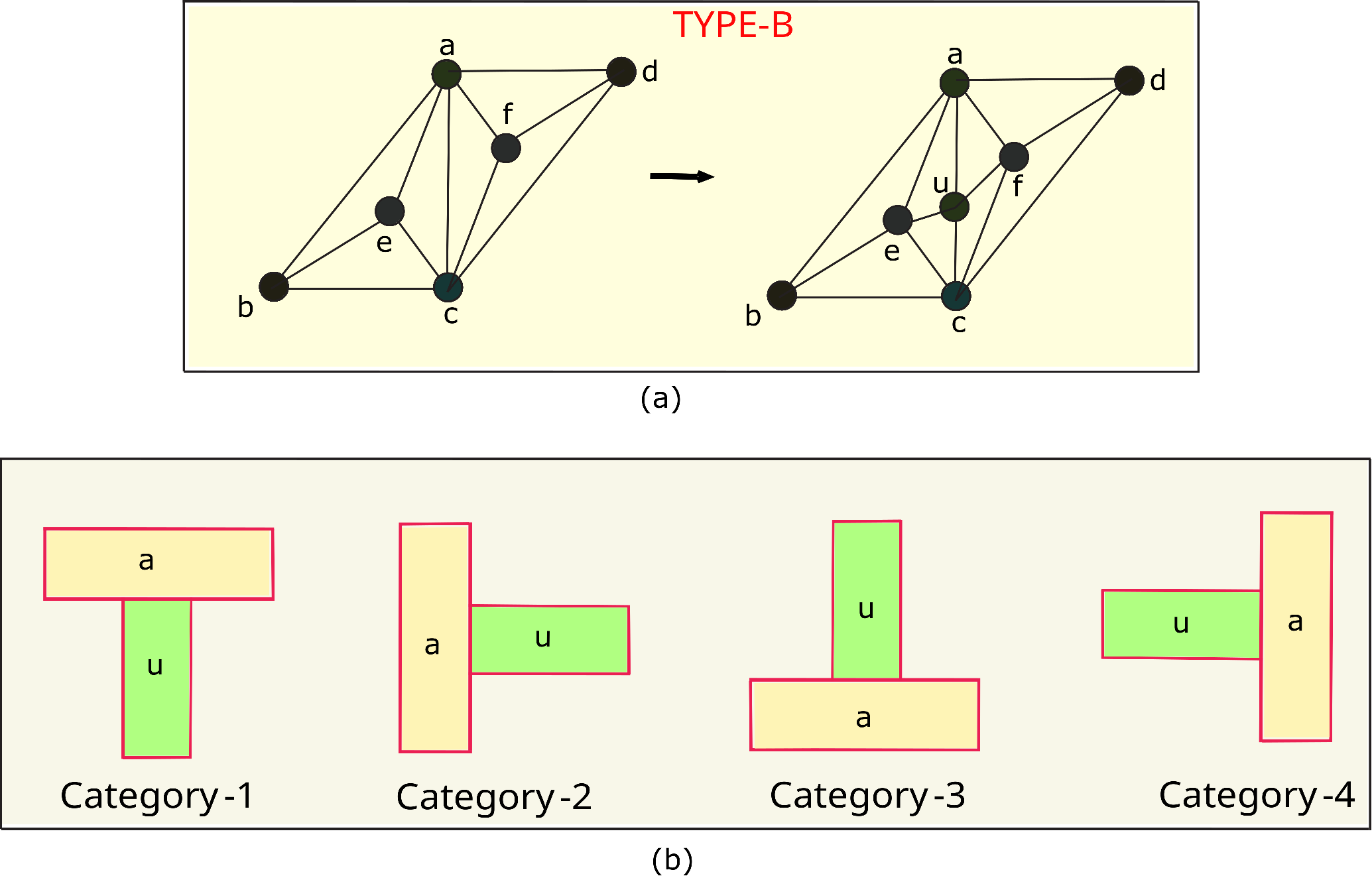}
    \caption{ (a-b) Several ways for merging $u$ with $a$ (Category: 1-4).}
   \label{T-2}
 \end{figure}
 \begin{figure}
   \centering
\includegraphics[width=1.0\textwidth]{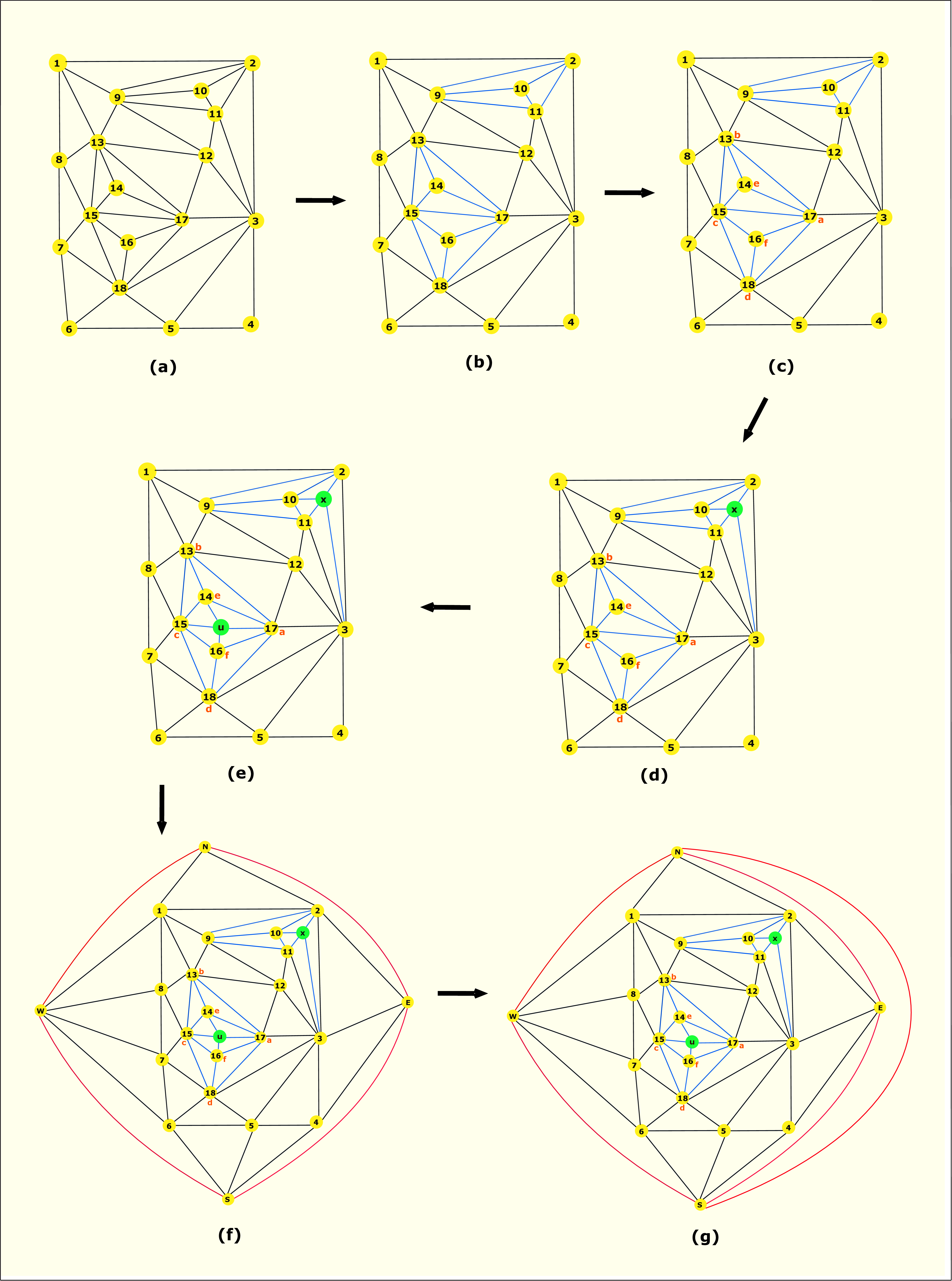}
    \caption{ (a) Input graph $G_T$. (b-c) Identifying the separating triangle (i.e., subgraph $K_T$ and others) and labeling. (d-e) Breaking the separating triangles by introducing new nodes and labeling the updated graph as $G^1_T$. (f-g) Applying Four-Completion and adding an extra edge in $G^1_L$ to construct $4$-connected triangulated graph. }
   \label{T-6}
 \end{figure}
  \begin{figure}
   \centering
\includegraphics[width=1.05\textwidth]{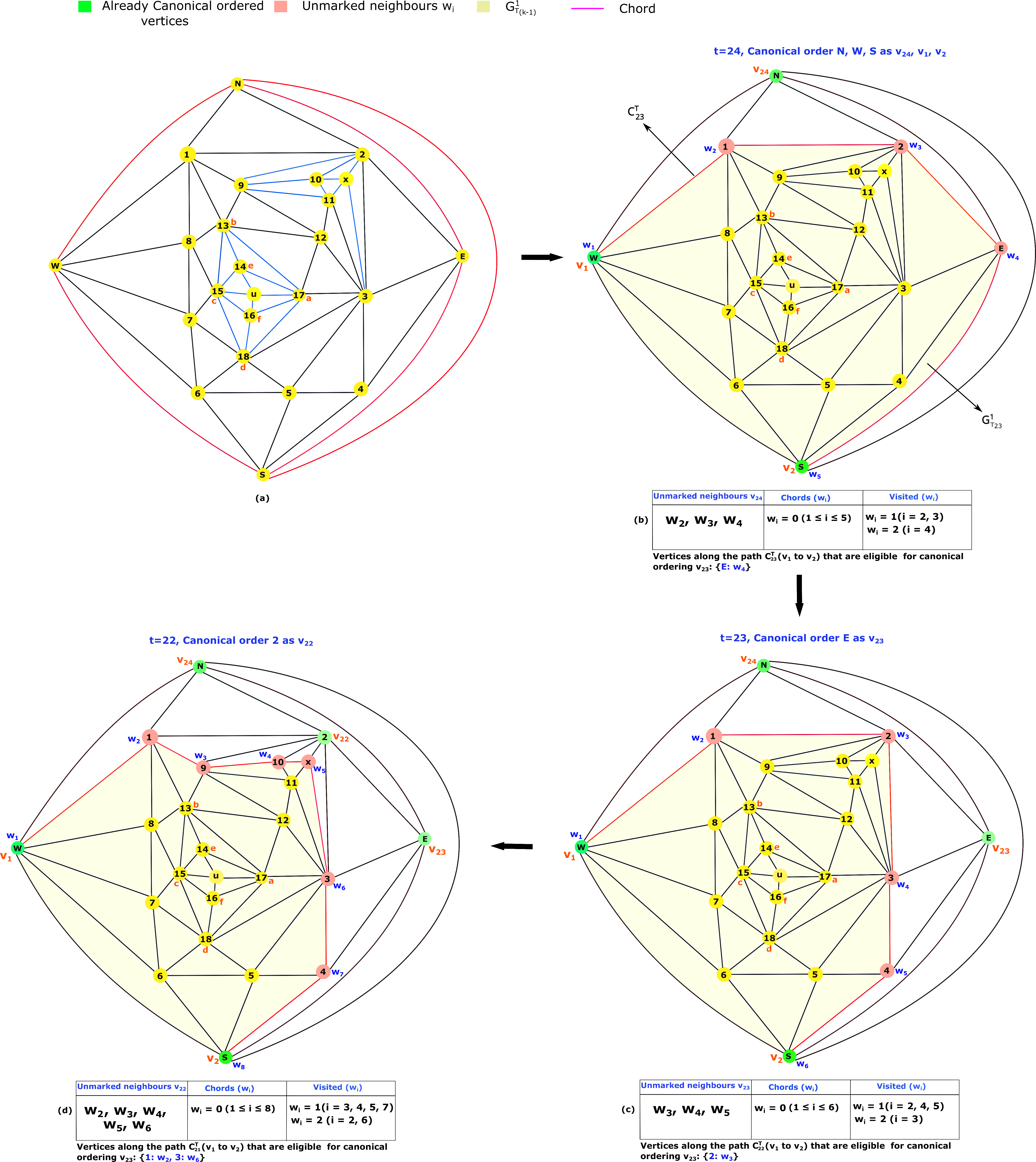}
    \caption{ (a-d) Generating the canonical order graph $G^2_T$ for the input graph $G^1_T$.}
   \label{T-7}
 \end{figure}
   \begin{figure}
   \centering
\includegraphics[width=1.05\textwidth]{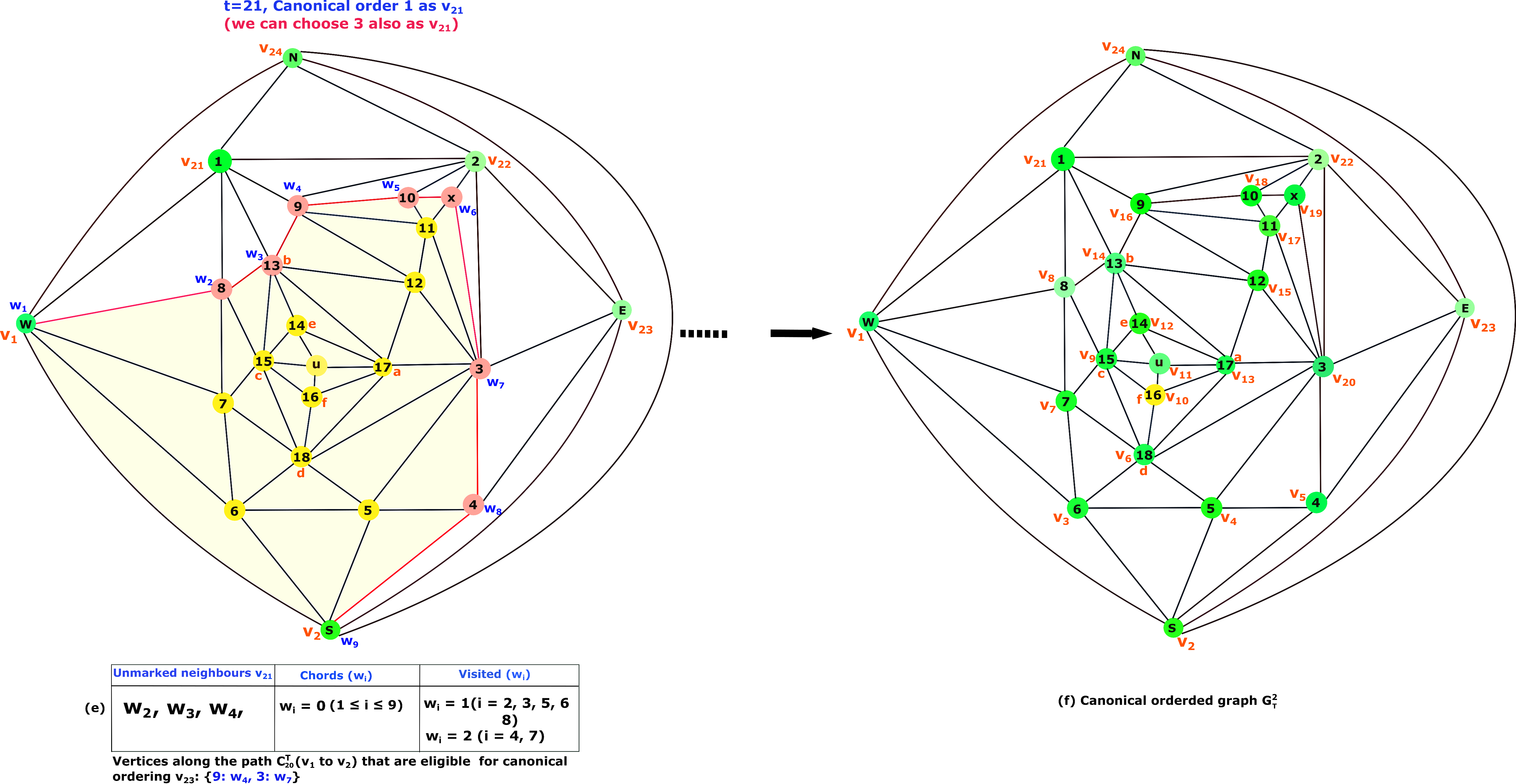}
    \caption{ (e-f) Generating the canonical order graph $G^2_T$ for the input graph $G^1_T$.}
   \label{T-8}
 \end{figure}
   \begin{figure}
   \centering
\includegraphics[width=1.05\textwidth]{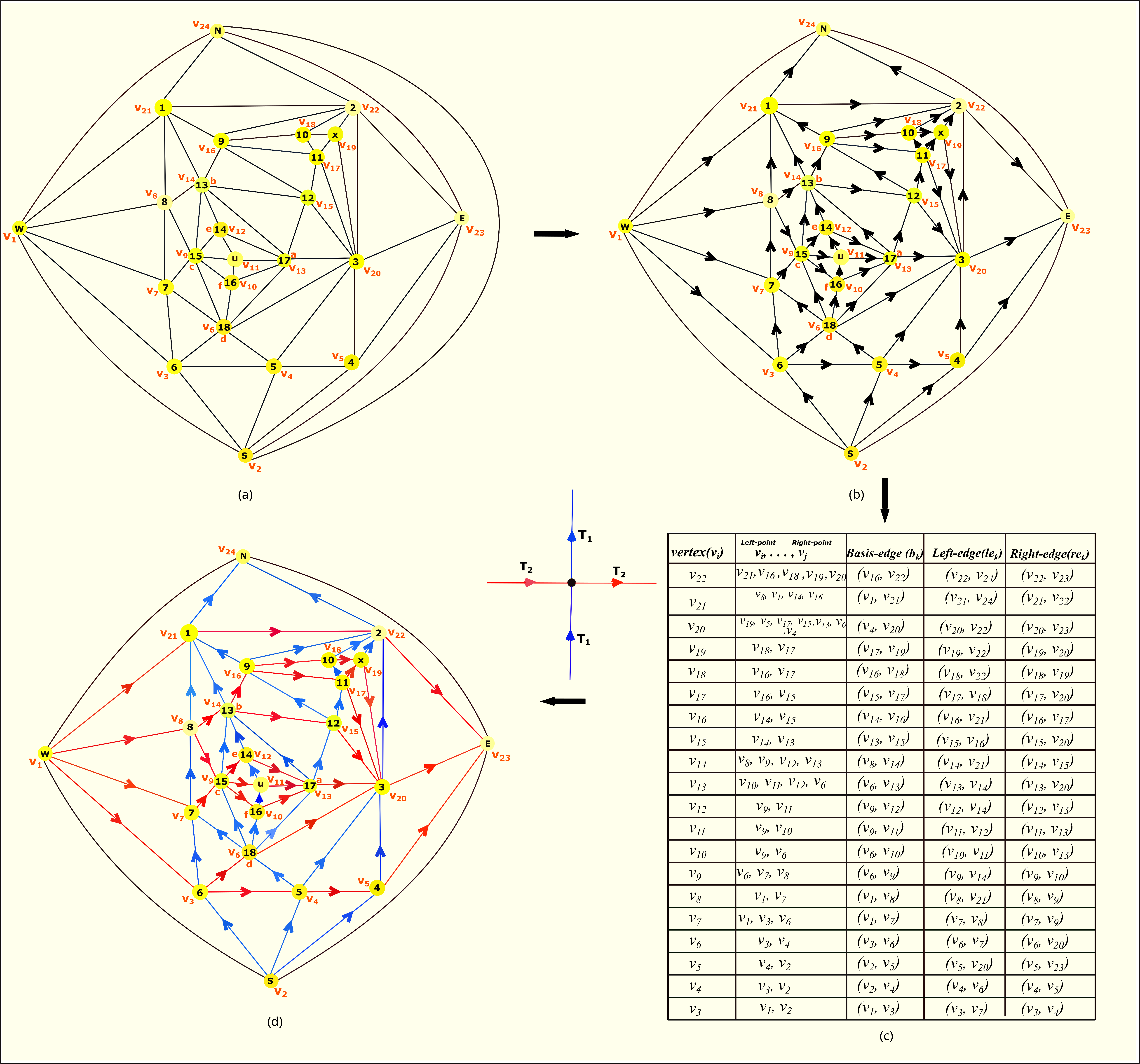}
    \caption{ (a) A possible canonical ordering represented as $G^2_T$. (b) A directed graph is constructed from this canonical ordering. (c) A listing of basis, left, and right edges associated with each vertex. (d) A regular edge labeling graph $G^3_T$ is derived from this canonical ordering.}
   \label{T-9}
 \end{figure}
 \begin{figure}
   \centering
\includegraphics[width=0.9\textwidth]{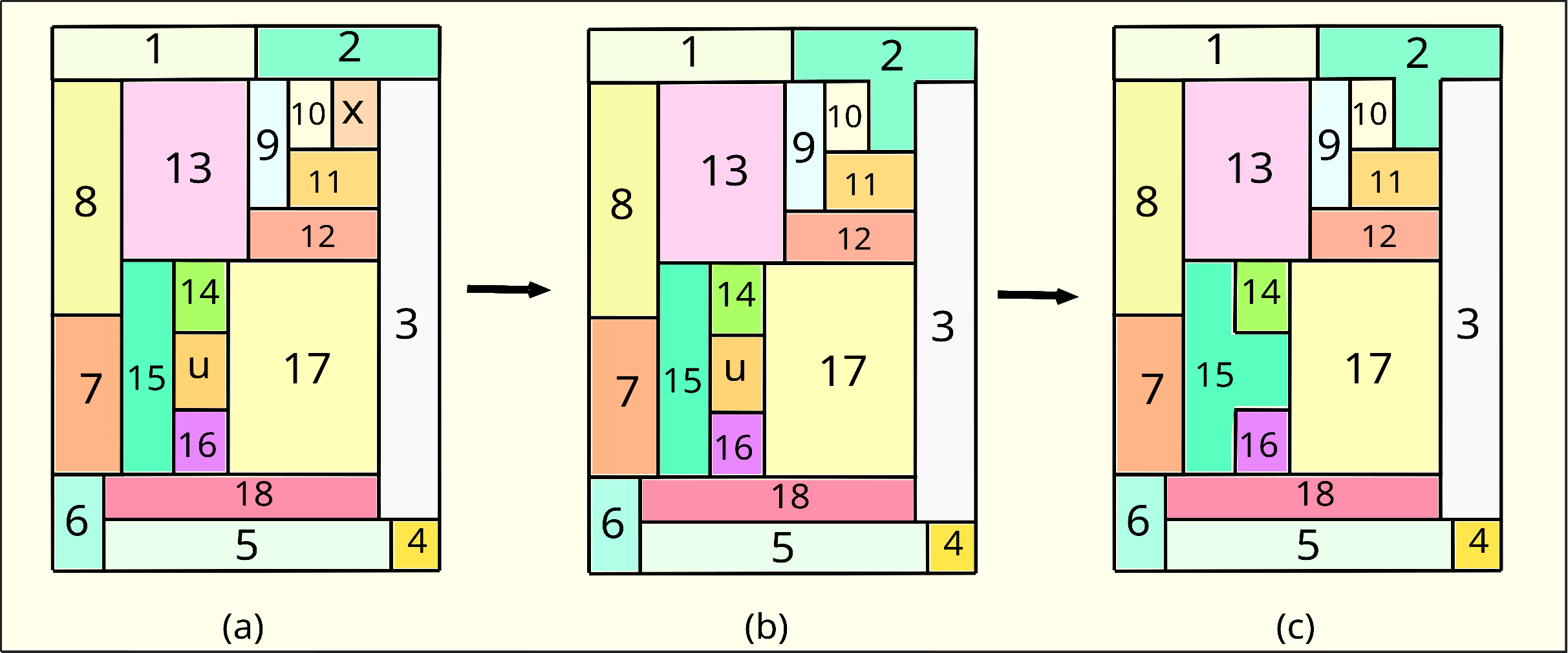}
    \caption{ (a) A floor plan $F'_T$ of $G^2_T$, constructed using the regular edge labeling graph $G^3_T$; (b–c) A final floor plan $F_T$ with a $T$-shaped module (module 15), obtained by merging modules (merging module $x$ with module $2$ and merging module $u$ with module 15) from the input graph $G_T$.}
   \label{T-10}
 \end{figure}
\subsection {Requirement of an $K_T$ within the Graph $G_T$} \label{5.5}
To design a floor plan ($F_T$) that incorporates a $T$-shaped module, it is essential to include a $K_T$ subgraph within $G^1_T$ (refer to Figure \ref{T-3}). Therefore, the presence of a $K_T$ subgraph in the input graph is necessary for successfully forming a $T$-shaped module within the floor plan. The outcome highlights the fundamental requirements for constructing a $T$-shaped module within a floor plan. Accordingly, the process begins by identifying a region $K_T$ in the input graph $G_T$, with its vertices ordered as $a$, $b$, $c$, $d$, $e$ and $f$ in a counter-clockwise order (see Figure \ref{T-5}).\\
To integrate a $T$-shaped module into the final floor plan, there are two types (see Figure \ref{T-1}). We are focusing on the generation of $T$-shape module of Type-B, whereas for this type, there are four further categories to generate $ T$-shaped modules in the final floor plan (see Figure \ref{T-2}). For each of the category (i.e., Category 1–4) of Type-B, it is possible to construct a $T$-shaped module within the floor plan. This study concentrates on building a $T$-shaped module within the floor plan, focusing on one type (i.e., Type-B)) out of two possible cases (see Figure \ref{T-1}). The presence of such a $T$-shaped module for any of the four categories (for Type-B) within a given graph is formally established in the correctness section. Therefore, our focus remains on forming a $T$-shaped module derived from the four defined categories, considering the subgraph $K_T$.
\subsection {$T$-shaped module generation within floor plan $F_T$}



 
This section illustrates our proposed Algorithm \ref{TLabel} using an example where we generate a $T$-shaped module within floor plan $F_T$ for the input graph $G_T$ with an internal subgraph isomorphic to $K_T$. \\\\
\textbf{1. Steps [1 to 4] of Algorithm \ref{TLabel}} : \textbf{Complex Triangle Identification and Removal (except subgraph $K_T$)} :\\
The method described by Roy et al. \cite{roy2001proof}  provides a way to identify and remove complex triangles from a graph. To eliminate all the complex triangles within a graph, one must first identify a subset of edges (denoted as $S_L$) such that every complex triangle contains at least one edge from this subset. Then each edge in $S_T$ is split by inserting new vertices into the graph $G_T$. This systematic modification ensures the removal of all the complex triangles while preserving the triangularity in the graph $G_T$.\\
Therefore, we apply the Complex Triangle identification and the Removal algorithm as described in \cite{roy2001proof} to first identify all complex triangles (see Figure \ref{T-6}(a-b)) and then remove all the complex triangles from the input triangulated graph $G_T$ (if their exist), leaving only $K_T$ subgraph (see Figure \ref{T-6}(c-d)). If any complex triangle other than subgraph $K_T$ exists, we split it by adding a new vertex $r$ and re-triangulate the input graph $G_T$. This produces an updated $G_T$ that contains no complex triangles except the subgraph $K_T$ (see Figure \ref{T-6}(a-d)). \\\\
\textbf{2. Steps [5 to 6] of Algorithm \ref{TLabel}} : \textbf{Removal of Complex Triangle $K_T$ and Four Completion Phase}:\\
To modify the remaining subgraph $K_T$, we proceed by choosing the edge $(a,c)$ and eliminating it by introducing a new vertex $u$. Subsequently, we insert new edges (i.e., $\{(u,e),(u,f),(u,a),(u,c)\}$) to maintain the triangulated structure of the graph. The resulting graph is denoted as $G^1_T$ (refer to Figure \ref{T-6}e). As a result of this transformation, $G^1_T$ no longer contains any complex triangles. \\
Once complex triangle elimination has been completed, the graph $G^1_T$ enters the four-completion process: Following the approach outlined in \cite{kant1997regular}, four paths {$P_4$, $P_3$, $P_2$, $P_1$} are first identified in $G^1_T$, after which directional vertices ($E$, $W$, $S$, $N$) are inserted into their respective paths. These inserted vertices correspond to the rectangular boundary modules that define the floor plan boundary. This procedure constitutes the four-completion phase, as illustrated in Figure \ref{T-6}(e-f). After completing the four-completion phase, the edge $(N, S)$ is introduced into $G^1_T$, resulting in a 4-connected triangulated graph ($G^1_T$), as shown in Figure \ref{T-6}(f-g). The graph $G^1_T$ then moves forward to the canonical ordering step.\\\\
 \textbf{3. Steps 7 to 17 (Algorithm \ref{TLabel}): Canonical ordering:}\\
 This section explains the algorithm (Algorithm \cite{kant1997regular}) for producing a canonical ordering ($G^2_T$), as defined in Definition 4 from a given 4-connected triangulated graph $G^1_T$. The procedure utilises Steps 7–17 from Algorithm \ref{TLabel}, illustrated with an example in Figures \ref{T-7}, \ref{T-8}. This canonical ordering is essential for constructing a $T$-shaped module within the floor plan $F_T$, based on the initial PTG $G_T$ that contains at least one interior subgraph $K_T$.\\
See Figures \ref{T-7}, \ref{T-8}: Beginning with the previously obtained 4-connected triangulated graph $G^1_T$, Steps 7–9 are used to set the $chord$ value to $0$ and the $status$ value to $False$ for each vertex of $G^1_T$. Vertex $W$ is assigned (canonical ordered) as $v_1$ and vertex $S$ is assigned as $v_2$, and the visited value for vertex $E$ is set to $1$. Following this, Steps 10–17 are applied to assign canonical orders to all the remaining vertices of $G^1_T$ one by one using the approach described in \cite{kant1997regular}, in accordance with Definition 4 (see Terminology Section \ref{Preliminaries}). The result is a canonically ordered graph $G^2_T$, which establishes a canonical vertex ordering. This ordering ($G^2_T$) is then used to construct the regular edge labeling $G^3_T$. \\\\
\textbf{4. Step 18 of Algorithm \ref{TLabel}:  Generation of Regular Edge Labeling:}\\
To construct a Regular Edge Labeling (REL), we follow the method described by Kant \cite{kant1997regular} using the canonical ordering $G^2_T$. Refer to Figure \ref{T-9}(a-d), which illustrates the process applied to $G^2_T$, where we determine the basis edges ($b_k$), along with the corresponding sets $C_k$ and $R_k$ for each vertex. These components are then used to produce the regular edge-labeled graph $G^3_T$ (see Figure \ref{T-9}d).\\\\
\textbf{5. Step 19 of Algorithm \ref{TLabel}: Generation of a Rectangular Floor Plan:}\\
Once we have generated the Regular Edge Labeling ($G^3_T$) using the canonical ordering concept, we use this REL to construct a rectangular floor plan ($F'_T$) for the generated graph $G^2_T$. This is done by following the method described by Bhasker and Sahni \cite{bhasker1988linear} (see Figure \ref{T-10}a).\\\\
\textbf{6. Steps [20 to 24] of Algorithm \ref{TLabel}: Generation of $T$-shaped Module within Orthogonal Floor plan by Merging Modules:}\\
The $Merge$ $Rooms$ function describes the method for integrating rectangular modules that correspond to extra vertices introduced during the removal of complex triangles in $G_t$. It requires three inputs: the rectangular floor plan $F'_T$ derived from the graph $G^2_T$, the canonical ordered graph $G^2_T$ itself and a set of extra nodes $Enodes_T$. The outcome is a $T$-shaped module in the floor plan that aligns with the structure of $K_T$.\\
See Figure \ref{T-10}(b,c) (here $S_T$ = $(2,11)$ and $Enodes_T$ is $x$), where we obtained the floor plan $F_T$ with a $T$-shaped module from a rectangular floor plan $F'_T$ while using the $function$ $Merge$ $Rooms$.\\\\
\textbf{Hence, given a $G_T$ (triangulated graph) with a internal subgraph $K_T$, a $T$-shaped module can be constructed within the corresponding floor plan $F_T$ using our proposed Algorithm \ref{TLabel}}. 
\section{Correctness of Algorithms}\label{correctness}
\begin{figure}
   \centering \includegraphics[width=0.95\textwidth]{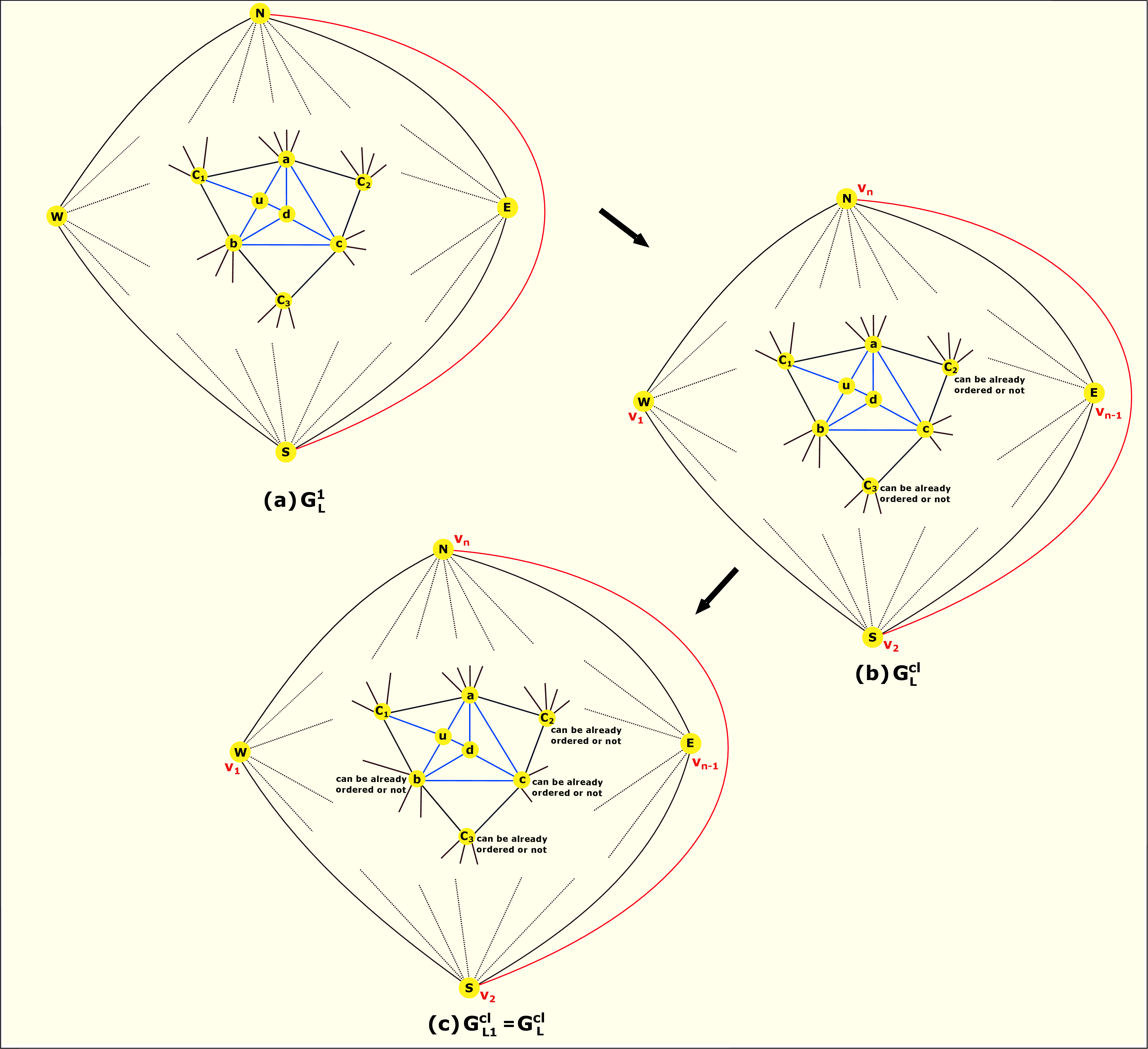}
     \caption{ (a) A 4-connected triangulated graph $G_L^1$. (b) $G^{cl}_L$ is generated from canonical ordering of all the vertices of $G^1_L$, except $\{a, b, c, d, u, C_1\}$, using steps (7--28) of Algorithm \ref{L-shaped}. (c) $G^{cl}_{L1}$ is generated by calling function $Types$ $of$ $Priority$ $label_L$ of Algorithm \ref{L-shaped}.}
   \label{CL-1}
  \end{figure}
  \begin{figure}
   \centering \includegraphics[width=0.95\textwidth]{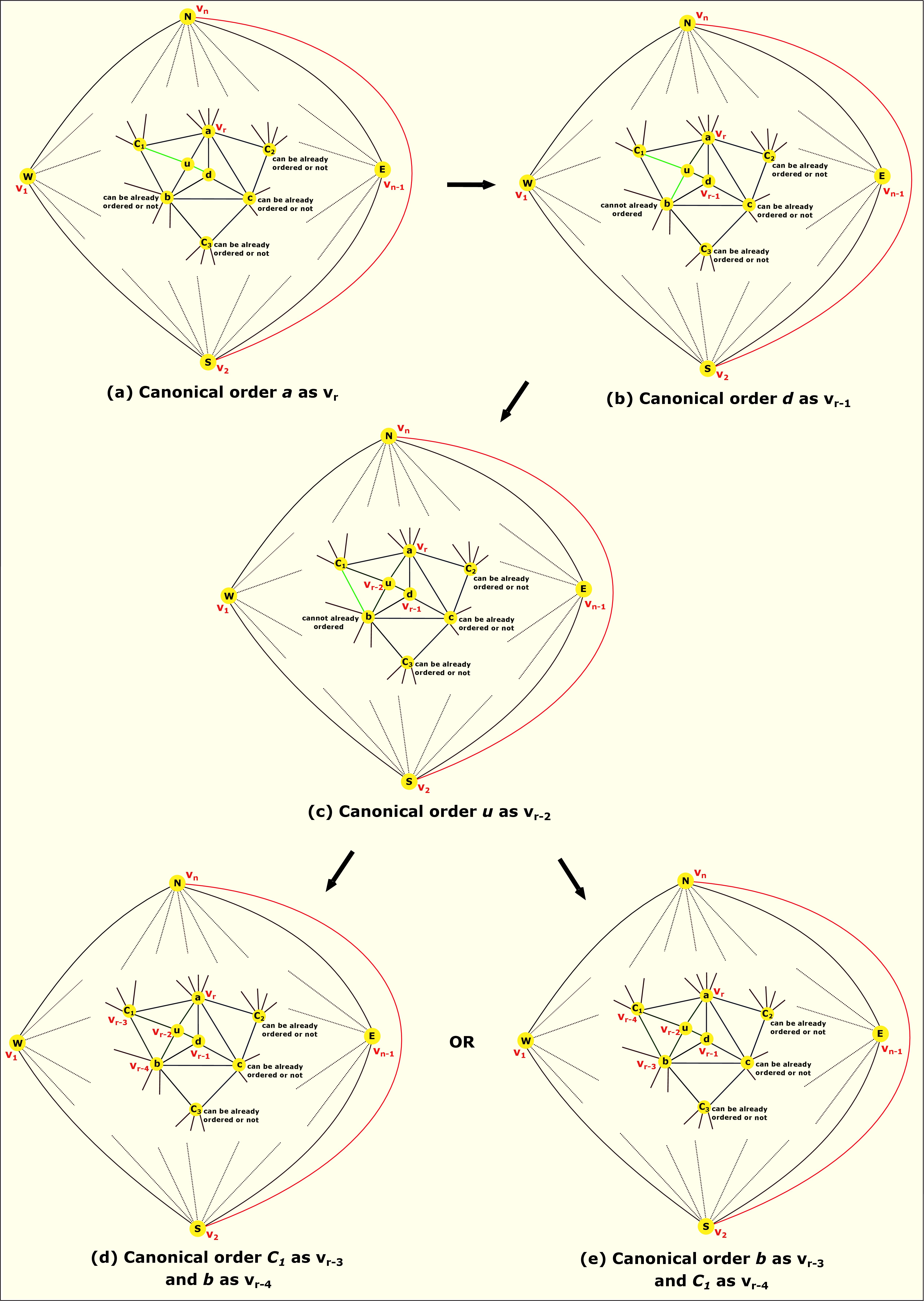}
     \caption{ (a-e) Existence of a Category-D (Canonical ordering) in $G^1_L$.}
   \label{CL-2}
  \end{figure}
   \begin{figure}
   \centering \includegraphics[width=0.95\textwidth]{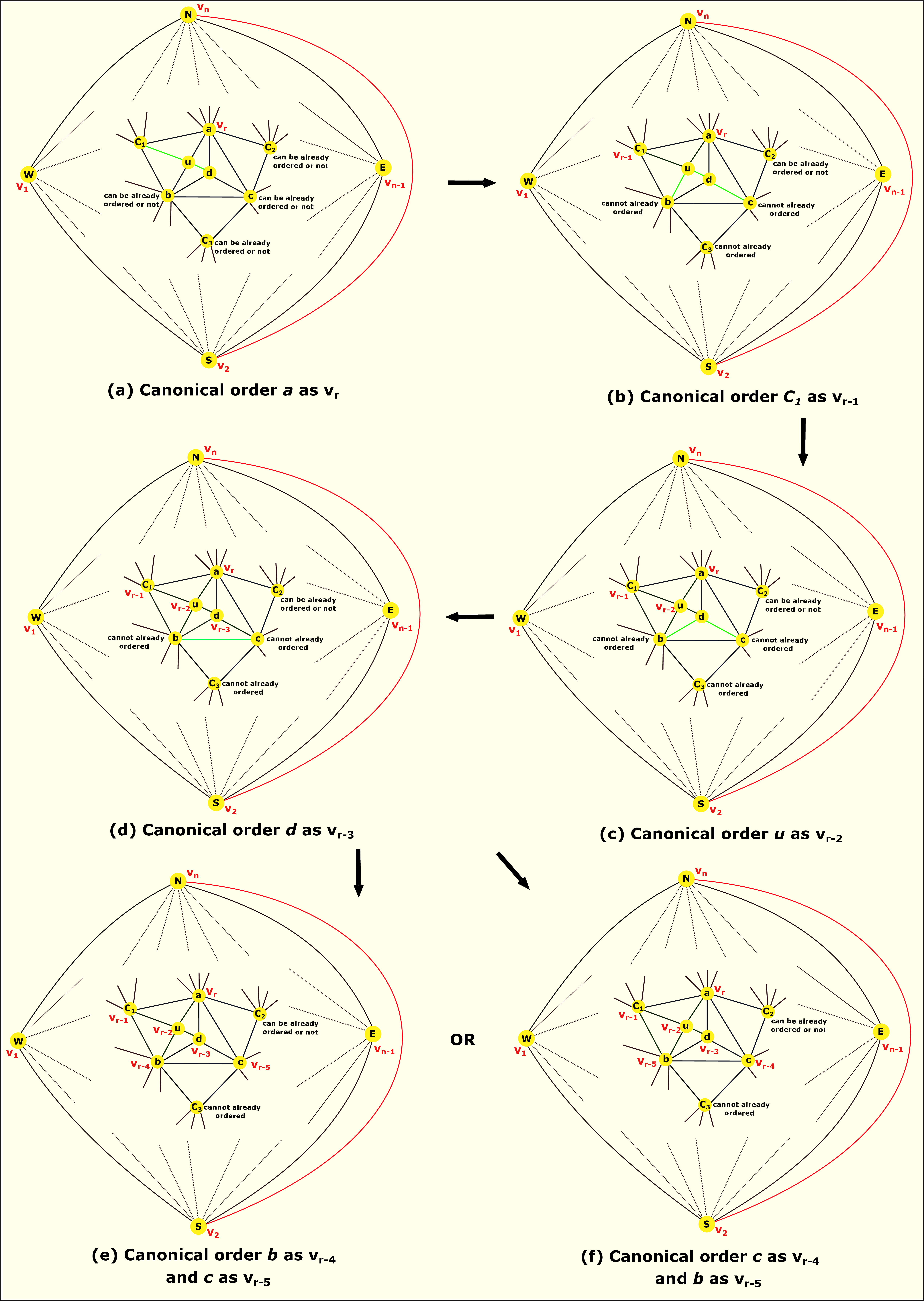}
     \caption{ (a-f) Existence of a Category-F (Canonical ordering) in $G^1_L$.}
   \label{CL-3}
  \end{figure}
   \begin{figure}
   \centering  \includegraphics[width=0.95\textwidth]{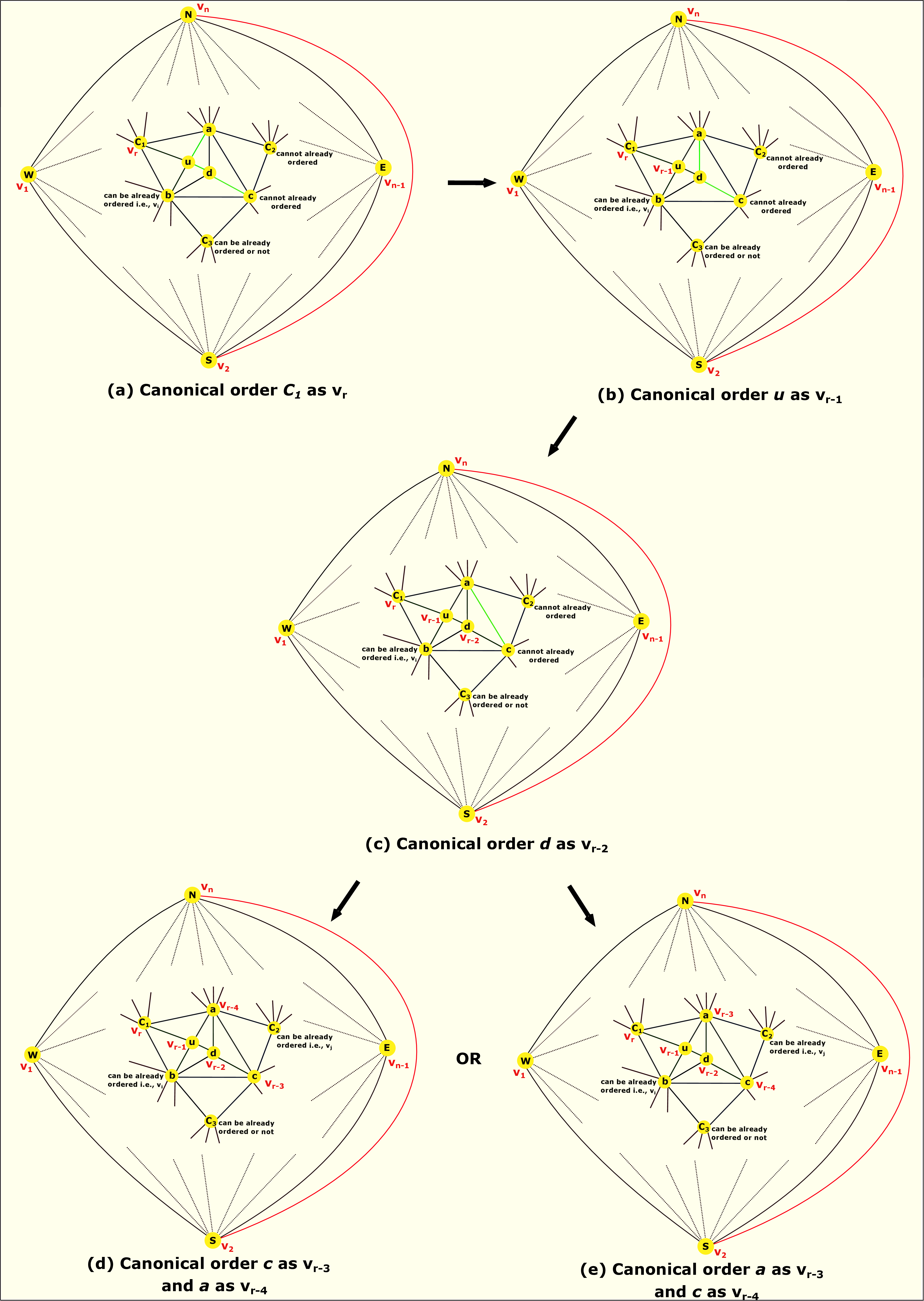}
     \caption{ (a-e) Existence of a Category-A (Canonical ordering) in $G^1_L$.}
   \label{CL-4}
  \end{figure}
     \begin{figure}
   \centering  \includegraphics[width=0.95\textwidth]{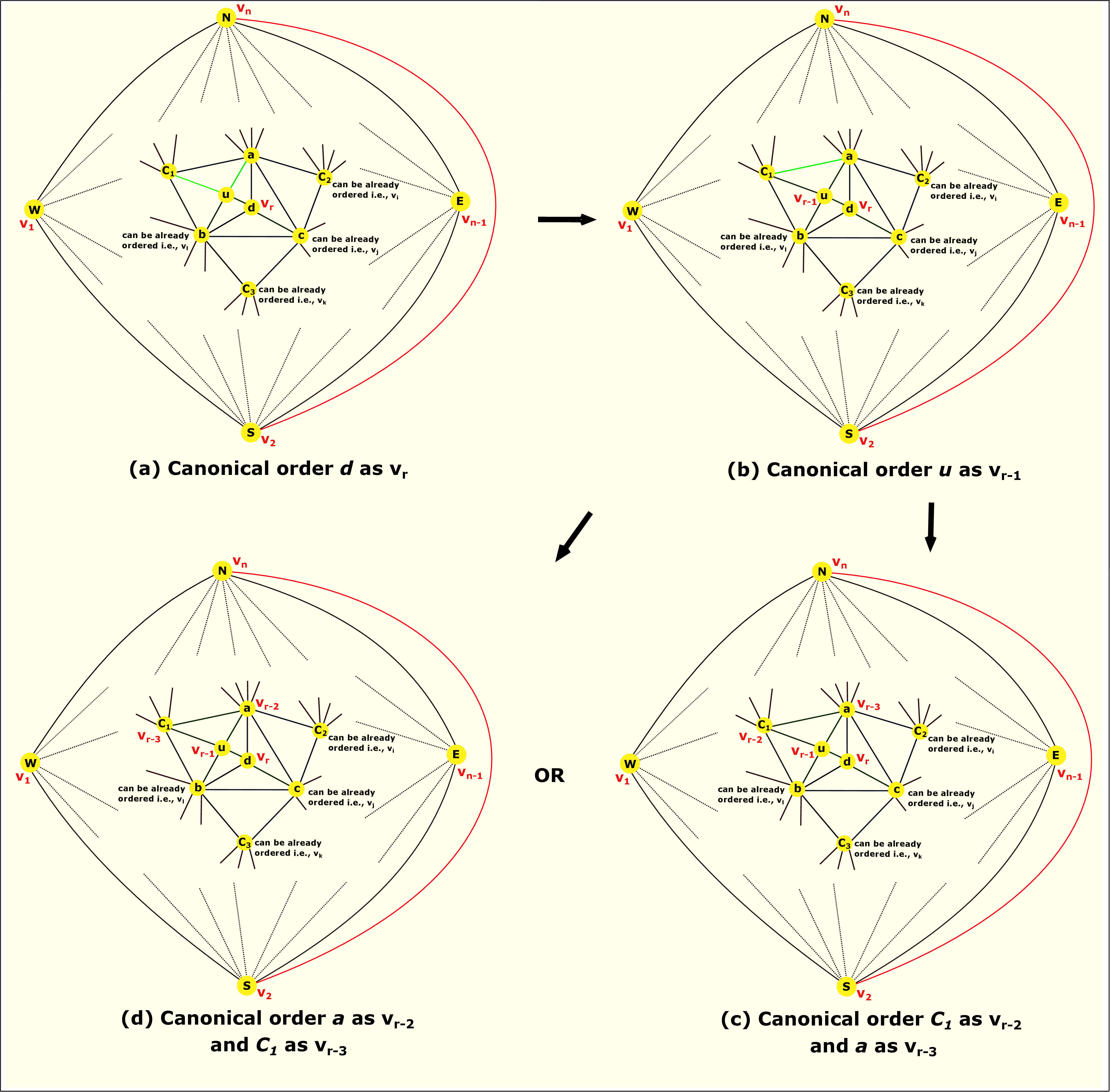}
     \caption{(a-d) Existence of a Category-B or Category-C (Canonical ordering) in $G^1_L$.}
   \label{CL-5}
  \end{figure}

\begin{theorem}
    For a given $G_L(V, E)$  that includes at least one internal subgraph isomorphic to $K_L$, Algorithm \ref{L-shaped} yields an orthogonal floor plan $F_L$ that necessarily contains a L‑shaped module corresponding to the subgraph $K_L$. 
\end{theorem}
\begin{proof}
Consider a triangulated graph $G_L$ that contains an internal complex triangle isomorphic to $K_L$. For every four-connected triangulated graph, a canonical ordering can be determined, as demonstrated by Kant \cite{kant1997regular}. From this ordering, a regular edge labeling (REL) can subsequently be obtained. This REL serves as the basis for creating a rectangular floor plan, following the method described by Kant \cite{kant1997regular}.\\
Our objective is to construct a $L$-shaped module within the floor plan corresponding to an input graph $G_L$. To accomplish this, we need an RFP that includes an additional module, called $u$. Creating a module of the $L$ shape requires merging module $u$ with module $a$ or module $b$, following a specified priority order and REL. To this end, we begin with a graph $G_L$ that contains an interior $K_L$ subgraph. We then modify $G_L$ by introducing a new vertex $u$, thereby subdividing the outer edges of $K_L$. This modification results in a 4-connected plane triangulated graph, denoted as $G^1_L$, which is a critical prerequisite for establishing a canonical ordering. The transformation process further involves the elimination of any remaining complex triangles, the application of 4-completion, and the addition of an auxiliary edge $(N, S)$, all of which are necessary to satisfy the requirements for canonical ordering.
Next, when assigning a prioritize canonical ordering to the vertices of $G^1_L$, we must prioritize the ordering of vertices \{$a,d,u, C_1$\} based on a defined specific order. However, the remaining vertices of $G^1_L$ (except for the \{$a,d,u, C_1$\} vertices) do not follow any specific priority order. The priority-based ordering of vertices \{$a,d,u, C_1$\} can fall into one of six defined categories (Category A to F): Refer to Figures \ref{L-5}, \ref{L-6} and \ref{L-7} and for illustration see subsection \ref{Category}.\\
We aim to show that a $L$-shaped module, representing the interior part $K_L$ of $G_L$, can always be constructed inside the floor plan $F_L$ of the graph $G_L$ by following our proposed Algorithm \ref{L-shaped}. To prove this, we will go through each step of Algorithm \ref{L-shaped} one by one, explaining how it works and confirming that it successfully builds the desired module.

\begin{enumerate}
    \item \textbf{Steps 1 to 4 (Algorithm \ref{L-shaped}):}  Given an input graph $G_L$ that contains at least one $K_L$ as a subgraph, we first check the existence of any complex triangles other than $K_L$ using the proposed algorithm in \cite{roy2001proof} (see Figure \ref{L-9}). If such complex triangles exist in $G_L$ we apply the Complex Triangle Removal algorithm proposed in  \cite{roy2001proof}, which introduces additional vertices and edges to eliminate them (see Figure \ref{L-11}c). This process ensures that all complex triangles (except $K_L$) are removed in $G_L$. As a result, we obtain a modified graph $G_L$ that is free of complex triangles other than $K_L$.\\
    \item \textbf{Steps 5 to 6 (Algorithm \ref{L-shaped}):} In order to modify the remaining $K_L$ subgraph, we first replace the internal edge $(a, b)$ of $K_L$ by inserting a new vertex $u$. To ensure that the updated graph remains triangulated, additional edges are introduced, as depicted in Figure \ref{L-11}d. We denote this modified graph as $G^1_L$. Next, we apply the Four-Completion algorithm, as described in \cite{kant1997regular}, to $G^1_L$ (see Figure \ref{L-12}b). At the final stage, we add an edge between vertex $N$ and $S$  to form $G^1_L$ (4-connected graph). Thus, a 4-connected graph $G^1_L$ can always be constructed from an input graph $G^1_L$ (see Figure \ref{L-12}c). \\
     \item \textbf{Steps 7 to 53 (Algorithm \ref{L-shaped}):} Starting from the 4-connected triangulated graph $G^1_L$ obtained earlier, we proceed with \textbf{Steps (7–28)} to assign canonical order to all possible vertices in $G^1_L$, excluding the set \{$a, C_1, d, u$\}, while following the process explain in Definition 4 (refer to the Terminology Section \ref{Preliminaries}). When we reach a stage where the next possible vertex for the order is only from the set \{$a, C_1, d, u$\}, we return $G^2_L$ and move to \textbf{Steps (29–53)}. At this point, we try systematically to order the vertices in the set \{$a, C_1, d, u$\}  according to Categories (A–F), checking each Category in order. If valid ordering exists under any of these Categories, we apply it and move forward by generating the canonical ordered graph $G^2_L$, which will be used for regular edge labeling construction: see Figures \ref{L1} to \ref{L7}.\\
     Therefore, to confirm the existence of a canonical ordered graph $G^2_L$, it suffices to demonstrate that there is always a valid, Category-based ordering for the set \{$a, C_1, d, u$\} within the generated graph $G_L^\text{cl}$. 
     \textbf{According to Lemma \ref{lemma-1}, there always exists a canonical priority ordering of the set \{$a, C_1, d, u$\}, which corresponds to any of the six possible Categories: Category A-F}.\\
     
    \item   \textbf{Step 54 (Algorithm \ref{L-shaped}):} Using the canonical ordered graph $G^2_L$ generated earlier, we now proceed with Algorithm \ref{RELF} to construct a Regular Edge labeling (REL, i.e., $T_1$ or $T_2$) of $G^2_L$, i.e., $G^3_L$. To construct a $L$-shape module within floor plan $F_L$, we have to merge the extra module $u$ with either $a$ or $b$ so that either module $a$ becomes $L$-shaped or module $b$ becomes $L$-shaped and to form module $a$ as $L$-shaped: the direction (that is, $T_1$ or $T_2$) of vertex $a$ with $C_1$ and $u$ must be opposite. Likewise, to form module $b$ as the $L$-shape: the direction (that is, $T_1$ or $T_2$) of the vertex $b$ with $C_1$ and $u$ must be opposite (see Figures \ref{L-5}, \ref{L-6}: the direction (that is, $T_1$ or $T_2$) of vertex $a$ with $C_1$ and $u$ is opposite and see Figure \ref{L-7}: the direction of vertex $b$ with $C_1$ and $u$ is opposite).\\
    \begin{enumerate}
   \item[(a.)] \textbf{Steps (1–10) of Algorithm \ref{RELF}}: We proceed to generate a regular edge labeling for the input graph $G^2_L$ utilizing the methodology described in\cite{kant1997regular}.
   Based on the resulting edge types $T_1$ or $T_2$ for the edges connecting vertex $C_1$ to vertices $a$ and $b$, we then determine the next appropriate operation.\\
    \item [(b.)] \textbf{{Steps (11–23) of Algorithm \ref{RELF}}}: Given that the neighbors of vertex $u$ are {$a, b, d, C_1$} and those of vertex $d$ are {$a, b, c, u$}, it follows that the edges incident to $u$ and $d$ possess fixed orientations ($T_1$, $T_2$) (refer to the definition of REL in Section \ref{Preliminaries}) during the construction of the REL from the canonically ordered graph $G^2_L$. Now, there can be two possible cases that arise based on the above generated REL:\\
    \textbf{Case- A:} If module $a$ shares walls of differing types (either $T_1$ or $T_2$) with both $u$ and $C_1$, or if the same condition holds for module $b$, then it is possible to directly merge module $a$ with $u$ by returning $m = 1$, or module $b$ with $u$ by returning $m = 2$. This results in the formation of a $L$-shaped module (see Figures \ref{RELL1}d, \ref{REL-L2}).\\
    \textbf{Case-B:} See Figure \ref{CL-6}(a-h): If the condition outlined in Case A does not hold, it necessitates that the input graph must contain two distinct vertices, denoted by $x$ and $y$. These vertices are defined as $x \in \left( \text{nbd}(a) \cap \text{nbd}(C_1) \right) \setminus {u}$ and $y \in \left( \text{nbd}(b) \cap \text{nbd}(C_1) \right) \setminus {u}$, with the condition that $x \ne y$; otherwise, this would contradict the requirement that the subgraph $K_L$ must exist in $G_L$ (see Figure \ref{CL-6}(a–b)). This scenario leads to two distinct subcases:\\
    \textbf{Subcase A:} If the walls shared between module $a$ and $C_1$ (i.e., $T_1$ or $T_2$) are opposite to those shared between module $x$ and $C_1$, then the appropriate sequence of operations is to first perform a flip on edge $(C_1, x)$, followed by a flip on edge $(a, C_1)$. After these modifications, we return $m = 1$ (refer to Figure \ref{CL-6}(e–h)).\\
    \textbf{Subcase B:} If the walls shared by both $a$ and $x$ with $C_1$ are aligned (i.e., the same), then a single flip on edge $(a, C_1)$ suffices, after which we return $m = 1$ (refer to Figure \ref{CL-6}(f–h)).\\

    \end{enumerate}
 Since, based on defined priority ordering (Category: A-F), the direction of the vertex $a$ with respect to $C_1$ and $u$ must always be opposite or this must hold for the vertex $b$. Hence, we can choose to merge module $u$ with either $a$ or $b$ to form a $L$-shape by returning the corresponding $m$-value.
 \textbf{Therefore, a REL $G^3_L$ can always be generated from the canonical ordered graph $G^2_L$.}\\
  \item \textbf{Steps 55 to 63 (Algorithm \ref{L-shaped}):} Based on the previously generated REL $G^3_L$ of $G^2_L$, we will now construct a rectangular floor plan (RFP) $F'_L$ using the $T_1$ and $T_2$ direction, following the algorithm outlined in \cite{kant1997regular}. Once the RFP $F'_L$ is obtained, we will apply the function $Merge$ $Rooms$. This function takes the following inputs: the canonical ordered graph $G^2_L$, the generated RFP $F'_L$, the parameter $m$, and the set of extra nodes ($Enodes_L$: which are additional vertices introduced in $G_L$ to eliminate complex triangles, excluding $K_L$): Figure \ref{REL-L2}.\\
  \begin{enumerate}
      \item [(a.)] \textbf{Steps 56 to 58 (Algorithm \ref{L-shaped}):} Now, we will perform the merging of additional modules in $F'_L$ using the set $Enodes_L$. Specifically, for each vertex in $Enodes_L$, the corresponding module in $F'_L$ (i.e., the module associated with $u_i$) is merged with either module $a_i$ or $b_i$ within $F'_L$.\\
      \item [(b.)] \textbf{Steps 59 to 63 (Algorithm \ref{L-shaped}):}  For the modified  $K_L$ vertices, if $m$ $=$ $1$, the extra module $u$ will merge with the module $a$ in $F'_L$; otherwise, it will merge with the module $b$ in $F'_L$. Ultimately, this process results in an orthogonal floor plan $F_L$ containing a $L$-shaped module corresponding to the subgraph $K_L$.\\
  \end{enumerate}
  \end{enumerate}
\textbf{Thus, for a given plane triangulated graph $G_L(V,E)$ that contains at least one interior $K_L$, the Algorithm \ref{L-shaped} always produces a module of the $L$-shape (corresponding to the interior $K_L$) within an orthogonal floor plan $F_L$.}.
\end{proof}
\qed
  \begin{figure}[H]
   \centering  \includegraphics[width=0.95\textwidth]{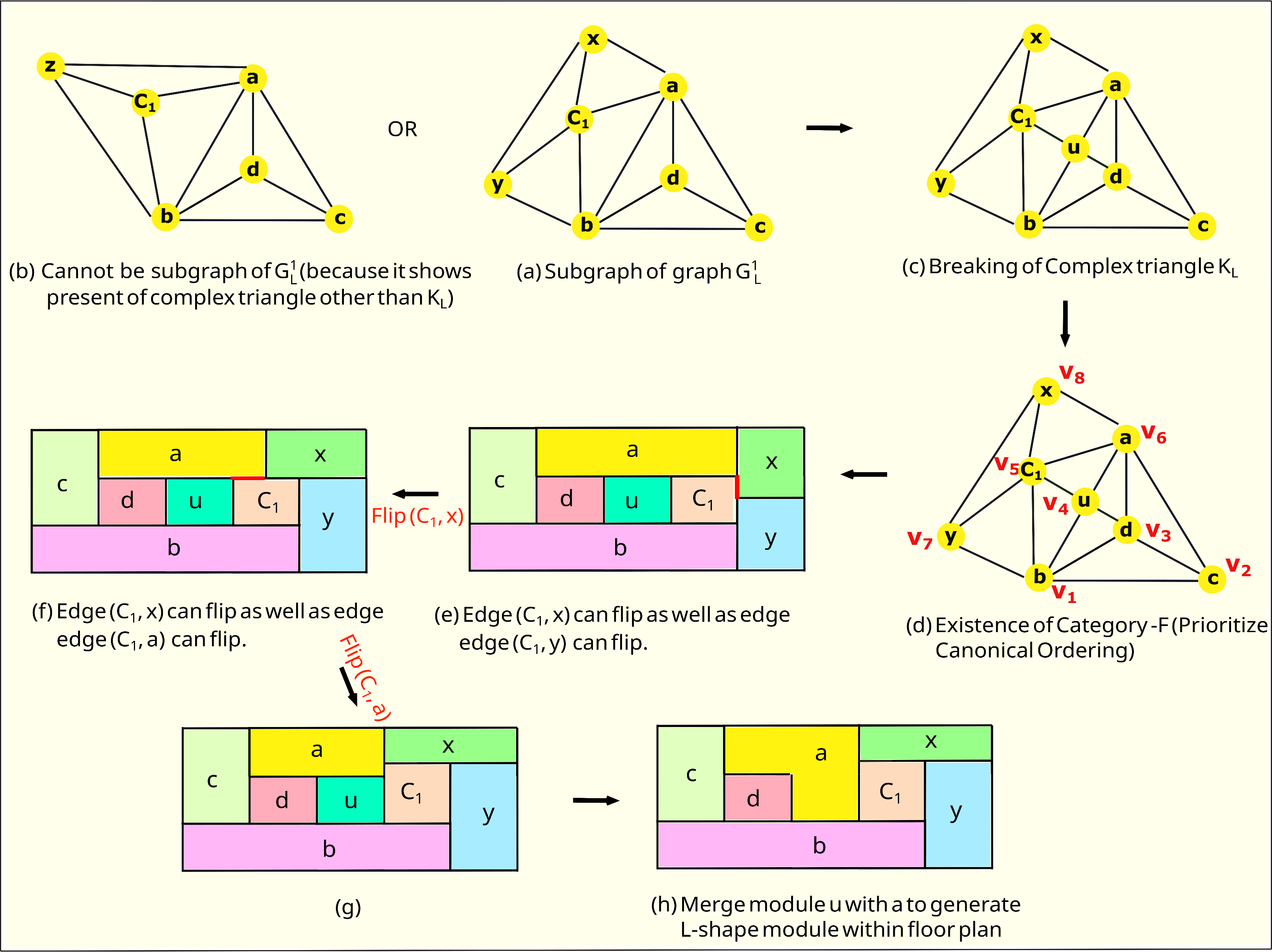}
     \caption{(a-h) Existence of a $L$-shaped module within floor plan by flipping an edge.}
   \label{CL-6}
  \end{figure} 
\begin{lemma}
\label{lemma-1}
 For a graph $G^1_L$, which is generated from an input graph $G_L$ with interior subgraph $K_L$, there always exists a canonical priority ordering of the set \{$a, C_1, d, u$\}, which corresponds to any of the six possible Categories: Category A-F.
 \end{lemma}
 \begin{proof}
For the generated graph $G^1_L$ containing the $modified$ subgraph $K_L$, we need to examine the canonical priority ordering for the set \{$a, C_1, d, u$\} from Categories A to F individually. The only remaining vertices available for ordering are from the set \{$a, C_1, d, u$\} (see Figure \ref{CL-1}(a-b)), while the vertices $v_n,......, v_{r+1}$ have already been canonically ordered in $G^1_L$. Therefore, the next possible vertex in $G^1_L$ that can be assigned the label $v_r$ must come from the set \{$a, C_1, d, u$\}.
\begin{enumerate}
    \item  See Figure \ref{CL-1}c: At $r$th step (Step 11 in Algorithm \ref{L-shaped}): vertex $u$ cannot initially be ordered as $v_r$ from the set \{$a, C_1, d, u$\} because $u$ has a degree of 4 and is adjacent to vertices $a$, $d$, $b$, and $C_1$. Among these, at most one vertex, i.e., $b$, could have already been ordered prior to ordering the vertices in the set \{$a, C_1, d, u$\} (here, "already ordered" refers to vertices that are canonical ordered earlier before ordering \{$a, C_1, d, u$\}). This means $visited(u)$ $<$ 2, which implies that it is not possible to order $u$ as $v_r$. \textbf{Therefore, a possible vertex for ordering as $v_r$ must come from the set \{$a, d, C_1$\}, excluding $u$.}\\
    \item  \textbf{Case-1:} Assume that the vertex $a$ can be canonical order as $v_r$ in $G^1_L$ (chosen from the set $\{a, d, C_1\}$) i.e., $visited(a)$ $\geq 2$ and $chord(a)$ $=$ $0$. \textbf{Canonical order $a$ as $v_r$ (see Figure \ref{CL-2}a)}.\\
    After ordering $a$ as $v_r$, the remaining graph $G_{L_{r-1}}^1$ ($G_{L_{r-1}}^1$ $=$ $G^1_L$ - \{$v_n$,...., $v_{r+1}$\}) has a component $C_{r-1}^L$ (\textbf{$C_{r-1}^L$} refers: the boundary of the exterior face of $G_{L_{r-1}}^1$ is cycle $C_{r-1}^L$ including the edge ($v_1$, $v_2$)) that contains a path that includes the subpath $C_1-u-d$ (see Figure \ref{CL-2}a). This implies that we can now either order $d$ or $C_1$ as $v_{r-1}$. Let us assume we \textbf{order $d$ as $v_{r-1}$: see Figure \ref{CL-2}b} (\textbf{if we label $C_1$ as $v_{r-1}$ then there exists Category-F canonical ordering: see Figure \ref{CL-3}(a-f)}). After that, the remaining graph $G^1_{L_{r-2}}$ has a component $C_{r-2}^L$ that contains a path that includes a subpath $C_1-u-b$ (see Figure \ref{CL-2}b). This implies that we can now canonical order $u$ as $v_{r-2}$ because $chord(u)$ $=$ $0$ and $visited(u)$ $\geq 2$ (since $d$ and $a$ are already canonical ordered). So \textbf{Canonical order $u$ as $v_{r-2}$: see Figure \ref{CL-2}c}. After that, ordering $C_1$ or $b$ is trivial (see Figure \ref{CL-2}(d,e)). \textbf{Hence, there exists a Category-D canonical ordering of set \{$a, C_1, d, u$\} with respect to the graph $G^1_L$}. \\
       \begin{figure}[H]
   \centering
 \includegraphics[width=0.95\textwidth]{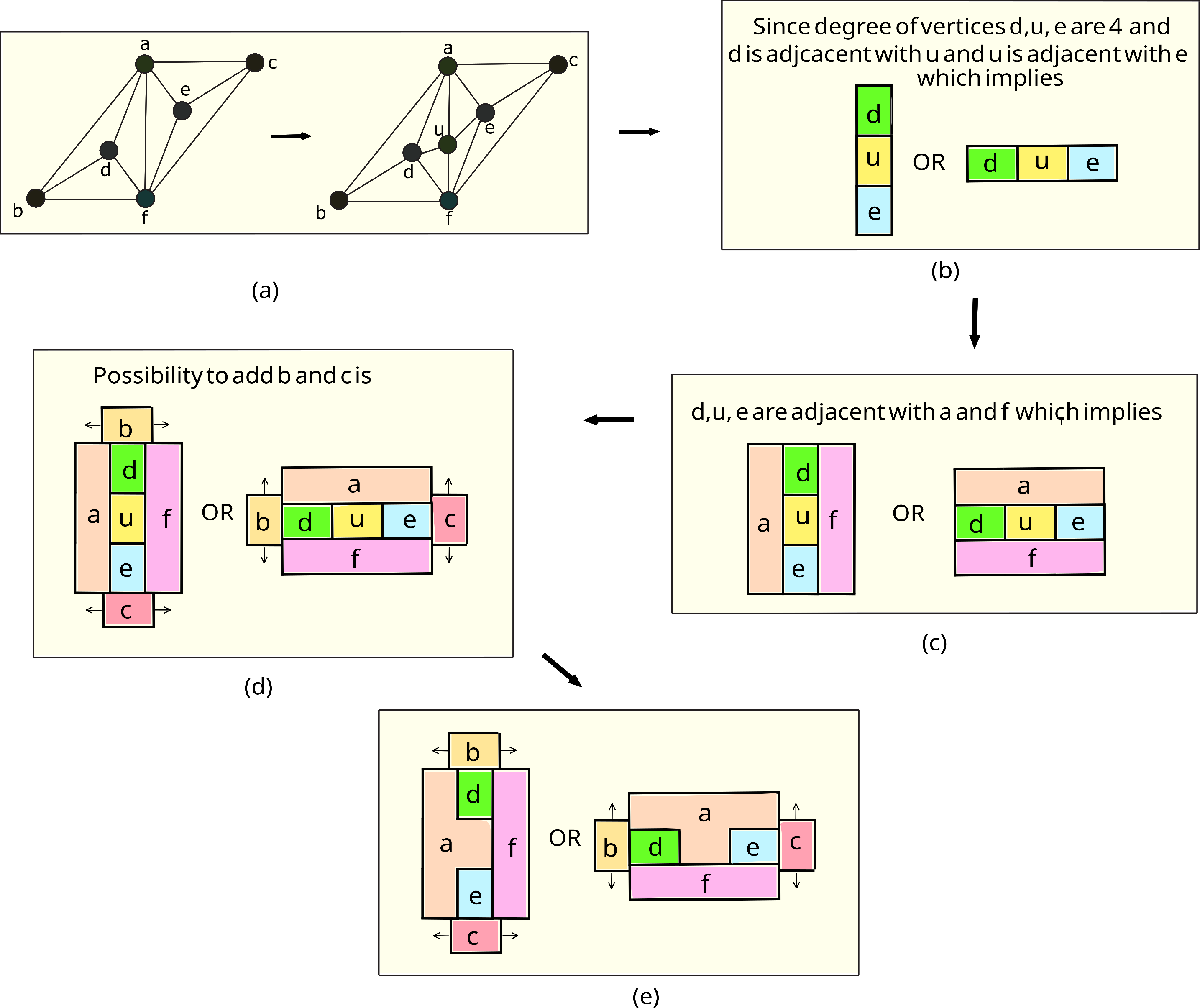}
    \caption{ (a-e) Existence of $T$-shaped module within floor plan corresponding to $K_T$.}
   \label{CT-1}
 \end{figure}
    \item \textbf{Case-2:} Assume that the vertex $a$ cannot be canonical order as $v_r$ in $G^1_L$. \textbf{Then the possible vertices that can be canonical order as $v_r$ in $G^1_L$ are either $C_1$ or $d$.}
    \begin{enumerate}
        \item [(a).] \textbf{Subcase 2(i):} Assume that the vertex $C_1$ can be canonical order as $v_r$ in $G^1_L$ (chosen from the set $\{C_1, d\}$) i.e., $visited(C_1)$ $\geq 2$ and $chord(C_1)$ $=$ $0$. \textbf{Canonical order $C_1$ as $v_r$: see Figure \ref{CL-4}a}.\\
        After ordering $C_1$ as $v_r$, the remaining graph $G^1_{L_{r-1}}$ has a component $C_{r-1}^L$ that contains a path that must includes the subpath $a-u-d-c$ (see Figure \ref{CL-4}a). This suggests that we can easily label $u$ as $v_{r-1}$ because $chord(u)$ $=$ $0$ and $visited(u)$ $\geq 2$ (since $C_1$ and $b$ are already canonical ordered). So \textbf{Label $u$ as $v_{r-1}$: see Figure \ref{CL-4}b}.  After that, the remaining graph $G^1_{L_{r-2}}$ has a component $C_{r-2}^L$ that contains a path that must include a subpath $c-d-a$ (see Figure \ref{CL-4}b). This implies that we can now label $d$ as $v_{r-2}$ because $chord(d)$ $=$ $0$ and $visited(d)$ $\geq 2$ (since $u$ and $b$ are already canonical ordered). So \textbf{Canonical order $d$ as $v_{r-2}$: see Figure \ref{CL-4}c}. After that, ordering $a$ or $c$ is trivial (see Figure \ref{CL-4}(d-e)). \textbf{Hence, there exists a Category-A canonical ordering of set \{$a, C_1, d, u$\} with respect to the graph $G^1_L$}. \\
         \item [(b).] \textbf{Subcase 2(ii):}  Assume that the vertex $C_1$ cannot be canonical order as $v_r$, which implies that only possible vertex for ordering as $v_r$ is $d$, i.e., $visited(d)$ $\geq 2$ and $chord(d)$ $=$ $0$. \textbf{Canonical order $d$ as $v_r$: see Figure \ref{CL-5}a}. After ordering $d$ as $v_r$, the remaining graph $G^1_{L_{r-1}}$ has a component $C_{r-1}^L$ that contains a path that must includes the subpath $a-u-C_1$ (see Figure \ref{CL-5}a). This suggests that we can easily label $u$ as $v_{r-1}$ because $chord(u)$ $=$ $0$ and $visited(u)$ $\geq 2$ (since $d$ and $b$ are already canonical ordered). So \textbf{Canonical order $u$ as $v_{r-1}$: see Figure \ref{CL-5}b}.  After that, the remaining graph $G^1_{L_{r-2}}$ has a component $C_{r-2}^L$ that contains a path that must include a subpath $a-C_1$ (see Figure \ref{CL-5}b). This implies that we can now label $C_1$ as $v_{r-2}$ and then $a$  as $v_{r-3}$ (see Figure \ref{CL-5}c) or $a$  as $v_{r-2}$ and then $C_1$ as $v_{r-3}$ (see Figure \ref{CL-5}d). \textbf{Hence, there exists either Category-B or Category-C canonical ordering of set \{$a, C_1, d, u$\} with respect to the graph $G^1_L$}.
         
    \end{enumerate}
\end{enumerate}
Hence, for a graph $G^1_L$, which is generated from an input graph $G_L$ with interior subgraph $K_L$, there always exists a canonical priority ordering of the set \{$a, C_1, d, u$\}, which corresponds to any of the six possible categories: Category A-F using \textbf{steps 7 to 53 of Algorithm \ref{L-shaped}}.
\end{proof}
\qed

\begin{theorem}
   For a given $G_T(V, E)$  that includes at least one internal subgraph isomorphic to $K_T$, Algorithm \ref{TLabel} yields an orthogonal floor plan $F_T$ that necessarily contains a $T$-shaped module corresponding to the subgraph $K_T$. 
\end{theorem}
\begin{proof}
Consider a triangulated graph $G_T$ that contains an internal complex triangle $K_T$. 
We aim to create a $T$-shaped module within the floor plan $F_T$ for the input graph $G_T$. For this purpose, a rectangular floor plan with an additional module, referred to as $u$, is needed. Creating a module of the $T$ shape requires merging module $u$ with either module $a$ or module $c$. To facilitate this process, we begin by considering a graph $G_T$ that includes an interior $K_T$ subgraph. We then modify $G_T$ by introducing two additional vertices, $u$ and $v$, thereby subdividing an interior edge of $K_T$. This modification yields a 4-connected triangulated graph, denoted as $G^1_T$, which is a critical prerequisite for establishing a canonical ordering. The transformation procedure further entails the elimination of any remaining complex triangles, the execution of 4-completion, and the insertion of an auxiliary edge $(N, S)$.
Next, we will assign a canonical ordering to the vertices of $G^1_L$ and then forward for REL generation and then using this to generate a rectangular floor plan (RFP) and lastly merge modules in RFP to generate $F_T$, which includes a $T$-shaped module.\\
We aim to show that a $T$-shaped module, representing the interior part $K_T$ of $G_T$, can always be constructed inside the floor plan $F_T$ of the graph $G_T$ by following our proposed Algorithm \ref{TLabel}. To prove this, we will go through each step of Algorithm \ref{TLabel} one by one, explaining how it works and confirming that it successfully builds the desired module.

\begin{enumerate}
    \item \textbf{Steps 1 to 4 (Algorithm \ref{TLabel}):}  Given an input graph $G_T$ that contains at least one $K_T$ as a subgraph, we first check the existence of any complex triangles other than $K_T$ (using the proposed algorithm in \cite{roy2001proof}). If such complex triangles exist in $G_T$ we apply the Complex Triangle Removal algorithm (proposed in  \cite{roy2001proof}), which introduces additional vertices and edges to eliminate them. This process ensures that all complex triangles (except $K_T$), are removed in $G_T$: see Figure \ref{T-6} (a-d). As a result, we obtain a modified graph $G_T$ that is free of complex triangles other than $K_T$.\\
    \item \textbf{Steps 5 to 6 (Algorithm \ref{TLabel}):} In order to modify the remaining $K_T$ subgraph, we first replace the internal edge $(a, c)$ of $K_T$  by inserting a new vertex $u$. To ensure that the updated graph remains triangulated, additional edges are introduced, as depicted in Figure 1. We denote this modified graph as $G^1_T$. Next, we apply the Four-Completion algorithm, as described in \cite{kant1997regular}, to $G^1_T$. At the final stage, we add an edge between vertex $N$ and $S$  to form $G^1_T$ (4-connected graph): see Figure \ref{T-6} (e-g). Thus, a 4-connected graph $G^1_T$ can always be constructed from an input graph $G^1_T$. \\
     \item \textbf{Steps 7 to 17 (Algorithm \ref{TLabel}):} Starting from the 4-connected PTG $G^1_T$ obtained earlier, we proceed with \textbf{Steps (7–9)} to initially assign $chord$ $value$ equal to $0$ and $status$ $value$ equal $False$ for each vertex in $G^1_T$. After that, mark vertex $W$ as $v_1$ and vertex $S$ as $v_2$, and set the visited value for vertex $E$ as $1$. After that, using \textbf{Steps (10–17)}, we assign canonical orders to all vertices of $G^1_T$ one by one using the method proposed in \cite{kant1997regular} and return the canonical ordered graph $G^2_T$: see Figures \ref{T-7}, \ref{T-8}. Hence, we can always construct a canonical ordered graph $G^2_T$ for a 4-connected PTG $G^1_T$.\\
     
    \item   \textbf{Step 18 (Algorithm \ref{TLabel}):} Using the canonical ordered graph $G^2_T$ generated earlier, we now proceed with Algorithm discussed in \cite{kant1997regular} to construct a regular edge labeling (REL, i.e., $T_1$ or $T_2$) of $G^2_T$, i.e., $G^3_T$: see Figure \ref{T-9}. Therefore, a REL $G^3_T$ can always be generated from the canonical ordered graph $G^2_T$.\\
  \item \textbf{Steps 19 to 24 (Algorithm \ref{TLabel}):}  Based on the previously generated REL $G^3_T$ of $G^2_T$, our next step is to design a rectangular floor plan (RFP) $F'_T$ using the $T_1$ and $T_2$ direction, following the algorithm outlined in \cite{kant1997regular}. Once the RFP $F'_T$ is obtained, we will apply the function $Merge$ $Rooms$. This function takes the following inputs: the canonical ordered graph $G^2_T$, the generated RFP $F'_T$, and the set of extra nodes ($Enodes_T$: which are additional vertices introduced in $G_T$ to eliminate complex triangles, excluding $K_T$). To construct a  $T$-shape module within the floor plan $F_T$, we have to merge the extra module $u$ with either $a$ or $b$ so that either module $a$ becomes a  $T$-shaped or $c$ becomes a $T$-shaped (See Figure \ref{CT-1}: Since degree of vertex $d$, $u$ and $e$ is 4 in graph $G^1_T$ and $d$ is adjacent to $u$, and $u$ is also adjacent to $e$, these modules can be arranged in the floor plan in two distinct configurations. Further adding modules $a$, $f$, $b$ and, $c$ in the floor plan implies module $a$ will always be $T$-shaped.  Hence, we can always construct a $T$-shaped module either by merging module $u$ with $a$ or by merging module $u$ with $c$). \\
  \begin{enumerate}
      \item [(a.)] \textbf{Steps 20 to 22 (Algorithm \ref{TLabel}):} Now, we perform the merging of additional modules in $F'_T$ using the set $Enodes_T$. Specifically, for each vertex in $Enodes_T$, the corresponding module in $F'_T$ (i.e., the module associated with $u_i$) is merged with either module $a_i$ or $b_i$ within $F'_T$: see Figure \ref{T-10} (a-b).\\
      \item [(b.)] \textbf{Steps 23 to 24 (Algorithm \ref{TLabel}):}  For the modified  $K_T$ vertices, the extra module $u$ will merge with the module $a$ in $F'_L$ or can merge with the module $c$ in $F'_L$. Ultimately, this process results in an orthogonal floor plan $F_T$ containing a $T$-shaped module corresponding to the subgraph $K_T$: see Figure \ref{T-10} (c).\\
  \end{enumerate}
  \end{enumerate}
\textbf{Thus, for PTG $G_T(V,E)$ that consists of a minimum one interior $K_T$, the Algorithm \ref{TLabel} always produces a module of the $T$-shape (corresponding to the interior $K_T$) within an orthogonal floor plan $F_T$.}
\end{proof}
\qed

\section{Time Complexity Analysis in Algorithms}
\label{Time}
\begin{theorem}
Given $G_L(V, E)$ that includes at least one internal subgraph isomorphic to $K_L$, Algorithm \ref{L-shaped} yields an orthogonal floor plan $F_L$, ensuring the inclusion of $L$‑shaped module (corresponds to the subgraph $K_L$) in linear time, i.e., $\textbf{O(n)}$.
\end{theorem}
\begin{proof}
$G_L$ that includes at least one internal subgraph isomorphic to $K_L$, Algorithm \ref{L-shaped} yields an orthogonal floor plan $F_L$ (corresponding to $G_L$) that necessarily contains a $L$‑shaped module corresponding to the subgraph $K_L$. We will demonstrate that this construction can be completed in linear time, $O(n)$.
We will provide a comprehensive examination of the computational complexity of Algorithm \ref{L-shaped}.
\begin{enumerate}
    \item [1.] \textbf{Steps (1-4): Algorithm \ref{L-shaped}} [$\bm{O(n)}$]: Given an input graph $G_L$ that contains at least one interior $K_L$ subgraph, we can identify any such interior $K_L$ within $G_L$ in $\bm{O(n)}$ time. This can be achieved by detecting a triangle formed by three vertices that include an interior vertex of degree 3, using the method described in \cite{johnson1975finding}. Once such a complex triangle ($K_4$) is found, we designate it as $K_L$, this process takes linear time, i.e., $O(n)$. The vertices of $K_L$ are ordered in counter-clockwise order as follows: first $a$, then $b$, followed by $c$, and finally $d$, which also requires only $\bm{O(1)}$ time.\\ If graph $G_L$ contains complex triangles apart from the selected $K_L$, we apply the method proposed in \cite{roy2001proof}, which operates in $\bm{O(n)}$ time, to eliminate these complex triangles by introducing new vertices and edges. This ensures that $G_L$ is free of all complex triangles except $K_L$. Therefore, the total computational complexity:  Steps 1 through 4 (Algorithm \ref{L-shaped}) is $
\bm{[O(n) + O(1) + O(n):  O(n)]}$, indicating that this preprocessing can be executed efficiently in linear time.\\
\item [2.] \textbf{Steps (5-6): Algorithm \ref{L-shaped}} [$\bm{O(n)}$]: To process the remaining $K_L$ subgraph within $G_L$, we begin by removing the exterior edge $(a,b)$ of the subgraph $K_L$, insertion of one vertex $u$, and insertion four new edges to ensure the graph is triangulated. This modification is achieved in constant time, i.e.,  $\bm{O(1)}$, since addition and deletion are constant-time operations. The resulting graph after this step is referred to as $G^1_L$. Subsequently, we apply the Four-completion algorithm (linear time algorithm proposed by Kant and He  \cite{kant1997regular}), which introduces four new vertices ordered as $E$, $N$, $S$, and $W$ in $G^1_L$, along with their associated new edges (for triangulation). This step is performed in linear time, $\bm{O(n)}$.
After completing this procedure, a new edge $(N, S)$ is added to $G^1_L$. This step takes constant time, as inserting an edge is an $\bm{O(1)}$ operation. As a result, the overall time complexity for Steps 5 to 6 in Algorithm \ref{L-shaped} is $\bm{[2O(1) + O(n):  O(n)]}$. This demonstrates that these steps can be carried out efficiently in linear time relative to the input size $n$.\\
\item [3.] \textbf{Steps (7-53): Algorithm \ref{L-shaped}} [$\bm{O(n)}$]:
Starting from the previously constructed 4-connected triangulated graph denoted as $G^1_L$, we proceed with Steps 7 through 53 of Algorithm \ref{L-shaped} to derive a canonically ordered graph $G^2_L$ of $G^1_L$. 
\begin{enumerate}
\item [(a.)] \textbf{Steps (7–28): Algorithm \ref{L-shaped}} [$\bm{O(n)}$]: 
Initially, every vertex in $G^1_L$ is canonically ordered using the function $CanonLabel$ (with the exception of the vertices in the set $\{a, b, c, C_1, d, u\}$), as described in \cite{kant1997regular}, which is linear time i.e., $\bm{O(n)}$.
This canonical ordering process proceeds until the only remaining vertices eligible for ordering are those in the set $\{a,b,c, C_1, d, u\}$. At this point, the labels are assigned to these vertices following a predetermined order of priority, which is divided into six separate categories (Category A-F).
\\
\item  [(b.)] \textbf{Steps (29–53): Algorithm \ref{L-shaped}} [$\bm{O(n)}$]: 
Next, we will try/check to label the vertices of the set \{$a, b, c, C_1, d, u$\} systematically with respect to the defined six priority ordering one by one. For every category, the ${CanonLabel}$ function is called, which requires time complexity of $\bm{O(n)}$, and a priority validation condition, denoted by ${PLabel}(L)$, which is performed in constant time, $\bm{O(1)}$ (Since checking condition require constant time only). If any ordering is found (out of Category A-F), the resulting ordered graph $G^2_L$ is then forwarded for the REL (Regular Edge Labeling) construction. Thus, the validation of the ordering for any specific category requires $\bm{O(n)}$ time. Since there are six categories to check/evaluate, this phase requires at most $\bm{6O(n)}$ time in total.\\
\end{enumerate}

Consequently, the total computational complexity for steps 7 through 53 is $\bm{[O(n) + 6O(n) + O(1) = O(n)]}$, which demonstrates that the prioritize canonical ordering process can be executed efficiently in linear time.\\
\item [4.] \textbf{Step 54: Algorithm \ref{L-shaped}} [$\bm{O(n)}$]: Starting with the canonical ordered graph $G^2_L$, the Algorithm \ref{RELF} is used to construct the regular edge labeling (REL i.e., $T_1$ and $T_2$) $G^3_L$ of $G^2_L$.\\
\begin{enumerate}
    \item [(a.)] \textbf{Steps (1–10): Algorithm \ref{RELF}} [$\bm{O(n)}$]:  Initially, the REL (i.e., $T_1$ and $T_2$) for the input graph $G^2_L$ is constructed following the approach detailed in \cite{bhasker1988linear}. This construction process operates in linear time, that is, $\bm{O(n)}$, as explained in \cite{bhasker1988linear}. Hence, steps 1–10 of Algorithm \ref{RELF} require linear time. \\
 \item [(b.)] \textbf{Steps (11–23): Algorithm \ref{RELF}} [$\bm{O(1)}$]: The next step involves selecting and executing the appropriate operation based on the REL generated above ($T_1$ and $T_2$). Specifically, the algorithm merges the vertex $u$ with the vertex $a$ or the vertex $b$ (by returning value $r$ $=$ $1$ or $2$), or it performs an edge flip followed by merging $u$ with $a$  (by returning value $r$ $=$ $1$). Each of these operations can be completed in constant time (Since checking the condition requires constant time and changing the direction of edges, i.e., $T_1$ or $T_2$, takes constant time), i.e.,  $\bm{O(1)}$.\\
\end{enumerate}
Therefore, the computational complexity for Step 45 of Algorithm \ref{L-shaped} is $\bm{[O(n) + O(1) = O(n)]}$, which is linear.\\
\item [5.] \textbf{Steps (55-63): Algorithm \ref{L-shaped}} [$\bm{O(n)}$]:\\
\begin{enumerate}
    \item [(a.)] \textbf{Step 55: Algorithm \ref{L-shaped}} [$\bm{O(n)}$]: Utilizing the regular edge labeling, i.e., $G^3_L$ (represented by $T_1$ and $T_2$), a rectangular floor plan $F'_L$ is constructed by employing the approach described in Bhasker and Sahni \cite{bhasker1988linear}, which operates with linear time complexity, $\bm{O(n)}$.\\
    \item [(b.)] \textbf{Steps (56-58): Algorithm \ref{L-shaped}} [$\bm{O(r)}$]:
    After that, we will merge additional modules within $F'_L$ by calling the function $Merge$ $Room$. For each vertex in the set $Enodes_L$ where the cardinality $r$ of $Enodes_L$ is less than the total number of vertices $n$ in $G^1_L$, the corresponding modules in $F'_L$ are merged as required. Since each merging operation can be performed in constant time, the cumulative time complexity for this step is $\bm{O(r)}$, which is less than $O(n)$.\\
   \item [(c.)] \textbf{Steps (59-63): Algorithm \ref{L-shaped}} [$\bm{O(1)}$]: Finally, for the vertices associated with $K_L$, the algorithm merges module $u$ with either $a$ or $b$, and this operation is performed in constant time, $\bm{O(1)}$.\\
\end{enumerate}
In summary, the overall computational complexity for steps 55 through 63 (Algorithm \ref{L-shaped}) is $\bm{O(r) + O(1) + O(n): O(n)}$, which simplifies to $\bm{O(n)}$, since $r < n$. This shows that these steps are executable efficiently in linear time with respect to the size of the input graph.\\
\end{enumerate}
Consequently, the time complexity of Algorithm \ref{L-shaped} is established as linear since each of its five core stages runs in $O(n)$ time. When combined, these yield a total of $\bm{(O(n) + O(n) + ..... + O(n))}$ (five times), which simplifies to an overall time complexity of $O(n)$. \textbf{Therefore, given $G_L(V, E)$  that includes at least one internal subgraph isomorphic to $K_L$, Algorithm \ref{L-shaped} yields an orthogonal floor plan $F_L$ that necessarily contains a $L$‑shaped module corresponding to the subgraph $K_L$ in linear time, i.e., $O(n)$, where $n$ is the number of vertices in $G_L$}. 

\end{proof}
\qed

\begin{theorem}
Given $G_T(V, E)$ that includes at least one internal subgraph isomorphic to $K_T$, Algorithm \ref{TLabel} yields an orthogonal floor plan $F_T$, ensuring the inclusion of $T$‑shaped module (corresponds to the subgraph $K_T$) in linear time, i.e., $\textbf{O(n)}$.
\end{theorem}
\begin{proof}
$G_T$ that includes at least one internal subgraph isomorphic to $K_L$, Algorithm \ref{TLabel} yields an orthogonal floor plan $F_T$ (corresponding to $G_T$) that necessarily contains a $T$‑shaped module corresponding to the subgraph $K_T$. We will demonstrate that this construction can be completed in linear time, $O(n)$.
We will provide a comprehensive examination of the computational complexity of Algorithm \ref{TLabel}.
\begin{enumerate}
    \item [1.] \textbf{Steps (1-4): Algorithm \ref{TLabel}} [$\bm{O(n)}$]: Given an input graph $G_T$ that contains at least one interior $K_T$ subgraph, we can identify any such interior $K_T$ within $G_T$ in $\bm{O(n)}$ time. This can be achieved by detecting two triangles sharing a common edge where each triangle is formed by three vertices that include an interior vertex of degree 3, using the method described in \cite{johnson1975finding}. Once such a complex triangle is found, we designate it as $K_T$, this process takes linear time, i.e., $O(n)$. The vertices of $K_T$ are ordered in counter-clockwise order as follows: first $a$, then $b$, followed by $c$, $d$, $e$ and finally $f$, which also requires only $\bm{O(1)}$ time.\\ If graph $G_T$ contains complex triangles apart from the selected $K_T$, we apply the method proposed in \cite{roy2001proof}, which operates in $\bm{O(n)}$ time, to eliminate these complex triangles by introducing new vertices and edges. This ensures that $G_T$ is free of all complex triangles except $K_T$. Therefore, the total computational complexity:  Steps 1 through 4 (Algorithm \ref{TLabel}) is $
\bm{[2 O(n) + O(1):  O(n)]}$, indicating that this preprocessing can be executed efficiently in linear time.\\
\item [2.] \textbf{Steps (5-6): Algorithm \ref{TLabel}} [$\bm{O(n)}$]: To process the remaining $K_T$ subgraph within $G_T$, we begin by removing the interior edge $(a,c)$ of the subgraph $K_T$, insertion of one vertex $u$, and insertion four new edges to ensure the graph is triangulated. This modification is achieved in constant time, i.e.,  $\bm{O(1)}$, since addition and deletion are constant-time operations. The resulting graph after this step is referred to as $G^1_T$. Subsequently, we apply the Four-completion algorithm (linear time algorithm proposed by G. Kant and X. He  \cite{kant1997regular}), which introduces four new vertices labeled as $E$, $N$, $S$, and $W$ in $G^1_T$, along with their associated new edges (for triangulation). This step is performed in linear time, $\bm{O(n)}$.
After completing this procedure, a new edge $(N, S)$ is added to $G^1_T$. This step takes constant time, as inserting an edge is an $\bm{O(1)}$ operation. As a result, the overall time complexity for Steps 5 to 9 (Algorithm \ref{TLabel}) is $\bm{[2 O(1) + O(n): O(n)]}$. This demonstrates that these steps can be carried out efficiently in linear time relative to the input size $n$.\\
\item [3.] \textbf{Steps (7-17): Algorithm \ref{TLabel}} [$\bm{O(n)}$]:
Starting from the 4-connected PTG $G^1_T$ obtained earlier, we proceed with \textbf{Steps (7–9)} to initially assign $chord$ $value$ equal to $0$ and $status$ $value$ equal to $False$ for each vertex in $G^1_T$ which requires constant time i.e, $\bm{O(1)}$.  After that, mark vertex $W$ as $v_1$ and vertex $S$ as $v_2$ and set the visited value for vertex $E$ as $1$, which also requires $\bm{O(1)}$. After that, using \textbf{Steps (10–17)}, we assign canonical order to all vertices of $G^1_T$ one by one using the method proposed in \cite{kant1997regular} which require linear time i,e., $\bm{O(n)}$, while ensuring the process is explained in Definition 4 (refer to the Terminology Section \ref{Preliminaries}) and return the canonical ordered graph $G^2_T$. 
\\
Consequently, the total computational complexity for steps 7 through 17 is $\bm{[2O(1) + O(n):  O(n)]}$, which demonstrates that the canonical ordering process can be executed efficiently in linear time.\\
\item [4.] \textbf{Step 18: Algorithm \ref{TLabel}} 
[$\bm{O(n)}$]: 
Using the generated canonical ordered graph $G^2_T$ generated earlier, we now proceed with the Algorithm discussed in \cite{kant1997regular} to construct a Regular Edge labeling (REL, i.e., $T_1$ or $T_2$) of $G^2_T$, i.e., $G^3_T$.  
This construction process operates in linear time, that is, $\bm{O(n)}$, as explained in \cite{kant1997regular}.\\
Consequently, the total computational complexity for step 18 is $\bm{[O(n)]}$, which is linear.\\
\item [5.] \textbf{Steps (19-24): Algorithm \ref{TLabel}} [$\bm{O(n)}$]:\\
\begin{enumerate}
    \item [(a.)] \textbf{Step 19: Algorithm \ref{TLabel}} [$\bm{O(n)}$]: Utilizing the regular edge labeling, i.e., $G^3_T$ (represented by $T_1$ and $T_2$), a rectangular floor plan $F'_T$ is constructed by employing the approach described in Bhasker and Sahni \cite{bhasker1988linear}, which operates with linear time complexity, $\bm{O(n)}$.\\
    \item [(b.)] \textbf{Steps (20-22): Algorithm \ref{TLabel}} [$\bm{O(r)}$]:
    After that, we will merge additional modules within $F'_T$ by calling the function $Merge$ $Room$. For each vertex in the set $Enodes_T$ where the cardinality $r$ of $Enodes_T$ is less than the total number of vertices $n$ in $G^1_T$, the corresponding modules in $F'_T$ are merged as required. Since each merging operation can be performed in constant time, the cumulative time complexity for this step is $\bm{O(r)}$, which is less than $\bm{O(n)}$.\\
   \item [(c.)] \textbf{Steps (23-24): Algorithm \ref{TLabel}} [$\bm{O(1)}$]: Finally, for the vertices associated with $K_T$, the algorithm merges module $u$ with either $a$ or $c$, and this operation is performed in constant time, $\bm{O(1)}$.\\
\end{enumerate}
In summary, the overall computational complexity for Steps 19 through 24 (Algorithm \ref{TLabel}) is $\bm{O(n) + O(1) + O(r): O(n)}$, which simplifies to [$\bm{O(n)}$], since $r < n$. This shows that these steps are executable efficiently in linear time with respect to the size of the input graph.
\end{enumerate}
Consequently, the time complexity of Algorithm \ref{TLabel} is established as linear since each of its five core stages runs in $O(n)$ time. When combined, these yield a total of $\bm{(O(n) + O(n) + ..... + O(n))}$ (five times), which simplifies to an overall time complexity of $O(n)$. \textbf{Therefore, for a given $G_T(V, E)$  that includes at least one internal subgraph isomorphic to $K_T$, Algorithm \ref{TLabel} yields an orthogonal floor plan $F_T$ that necessarily contains a $T$‑shaped module corresponding to the subgraph $K_T$ in linear time, i.e., $\textbf{O(n)}$, where $n$ is the number of vertices in $G_T$}. 

\end{proof}
\qed

\section{Conclusion and future work}\label{conclusion}
This paper presents a structured approach outlined in Algorithm \ref{L-shaped} and Algorithm \ref{TLabel} for embedding $L$-shaped and $T$-shaped modules within the final floor plans $F_L$ and $F_T$, respectively. Each construction begins with a specific input graph: a PTG $G_L$ containing at least one internal complex triangle ($K_L$) for the $L$-shaped case and a PTG $G_T$ with at least one internal subgraph ($K_T$) for the $T$-shaped case. The process unfolds in four key stages: first, the input PTG is transformed into a 4-connected triangulated graph by introducing auxiliary vertices and edges. Next, a canonical ordering is applied to this modified graph. Using the resulting order, a rectangular floor plan is constructed. Finally, the auxiliary modules, which represent the added vertices, are merged to form an orthogonal floor plan that incorporates either a $L$-shaped or $T$-shaped module corresponding to the original graph $G_L$ or $G_T$.\\
An important characteristic of our proposed work for the generation of $L$-shaped and $T$-shaped modules is their inherently non-trivial structure, meaning that the desired module (either $L$ or $T$-shaped) in the final floor plan cannot be reduced to simpler forms through shrinking or stretching walls of its modules. Ensuring this structural uniqueness relies significantly on the application of canonical ordering methods.\\
In our forthcoming research, we plan to broaden the scope of our present framework to incorporate additional non-rectangular module types, specifically plus-shaped, stair-shaped, and Z-shaped forms, into the floor planning process. This will involve determining the structural criteria that permit the inclusion of these complex shapes in both floor plans and their corresponding graphs, followed by the design of algorithms to facilitate their generation.
\section {Declarations}
\textbf{Competing Interests:}
The authors declare that there are no competing interests, financial or non-financial, directly or indirectly related to the work submitted for publication. No funds, grants, or other support were received during the preparation of this manuscript.\\
\textbf{Data availability:} Enquiries about data availability should be directed to the authors.

\bibliographystyle{unsrt}
\bibliography{mybib} 
\end{document}